\documentclass[a4paper,reqno]{amsart}

\usepackage[utf8]{inputenc}
\usepackage[english]{babel}
\usepackage[T1]{fontenc}
\usepackage{microtype}
\usepackage{amsmath,amssymb,amsthm}
\usepackage{thmtools,thm-restate}
\usepackage{mathtools}
\usepackage{upgreek}
\usepackage{orcidlink}
\usepackage{hyperref}
\hypersetup{hidelinks}
\usepackage[nameinlink]{cleveref}
\usepackage{suffix}
\usepackage{nccmath}
\usepackage{nicefrac}
\usepackage{enumitem}

\usepackage[hang,flushmargin]{footmisc}
\crefname{equation}{equation}{equation}

\title[{Evolutionary equations with state-dependent delay}]{Evolutionary equations with state-dependent delay}
\author[B.~Aigner]{Bernhard Aigner\orcidlink{0009-0009-8252-162X}}
\thanks{supported by the state of Saxony via a graduate student stipend}
\address[B.~Aigner]{Institut f\"{u}r Angewandte Analysis\\
  TU Bergakademie Freiberg\\
  Institut für Angewandte Analysis\\
  Pr\"{u}ferstr. 9, 09599 Freiberg\\
  Germany}
\email[B.~Aigner]{bernhard.aigner@doktorand.tu-freiberg.de}

\author[M.~Waurick]{Marcus Waurick\orcidlink{0000-0003-4498-3574}}
\address[M.~Waurick]{Institut f\"{u}r Angewandte Analysis\\
  TU Bergakademie Freiberg\\
  Institut für Angewandte Analysis\\
  Pr\"{u}ferstr. 9, 09599 Freiberg\\
  Germany}
\email[M.~Waurick]{marcus.waurick@math.tu-freiberg.de}
\date{\today}

\DeclarePairedDelimiterX{\norm}[1]{\lVert}{\rVert}{#1}
\DeclarePairedDelimiterX{\abs}[1]{\lvert}{\rvert}{#1}
\DeclarePairedDelimiterX{\dset}[2]{\{}{\}}{#1\,\delimsize\vert\,\mathopen{} #2}
\DeclarePairedDelimiterX{\scprod}[2]{(}{)}{#1\delimsize| #2}
\DeclarePairedDelimiterX{\dualprod}[2]{\langle}{\rangle}{#1,#2}
\newcommand{\e}{\mathrm{e}}
\newcommand{\iu}{\mathrm{i}}
\newcommand{\dd}{\mathrm{d}}
\newcommand{\dx}[1][x]{\,\dd{}#1}
\newcommand{\argdot}{\cdot}
\renewcommand{\Re}{\operatorname{Re}}

\newcommand{\adjun}{^{\ast}}

\newcommand{\R}{\mathbb{R}}
\newcommand{\C}{\mathbb{C}}
\newcommand{\N}{\mathbb{N}}
\newcommand{\supp}{\mathrm{spt}}

\newtheorem*{definition*}{Definition}
\newtheorem*{theorem*}{Theorem}
\newtheorem*{proposition*}{Proposition}
\newtheorem*{lemma*}{Lemma}
\newtheorem*{corollary*}{Corollary}
\newtheorem*{remark*}{Remark}

\newtheorem{definition}{Definition}[section]
\newtheorem{theorem}[definition]{Theorem}
\newtheorem{proposition}[definition]{Proposition}
\newtheorem{lemma}[definition]{Lemma}
\newtheorem{corollary}[definition]{Corollary}
\newtheorem{remark}[definition]{Remark}
\newtheorem{example}[definition]{Example}
\newtheorem{assumptions}[definition]{Assumptions}

\begin{document}

\begin{abstract}
  We extend a contraction mapping argument for ordinary state-dependent delay differential equations
  to evolutionary partial differential equations in the sense of R.~Picard, that is, to equations of the form $\bigl(\partial_{t} M(\partial_{t}) + A\bigr) u(t) = F\bigl(t,u_{(t)}\bigr)$, where $A$ is an $\mathrm{m}$-accretive (unbounded) linear operator and $M$ is a material law. We establish local well-posedness (in the sense of weak solutions) of generalized initial value problems that stem from a distributional formulation. We require prehistories in $H^{1}$ with bounded derivative, a regularity increasing right-hand side and a consistency condition. We showcase the viability of our results by applying them to classical examples (heat, wave and Maxwell's equations), examples from semigroup theory, port-Hamiltonian systems, as well as equations featuring fractional derivatives and convolutions (in time) with bounded operators.
\end{abstract}

\maketitle

\section{Introduction}
\label{sec:Intro}
In this article we study partial differential equations (PDEs) with state-dependent delay, which are PDEs of the form
\begin{equation*}
  (Du)(t)=F\bigl(t,u_{(t)}\bigr)\qquad t>0\text{,}
\end{equation*}
where $D$ is a (partial) differential operator in space-time and the right-hand side $F$ is a function depending on the history
\begin{equation*}
  u_{(t)}\colon\,[-h,0]\rightarrow H\text{,}\qquad s\mapsto u(t+s)
\end{equation*}
 of the state $u$. The parameter $h>0$ describes a suitable backward time-horizon and $H$ at this point denotes any Hilbert space. The dependence on prior histories demands --- in contrast to equations without delay --- a prescribed initial prehistory $\Phi$ instead of traditional initial values:
\begin{equation*}
  u_{(0)} = \Phi \qquad \Longleftrightarrow \qquad u(t)=\Phi(t)\quad\text{for }-h\leq t\leq 0\text{.}
\end{equation*}
This article follows up on recent discoveries concerning the solution theory of initial value problems (IVPs) for ordinary differential equations (ODEs) (i.e., $D=\nicefrac{\operatorname{d}}{\operatorname{dt}}$) with state-dependent delay utilizing a weak solution theory and exponentially weighted Sobolev spaces. We refer to \cite{Waurick2023} for the principal publication and to \cite{Aigner2024} for a summary of the main results and some elementary examples. We will extend the approach presented there to evolutionary equations in the sense of R.~Picard, i.e., PDEs of the form
\begin{equation}
  \label{eq:FDE}
  \bigl(\partial_{t}M(\partial_{t})+A\bigr)u(t)= f(t) + F\bigl(t, (u+Z)_{(t)}\bigr) \qquad\, t>0\text{,}\\
\end{equation}
where $M$ is a holomorphic operator-valued function that is bounded on a right-half space, $M(\partial_{t})$ is a so-called material law operator (cf.~\cref{subsec:MatLaw} or \cite[ch.~5.2]{Waurick2022}) and $A$ is an $\mathrm{m}$-accretive operator on the Hilbert space $H$. In particular, our results hold for skew-selfadjoint $A$. The function $f$ is assumed to be suitably regular and integrable, $Z$ is an extension of some initial prehistory $\Phi \in H^{1}(-h,0;H)$ and $F\colon [0,\infty)\times L_{2}(-h,0;H)\to H$ is a continuous map incorporating the state-dependent delay $u_{(t)}$.
In the theory of evolutionary equations, IVPs of the form
\begin{alignat}{3}
  \label{eq:IVP}
  \begin{aligned}
    \bigl(\partial_{t}M(\partial_{t}) + A\bigr)u(t)&= F\bigl(t,u_{(t)}\bigr) && \qquad\qquad\, t>0\text{,}\\
    u(t)&=\Phi(t) && \quad -h\leq t\leq 0\text{,}
  \end{aligned}
\end{alignat}
are equivalent to a reformulation as evolutionary equations on the entire line without initial value, albeit in a distributional setting, cf.~\cite[thm.~9.4.3]{Waurick2022}. The term $f$ in \cref{eq:FDE} originates from this reformulation and suitably incorporates the initial prehistory $\Phi$. For more complex material laws, such a formulation is obtained using a projection argument. We provide some background in \cref{subsec:DistribSpaces_IVPs}. The appearance of a distributional formulation is also present in other approaches. Indeed, in the context of solution theory utilizing semigroup theory, well-posedness is understood in the sense of existence of mild solutions. It is a well-known fact, that mild solutions of a Cauchy problem are the classical solutions of the extrapolated Cauchy problem (cf.~\cite[sec.~VI.7]{Engel/Nagel2000} or \cite[sec.~2.2]{Schnaubelt2024}).\\
In the context of delay differential equations (DDEs), semigroup theory has been utilized extensively, e.g., cf.~\cite{Batkai2005} for a comprehensive overview and also \cite{Wu1996}. There are results for the non-autonomous setting (cf.~\cite{Hadd2006}) or with asymptotically autonomous delay (cf.~\cite{Schnaubelt2004}). Here, we offer a different perspective, but we compare our results to semigroup approaches in \cref{subsec:SG1stOrder} and \cref{subsec:SG2ndOrder}.\\
In contrast to semigroup theory, the theory of evolutionary equations can make short work of mixed type problems (e.g., mixed eddy current models, cf.~\cref{subsec:Maxwell}) and the material law can incorporate operators that are even nonlocal in time (cf.~\cref{subsec:ViscoElasticity}). The large applicability of evolutionary equations comes at the price of a more involved problem formulation for initial value problems. Indeed, partial differential algebraic equations are included; a subclass for which it is already a challenging task to identify a suitable space of initial values. We choose to forego this problem and focus on the well-posedness aspect. This approach is reminiscent of the foundational article of evolutionary equations \cite{Picard2009} and the works \cite{Trostorff2018,Trostorff2020,Trostorff2017} of S.~Trostorff.\\
We want to transform the problem (\ref{eq:IVP}) into a fixed point problem (FPP) and then utilize a contraction mapping argument. To that end, we employ weighted Lebesgue and Sobolev spaces:
\begin{definition}
  \label{def:LebesgueSobolev}
  For $\rho>0$ and $-\infty\leq a<b\leq +\infty$ let
  \begin{align*}
    L_{2,\rho}(a,b;H)&\coloneq \overline{\mathcal{C}_{c}^{\infty}\bigl(a,b;H\bigr)}^{\norm{\argdot}_{2,\rho}}\text{,}\\
    \norm{u}_{2,\rho}&\coloneq \Bigl(\medint\int_a^b \norm{u(t)}^{2}_{H}\e^{-2\rho t} \dx[t]\Bigr)^{\nicefrac{1}{2}}\text{,}
  \intertext{and, accordingly,}
    H_{\rho}^{1}(a,b;H)&\coloneq \bigl\{u\in L_{2,\rho}(a,b;H)\colon\,u'\in L_{2,\rho}(a,b;H)\bigr\}\text{,}\\
    \norm{u}_{H^{1}_{\rho}}&\coloneq \bigl(\norm{u}^{2}_{2,\rho}+\norm{u'}^{2}_{2,\rho}\bigr)^{\nicefrac{1}{2}}\text{,}
  \end{align*}
\end{definition}
where $\mathcal{C}_{c}^{\infty}\bigl(a,b;H\bigr)$ is the space of all smooth functions with values in $H$ and compact support in $(a,b)$. The derivative in the definition is to be understood in the distributional sense. The use of weighted Sobolev spaces, rather than the classical (unweighted) versions will allow us to force contractivity of a Lipschitz-mapping, reminiscent of a classical proof of the Picard--Lindel\"of theorem for ODEs by Morgenstern, cf.~\cite{Morgenstern1952}. For a general introduction to Sobolev spaces we refer to \cite[ch.~5]{Evans2010}, for exponentially weighted Sobolev spaces to \cite[ch.~3]{Waurick2022}.\\
The article is structured as follows: We begin with a short recap on evolutionary equations and mostly recall established results with a few new considerations. We follow up with some observations on the delay and on the nature of evolutionary equations with state-dependent delay. In \cref{sec:SolThy} we prove our main (well-posedness) results, which we make more explicit for the case of simple material laws in \cref{sec:SimMatLaws}. We conclude with a large variety of examples in \cref{sec:Examples} and a short review of our results paired with a comparison to other results, in particular to results utilizing semigroup theory. The appendix contains definitions of differential operators employed in \cref{sec:Examples} and recapitulates known results for the formulation of IVPs in the framework of evolutionary equations.

\section{Background on evolutionary equations}
\label{sec:EvEq}
In the context of evolutionary equations, the left-hand side of \cref{eq:FDE} is split up into two parts, on one hand the ``material dependence'' $\partial_{t}M(\partial_{t})$, which is time-dependent, and on the other hand a purely spatial part $A$, which is assumed to be a densely defined, linear, $\mathrm{m}$-accretive operator on the Hilbert space $H$. To explain the first part, we first need to explain a central concept in evolutionary equations, the time derivative in exponentially weighted spaces.

\subsection{The time derivative}
\label{subsec:TimeDerivative}
Let $\rho > 0$ and $-\infty < a < b \leq +\infty$. Let
\begin{align*}
  I_{\rho}\colon\,L_{2,\rho}(a,b;H)&\to L_{2,\rho}(a,b;H)\text{,}\\
  u&\mapsto \Bigl( t\mapsto \medint\int_a^{t} u(s)\dx[s]\Bigr)\text{.}
\end{align*}
This operator is well-defined on $L_{2,\rho}(a,b;H)$, as the next proposition shows.
\begin{proposition}
  Let $\rho > 0$ and $-\infty < a < b \leq +\infty$. Then $I_{\rho}\in \mathcal{L}\bigl(L_{2,\rho}(a,b;H)\bigr)$ with $\norm{I_{\rho}}\leq \nicefrac{1}{\rho}$.
\end{proposition}
The proof is a simple calculation and can be checked in \cite[prop.~2.2]{Trostorff2011}, \cite[prop.~4.7]{Waurick2023} or \cite[prop.~2.3]{Aigner2024}.
An easy corollary is the following result:
\begin{corollary}
  Let $\rho > 0$ and $-\infty < a < b \leq +\infty$. For $I_{\rho}\colon\,L_{2,\rho}(a,b;H)\rightarrow H^{1}_{\rho}(a,b;H)$ we have $\norm{I_{\rho}}\leq \sqrt{\frac{1}{\rho^{2}}+1}$.
\end{corollary}
\begin{proof}
  For $v\in L_{2,\rho}(a,b;H)$ we infer $I_{\rho}v \in H^{1}_{\rho}(a,b;H)$ and hence
  \begin{equation*}
    \norm{I_{\rho}v}^{2}_{H^{1}_{\rho}}=\norm{I_{\rho}v}^{2}_{L_{2,\rho}}+\norm{(I_{\rho}v)'}^{2}_{L_{2,\rho}}\leq \bigl(\tfrac{1}{\rho^{2}}+1\bigr)\norm{v}^{2}_{L_{2,\rho}}\text{.} \qedhere
  \end{equation*}
\end{proof}
This allows us to define the time derivative with boundary condition in $L_{2,\rho}$. We could do this in general, but for the purposes of this article it suffices to consider the case $a=0$ and $b=\infty$ in this instance.
\begin{definition}
  \label{def:TimeDerBoundary}
  For $\rho >0$ let
  \begin{equation*}
    \mathring{\partial}_{\rho}\colon \, L_{2,\rho}(0,\infty;H)\supseteq \operatorname{dom}\bigl(\mathring{\partial}_{\rho}\bigr) \to L_{2,\rho}(0,\infty;H)\text{,}
    \qquad\mathring{\partial}_{\rho}\coloneq I_{\rho}^{-1}\text{.}
  \end{equation*}
  Furthermore, let $\partial_{\rho}$ be the standard (distributionally defined) derivative on $L_{2,\rho}$.\footnote{We simply attach the index $\rho$ to make it clear that the derivative is taken in $L_{2,\rho}$. We will make no distinction between $\partial_{\rho}$ defined on $\mathbb{R}$ or $(0,\infty)$.}
\end{definition}
The operator $\mathring{\partial}_{\rho}$ is well-defined since $I_{\rho}$ is injective and a bounded linear operator. In particular, $\mathring{\partial}_{\rho}$ is closed. To show that $\mathring{\partial}_{\rho}$ is indeed an operator with a boundary condition, we need to evaluate $H^{1}_{\rho}(0,\infty;H)$-functions pointwise. As with unweighted Sobolev spaces, we are free to do this appealing to a continuous embedding. For the purpose of stating such an embedding, we introduce the following notation:
\begin{definition}
  \label{def:ContFunctionsWeighted}
  For $h\geq 0$ let
  \begin{align*}
    \mathcal{C}_{\rho}(-h,\infty;H)&\coloneq \bigl\{u\colon [-h,\infty)\to H\colon\,u\,\text{is continuous and}\,\,\norm{u}_{\rho,\infty}<\infty\bigr\}\text{,}\\
    \norm{u}_{\rho,\infty}&\coloneq \sup_{t\geq -h}\e^{-\rho t}\norm{u(t)}_{H}\text{,}\\
    \mathcal{C}_{0,\rho}(-h,\infty;H) &\coloneq \bigl\{u\in \mathcal{C}_{\rho}(-h,\infty;H)\colon\, \lim_{t\to\infty}\e^{-\rho t}u(t)=0\bigr\}\text{.}
  \end{align*}
\end{definition}
Due to the one-dimensional setting (in time), the following classical result holds:
\begin{theorem}[Sobolev-embedding]
  \label{th:Sobolev}
  Let $h\geq 0$. Then the continuous embedding  $H^{1}_{\rho}(-h,\infty;H)\hookrightarrow \mathcal{C}_{0,\rho}(-h,\infty;H)$ holds. Moreover, every $u\in H^{1}_{\rho}$ admits a continuous representative.
\end{theorem}
For a proof of this statement we refer to \cite[prop.~2.29]{Trostorff2011}. \Cref{th:Sobolev} in particular allows us to evaluate $H^{1}$-functions pointwise.
\begin{proposition}
  \label{th:DomDerivativeFor0}
  Let $\rho > 0$ and $u\in \operatorname{dom}(\partial_{\rho})$ on $L_{2,\rho}(0,\infty;H)$. Then $u\in \operatorname{dom}\bigl(\mathring{\partial}_{\rho}\bigr)$ if and only if $u(0)=0$.
\end{proposition}

\begin{proof}
  Let $u \in \operatorname{dom}(\partial_{\rho})$. Due to the fundamental theorem of calculus,
  \begin{equation*}
    u(t) = u(0) + \medint\int_{0}^{t}\partial_{\rho}u (s) \dx[s]\text{,}\qquad t\geq 0\text{.}
  \end{equation*}
  The equivalence is now immediate.
\end{proof}
Consequently, we set for $\rho >0$:
\begin{equation*}
  H_{0,\rho}^{1}\coloneq \, \bigl\{u\in H^{1}_{\rho}\left(0,\infty;H\right)\colon\,u(0)=0\bigr\} = \operatorname{dom}\bigl(\mathring{\partial}_{\rho}\bigr)\text{.}
\end{equation*}

\begin{remark}
  Completely analogously to the half-line case, one defines the operator $I_{\rho}$ on $L_{2,\rho}(\mathbb{R};H)$ or $H^{1}_{\rho}(\mathbb{R};H)$ as antiderivative $I_{\rho} = \smallint_{-\infty}^{\argdot}$, which produces a bounded linear operator adhering to the same estimates as obove and gives rise to the “standard” (i.e., distributionally defined) time derivative $\partial_{\rho}$. On the entire real line, integrability replaces the boundary condition on the left endpoint.
\end{remark}
To conclude this first subsection, we want to investigate the relationship between $\mathring{\partial}_{\rho}$ and $\partial_{\rho}$ a bit more. To that end, we make use of an integration by parts formula:
\begin{proposition}[integration by parts]
  \label{th:IntbyParts}
  On $H^{1}_{\rho}\left(a,b;H\right)$, $-\infty < a \leq b \leq \infty$, the following holds for $u,v\in H^{1}_{\rho}\left(a,b;H\right)$:
  \begin{equation*}
    \dualprod{\partial_{\rho}u}{v}_{L_{2,\rho}} \!= \!-\dualprod{u}{\partial_{\rho}v}_{L_{2,\rho}} \!\!+ \!2\rho \dualprod{u}{v}_{L_{2,\rho}} \!\!+ \!\e^{-2\rho b} \dualprod{u(b)}{v(b)}_{H} \!- \!\e^{-2\rho a} \dualprod{u(a)}{v(a)}_{H}\text{,}
  \end{equation*}
  where the boundary term at $b$ vanishes for $b=\infty$ and the boundary term at $a$ vanishes if at least one of $u,v$ is in $H^{1}_{0,\rho}(a,\infty;H)$.\footnote{For $a\neq 0$ one defines $H_{0,\rho}^{1}(a,\infty;H)\coloneq \{u \in H^{1}_{\rho}(a,\infty;H)\colon u(a) = 0\}$.}
\end{proposition}
\begin{proof}
For $b<\infty$ we calculate for $\varphi,\psi \in \mathcal{C}^{1}(a,b;H)$:
\begin{align*}
  \dualprod{\partial_{\rho}u}{v}_{L_{2,\rho}}  \!&= \medint\int_{a}^{b} \dualprod{u'(s)}{v(s)}_{H} \e^{-2\rho s}\dx[s]\\
                                              &= - \!\medint\int_{a}^{b} \bigl[\dualprod{u(s)}{v'(s)}_{H}\!- \!\dualprod{u(s)}{v(s)}_{H} 2\rho \bigr] \e^{-2\rho s}\dx[s] \\
                                              &\quad \,+ \!\e^{-2\rho b}\dualprod{u(b)}{v(b)}_{H} \!- \!\e^{-2\rho a}\dualprod{u(a)}{v(a)}_{H}\\
                                              &= - \dualprod{u}{\partial_{\rho}v}_{L_{2,\rho}} \!\!+ \!2\rho \dualprod{u}{v}_{L_{2,\rho}} \!\!+ \!\e^{-2\rho b} \dualprod{u(b)}{v(b)}_{H} \!- \!\e^{-2\rho a} \dualprod{u(a)}{v(a)}_{H}\text{.}
\end{align*}
We obtain the proclaimed formula by appealing to density of $\mathcal{C}^{1}(a,b;H)$ in $H^{1}_{\rho}(a,b;H)$ and the (classical) Sobolev embedding theorem.\\
For the case $b = \infty$ we appeal to the embedding \cref{th:Sobolev}, specifically to the fact that for $u\in H^{1}_{\rho}(0,\infty;H)$ one has $\lim_{t\to \infty} \e^{-\rho t}u(t)=0$. That the boundary terms at $a$ vanish for $u$ or $v$ in $H^{1}_{0,\rho}(a,b;H)$ is clear.
\end{proof}
\begin{proposition}
  \label{th:TimeDer_and_Adjoint}
  Let $\rho >0$. Then, on $L_{2,\rho}(0,\infty;H)$ hold
  \begin{equation*}
    \mathring{\partial}_{\rho}\subseteq \partial_{\rho} \qquad\text{and}\qquad
    \partial_{\rho}=-\bigl(\mathring{\partial}_{\rho}\bigr)^{\ast}+2\rho\text{.}
  \end{equation*}
\end{proposition}
\begin{proof}
  For the first statement, we take $\varphi \in \mathcal{C}^{\infty}_{c}(0,\infty)$ and $u \in \operatorname{dom}\bigl(\mathring{\partial}_{\rho}\bigr)$. The latter implies by \cref{def:TimeDerBoundary} that $u = I_{\rho}v$ for some $v \in L_{2,\rho}(0,\infty;H)$. We can calculate:
  \begin{align*}
    \medint\int_{0}^{\infty}u(t) \varphi'(t)\dx[x]
    &= \medint \int_{0}^{\infty} \medint\int_{0}^{t} v(s) \dx[s] \varphi'(t) \dx[t]
    = \medint \int_{0}^{\infty} \medint\int_{0}^{t} v(s) \varphi'(t) \dx[s] \dx[t]\\
    &= \medint \int_{0}^{\infty} \medint\int_{s}^{\infty} v(s) \varphi'(t) \dx[t] \dx[s]
    = - \medint \int_{0}^{\infty} v(t)\varphi(t) \dx[t] \text{.}
  \end{align*}
  Hence, $u \in \operatorname{dom}\bigl(\partial_{\rho}\bigr)$.\\
  For the second statement, we first observe that $\mathcal{C}^{\infty}_{c}(0,\infty;H) \subseteq \operatorname{dom}\bigl(\mathring{\partial}_{\rho}\bigr)\subseteq \operatorname{dom}\bigl(\partial_{\rho}\bigr)$, appealing to the first part. Now let $u \in \operatorname{dom}\bigl(\partial_{\rho}\bigr)$. For $v \in \operatorname{dom}\bigl(\mathring{\partial}_{\rho}\bigr)$ we compute using \cref{th:IntbyParts} and $v(0)= 0$:
  \begin{align*}
    -\dualprod{u'}{v} + 2\rho \dualprod{u}{v} &= \medint\int_0^{\infty} \dualprod{u(t)}{v'(t)}\e^{-2\rho t} \dx[t]\\
                                              &= \medint\int_0^{\infty} \dualprod{u(t)}{\partial_{\rho}v(t)}\e^{-2\rho t} \dx[t]
                                              = \dualprod{u}{\mathring{\partial}_{\rho} v}\text{.}
  \end{align*}
  Hence, $u \in \operatorname{dom}\bigl((\mathring{\partial}_{\rho})^{\ast}\bigr)$ and consequently  $\operatorname{dom}\bigl(\partial_{\rho}\bigr)\subseteq \operatorname{dom}\bigl((\mathring{\partial}_{\rho})^{\ast}\bigr)$.\\
  Now let $u\in \operatorname{dom}\bigl((\mathring{\partial}_{\rho})\adjun\bigr)$ and $\varphi \in \mathcal{C}^{\infty}_{c}(0,\infty)$. We note that $\operatorname{lin}\{x\cdot \varphi \colon x\in H,\,\varphi \in \mathcal{C}^{\infty}_{c}(0,\infty)\} \subseteq \operatorname{dom}\bigl(\mathring{\partial}_{\rho}\bigr)$ is a dense subspace. We let $\psi \coloneq \e^{-2\rho \argdot}\varphi \in \mathcal{C}^{\infty}_{c}(0,\infty)$ and compute:
  \begin{align*}
    \medint\int_{0}^{\infty}(\mathring{\partial}_{\rho})\adjun u (t) \overline{\psi (t) x} \dx[t]
    &= \dualprod[\big]{(\mathring{\partial}_{\rho})\adjun u}{\varphi x}_{L_{2,\rho}}
      = \dualprod[\big]{u}{\mathring{\partial}_{\rho}\varphi x}_{L_{2,\rho}}\\
    &= \medint\int_{0}^{\infty} \dualprod{u(t)}{x}_{H} \overline{\varphi'(t)}\e^{-2\rho t} \dx[t]\\
    &= \medint\int_{0}^{\infty} \dualprod{u(t)}{x}_{H} \bigl[\bigl(\overline{\varphi(t)}\e^{-2\rho t}\bigr)' + 2\rho \overline{\varphi(t)}\e^{-2\rho t}\bigr] \dx[t]\text{.}
  \end{align*}
  Since the map
  \begin{equation*}
    \mathcal{C}^{\infty}_{c}(0,\infty) \to \mathcal{C}^{\infty}_{c}(0,\infty)\text{,} \qquad \varphi \mapsto \varphi \e^{-2\rho \argdot}
  \end{equation*}
  is a bijection, and because of the density of $\operatorname{lin}\{x\cdot \varphi \colon x\in H,\,\varphi \in \mathcal{C}^{\infty}_{c}(0,\infty)\}$ in $\operatorname{dom}\bigl(\mathring{\partial}_{\rho}\bigr)$, the following vector-valued integral equality holds for all $\psi \in \mathcal{C}^{\infty}_{\mathrm{c}}(0,\infty)$:
  \begin{equation*}
    \medint\int_{0}^{\infty} u(t)\psi'(t) \dx[t] = \medint\int_{0}^{\infty} \bigl[\bigl(\mathring{\partial}_{\rho}\bigr)^{\ast}u(t) - 2 \rho u(t)\bigl]\psi(t) \dx[t]\text{.}
  \end{equation*}
  Hence, $u \in \operatorname{dom}\bigl(\partial_{\rho}\bigr)$. The claimed formula follows from either calculation.
\end{proof}
We make an additional observation:
\begin{remark}[$\mathring{\partial}_{\rho}$ is $\mathrm{m}$-accretive]
  \label{rmk:evd_is_maximal_monotone}
  We can appeal to \cref{th:TimeDer_and_Adjoint} and infer for $u \in \operatorname{dom}(\mathring{\partial}_{\rho})$:
  \begin{align*}
    2 \Re \dualprod[\big]{\mathring{\partial}_{\rho}u}{u}
    &= \dualprod[\big]{\mathring{\partial}_{\rho} u}{u} + \dualprod[\big]{u}{\mathring{\partial}_{\rho}u}\\
    &= \dualprod[\big]{\mathring{\partial}_{\rho} u}{u} + \dualprod[\big]{(\mathring{\partial}_{\rho})^{\ast}u}{u}\\
    &= \dualprod{u'}{u} - \dualprod{u'}{u} + 2\rho\dualprod{u}{u}\\
    &= 2\rho \norm{u}^{2}\text{,}
  \end{align*}
  i.e., $\mathring{\partial}_{\rho}$ is monotone/accretive. It is maximal, because the inverse $\bigl(\mathring{\partial}_{\rho}\bigr)^{-1}=I_{\rho}$ is a bounded linear operator; in particular, $\mathring{\partial}_{\rho}$ is surjective and therefore there cannot exist a proper monotone extension. We note that, in consequence, $\mathbb{R}_{< 0}\subseteq \uprho (\mathring{\partial}_{\rho})$.
\end{remark}
Having introduced the fundamental concept of the invertible time derivative, we can proceed to define material law operators.

\subsection{Material law operators}
\label{subsec:MatLaw}
\begin{definition}
  \label{def:MatLaw}
  A map $M\colon \C \supseteq \operatorname{dom}(M) \to \mathcal{L}(H)$ is called a {\em material law} if
  \begin{enumerate}[label=\roman*), leftmargin=4ex]
    \item $\operatorname{dom}(M)$ is open and $M$ is holomorphic and
    \item there exists some $\rho \in \R$ such that $\C_{\Re > \rho}\subseteq \operatorname{dom}(M)$ and $\norm{M}_{\infty, \C_{\Re > \rho}}<\infty$.
  \end{enumerate}
  Moreover, we call $s_{\mathrm{b}} (M)\coloneq \inf \{\rho \in \R \colon \C_{\Re > \rho} \subseteq \operatorname{dom}(M) \,\, \wedge \,\, \norm{M}_{\infty, \C_{\Re > \rho}}<\infty\}$ the {\em abscissa of boundedness} of $M$.
\end{definition}
Utilizing a functional calculus, one can give meaning to $M(\partial_{\rho})$ via the so-called {\em Fourier-Laplace transform} $\mathcal{L_{\rho}}$ which is the unitary extension of
\begin{align}
  \label{def:FourierLaplace}
  \begin{aligned}
  \mathcal{C}_{c}^{\infty}(\mathbb{R};H) &\to \mathcal{C}^{\infty}(\mathbb{R};H)\text{,}\\
  \phi &\mapsto \Bigl(t\mapsto \tfrac{1}{\sqrt{2\pi}} \medint\int_{\mathbb{R}} \e^{-(\iu t + \rho)s} \phi(s)\dx[s]\Bigr)
  \end{aligned}
\end{align}
to a map $\mathcal{L_{\rho}}\colon L_{2,\rho}\left(\mathbb{R};H\right)\to L_{2}(\mathbb{R};H)$, cf.~\cite[ch.~5.2]{Waurick2022}.
\begin{definition}
  \label{def:MatLawOp}
  We define the {\em material law operator}
  \begin{equation*}
    M(\partial_{\rho})\coloneq \mathcal{L}^{\ast}_{\rho} M(\iu \mathrm{m} + \rho)\mathcal{L}_{\rho} \in \mathcal{L}\bigl(L_{2,\rho}(\R;H)\bigr)\text{,}
  \end{equation*}
  where $\mathrm{m}$ denotes the multiplication with the argument operator.
\end{definition}
This definition holds for the entire real line (i.e., the case $L_{2,\rho}(\mathbb{R};H)$). It is not immediately obvious what $M\bigl(\mathring{\partial}_{\rho}\bigr)$ is supposed to be in case of the half-line (i.e., on $L_{2,\rho}(0,\infty;H)$). We will forego this problem later on by using a projection argument. Material laws exhibit the following important property:
\begin{definition}
  Let $H_{0},H_{1}$ be Hilbert spaces, $\rho \in \mathbb{R}$ and $F\colon L_{2,\rho}(\mathbb{R};H_{0}) \to L_{2,\rho}(\mathbb{R};H_{1})$. F is called {\em causal} if
  \begin{equation*}
    \forall a \in \mathbb{R}\forall f,g\in L_{2,\rho}(\mathbb{R};H_{0})\colon f|_{(-\infty,a]} = g|_{(-\infty,a]} \implies F(f)|_{(-\infty,a]} = F(g)|_{(-\infty,a]}\text{.}
  \end{equation*}
\end{definition}
This property is of key importance in the proof of the main theorem for evolutionary equations, which is the following result due to R.~Picard.
\begin{theorem}[Picard, {\cite[thm~6.2.1, rem.~6.3.3]{Waurick2022}}]
  \label{th:Picard}
  Let $\rho_{0}\in \mathbb{R}$, let $M\colon \C \supseteq \operatorname{dom}(M)\to \mathcal{L}(H)$ be a material law with $s_{\mathrm{b}}(M)<\rho_{0}$ and let $A\colon H\supseteq \operatorname{dom}(A) \to H$ be $\mathrm{m}$-accretive. Assume that
  \begin{equation*}
    \exists c > 0 \forall\varphi \in H, z \in \C_{\Re > \rho_{0}}\colon \quad\Re \dualprod{\varphi}{zM(z)\varphi}_{H} \geq c \norm{\varphi}_{H}^{2}\text{.}
  \end{equation*}
  Then, for all $\rho > \rho_{0}$, the operator $\partial_{\rho}M(\partial_{\rho}) + A$ is closable and
  \begin{equation*}
    S_{\rho} \coloneq \bigl(\overline{\partial_{\rho}M(\partial_{\rho}) + A}\bigr)^{-1} \in \mathcal{L}\bigl(L_{2,\rho}(\mathbb{R};H)\bigr)\text{.}
  \end{equation*}
  Furthermore,
  \begin{enumerate}[label=\roman*), leftmargin=5ex]
    \item $S_{\rho}$ is causal and satisfies $\norm{S_{\rho}}_{L_{2,\rho}}\leq \nicefrac{1}{c}$,
    \item for all $g \in \operatorname{dom}(\partial_{\rho})$, $S_{\rho}g \in \operatorname{dom}(\partial_{\rho})\cap \operatorname{dom}(A)$,
    \item for $\eta, \nu > \rho_{0}$ and $g \in L_{2,\eta}(\mathbb{R};H)\cap L_{2,\nu}(\mathbb{R};H)$, $S_{\eta}g = S_{\nu}g$.
  \end{enumerate}
\end{theorem}

\Cref{th:Picard} leads to a natural definition of weak solutions for the problem
\begin{equation}
  \label{eq:EvolEq}
  \bigl(\partial_{\rho}M(\partial_{\rho}) + A\bigr)u = g
\end{equation}
on $L_{2,\rho}(0,\infty;H)$ via the {\em Picard operator} $S_{\rho} = \bigl(\overline{\partial_{\rho}M(\partial_{\rho}) + A}\bigr)^{-1}$:
\begin{definition}
  \label{def:weakSol}
  For a right-hand side $g \in L_{2,\rho}(0,\infty;H)$, we refer to $u=S_{\rho}g$ as a {\em weak solution}\footnote{We simply extend $g$ trivially by $0$ to obtain a function in $L_{2,\rho}(\mathbb{R},H)$.} to \cref{eq:EvolEq}. For a {\em strong solution} we additionally demand $M(\partial_{\rho})u \in \operatorname{dom}\bigl(\mathring{\partial}_{\rho}\bigr)$.
\end{definition}
In the latter case, appealing to \cref{th:Picard}, we have
\begin{equation*}
  \bigl(\overline{\partial_{\rho}M(\partial_{\rho}) + A}\bigr)^{-1} g
  = \bigl(\partial_{\rho}M(\partial_{\rho}) + A\bigr)^{-1}g\text{.}
\end{equation*}
If not mentioned otherwise, under {\em solution}, we will always understand a weak solution of a problem in this article.
\begin{remark}
  Picard's theorem proves well-posedness of the evolutionary equation
  \begin{equation*}
    \bigl(\overline{\partial_{\rho}M(\partial_{\rho}) + A}\bigr)u = g,
  \end{equation*}
  as well as eventual independence of the weight parameter $\rho$ (i.e., item iii) in \cref{th:Picard}), provides a criterion for a solution being a classical solution (i.e., item ii)) and links material laws to the solution operator by proving that the {\em Picard operator} $S_{\rho}$ is a material law operator (cf.~\cite[rem.~6.3.4]{Waurick2022}).
\end{remark}

\subsection{Distributional spaces and IVPs}
\label{subsec:DistribSpaces_IVPs}
In this article, we want to investigate \cref{eq:FDE}, which can be viewed as a ``generalized initial value problem'', because IVPs are closely related to evolutionary equations in distributional spaces over the full line. This subsection serves as a motivation. We recall the following definition:
\begin{definition}
  Let $H_{0}$ and $H_{1}$ be Hilbert spaces and let $C\colon H_{0}\supseteq \operatorname{dom}(C) \to H_{1}$ be linear, densely defined and closed. We set
  \begin{align*}
    C^{\diamond}&\colon H_{1}\to \operatorname{dom}(C)' \eqcolon H^{-1}(C)\text{,}\\
    (C^{\diamond}\varphi)(x)&\!\coloneq \dualprod{\varphi}{C x}_{H_{1}} \qquad \varphi \in H_{1},x\in \operatorname{dom}(C)\text{.}
  \end{align*}
  We call $C_{-1}\coloneq (C^{\ast})^{\diamond}$ the {\em extrapolated operator} of $C$.
\end{definition}
Note that $C^{\diamond}$ is related to the (Banach space) dual $C'$ of $C$ when considered as a bounded operator from $\operatorname{dom}(C)$ to $H_{1}$ via $C^{\diamond}= C'R_{H_{1}}^{-1}$, where $R_{H_{1}}\colon H_{1}\to H_{1}'$ is the Riesz-isomorphism.\\
The idea behind extrapolated operators is to formalize computations that are not possible in the given (small) space, by introducing a larger space. Indeed, the inclusions $C_{-1}\supseteq C$ and $(C^{\ast})_{-1}\supseteq C^{\ast}$ hold. For basics on extrapolation spaces, we refer to \cite[sec.~9.2]{Waurick2022}.\\
On the entire real line we have $H^{1}_{\rho}(\mathbb{R};H) = \operatorname{dom}(\partial_{\rho})$. Similarly we can now define
\begin{equation*}
  H^{-1}_{\rho}(\mathbb{R};H) \coloneq \operatorname{dom}(\partial_{\rho})'\text{.}
\end{equation*}
The definition for higher exponents is self-explanatory. The next theorem
will provide a crucial transition point:
\begin{theorem}[{\cite[thm.~9.3.2]{Waurick2022}}]
  \label{th:EquivDistrProb}
  Let $u,f \in L_{2,\rho}\left(\mathbb{R};H\right)$. The following are equivalent:
  \begin{enumerate}[label= \roman*), leftmargin=4ex]
    \item $u\in \operatorname{dom}\bigl(\overline{\partial_{\rho}M(\partial_{\rho})+A}\bigr)$ and $\bigl(\overline{\partial_{\rho}M(\partial_{\rho}) + A}\bigr)u=f$.
    \item $u \in \operatorname{dom}\bigl(\overline{\partial_{\rho}M(\partial_{\rho})+A}\bigr)$ and $\bigl(\partial_{\rho}M(\partial_{\rho})\bigr)_{-1}u + A_{-1}u\subseteq f$.
    \item $\partial_{\rho}M(\partial_{\rho})u + Au \subseteq f$, where the left-hand side is considered as an element of $H^{-1}_{\rho}(\mathbb{R};H)\cap L_{2,\rho}\bigl(\mathbb{R};H^{-1}(A)\bigr)$.
  \end{enumerate}
\end{theorem}
The gist of this theorem is that the closure bar can be ``removed'' by passing to the distributional formulation. This allows us to tackle the sum of the operators directly and we can start to take a look at IVPs. The problem we want to investigate is the following pair of informal equations
\begin{alignat}{3}
  \tag{\ref{eq:IVP}}
  \begin{aligned}
    \bigl(\partial_{t}M(\partial_{t}) + A\bigr)u(t)&= F(t) &&\qquad t>0\text{,}\\
    u(t)&=\Phi(t) && \qquad t\leq 0\text{.}
  \end{aligned}
\end{alignat}
We point out that the first equation is not well-defined on the half-line, but on the entire real line only. To formulate evolutionary equations with initial values on the half-line, we follow S.~Trostorff, \cite[ch.~3]{Trostorff2018} and use a projection technique.
\begin{definition}
  \label{def:cutoff}
  For $\rho>0$ and $t\in \mathbb{R}$ let
  \begin{align*}
    P_{t}\colon H_{\rho}^{-1}(\mathbb{R};H) \supseteq \operatorname{dom}(P_{t}) &\to H_{\rho}^{-1}(\mathbb{R};H)\text{,}\\
    f &\mapsto \partial_{\rho}\chi_{\mathbb{R}_{\geq t}}\partial_{\rho}^{-1}f
        - \e^{-2\rho t}\bigl(\partial_{\rho}^{-1}f\bigr)(t+)\delta_{t}\text{,}
  \end{align*}
  where $\operatorname{dom}(P_{t})\coloneq \{f\in H_{\rho}^{-1}(\mathbb{R};H)\colon (\partial_{\rho}^{-1}f)(t+)\,\,\text{exists}\}$ and $\delta_{t}$ denotes the {\em Dirac distribution}, cf.~\cref{subsec:Dirac} for a definition.
\end{definition}
\begin{remark}[$P_{t}$ is a generalized cutoff $\chi_{\mathbb{R}_{\geq t}}$]
  \label{th:P_t_is_cutoff}

  Let $f \in L_{2,\rho}\left(\mathbb{R};H\right)$. Then $\partial_{\rho}^{-1}f = I_{\rho}f = \int_{-\infty}^{\argdot}f(s)\dx[s]$ exists and for $\varphi\in \mathcal{C}_{c}^{\infty}(\mathbb{R};H)\subseteq L_{2,\rho}\left(\mathbb{R};H\right)$ we can calculate using \cref{th:TimeDer_and_Adjoint}:
  \begin{align*}
    \dualprod[\big]{\partial_{\rho}\chi_{\mathbb{R}_{\geq t}}\partial_{\rho}^{-1}f}{\varphi}
    &= \dualprod[\big]{\chi_{\mathbb{R}_{\geq t}}I_{\rho}f}{(\partial_{\rho})^{\ast}\varphi}\\
    &= \dualprod[\big]{\chi_{\mathbb{R}_{\geq t}}I_{\rho}f}{-\partial_{\rho}\varphi + 2\rho \varphi}\\
    &= \medint\int_t^{\infty} \dualprod[\big]{(I_{\rho}f)(s)}{-(\varphi(s)\e^{-2\rho s})'}_{H} \dx[s]\\
    &= \medint\int_{t}^{\infty} \dualprod[\big]{f(s)}{\varphi(s)}_{H}\e^{-2\rho s} \dx[s] + \e^{-2\rho t}(I_{\rho}f)(t) \varphi(t)\\
    &= \dualprod[\big]{\chi_{\mathbb{R}_{\geq t}}f}{\varphi} + \dualprod[\big]{\e^{-2\rho t} (I_{\rho}f)(t)\delta_{t}}{\varphi}\text{.}
  \end{align*}
  Subtraction of the rightmost term from the last equation shows $P_{t}f = \chi_{\mathbb{R}_{\geq t}}f$, appealing to density of $\mathcal{C}^{\infty}_{c}(\mathbb{R};H)$ in $L_{2,\rho}(\mathbb{R};H)$.
\end{remark}
Furthermore one can prove:
\begin{proposition}[{\cite[prop.~3.1.16]{Trostorff2018}}]
  Let $f\in H_{\rho}^{-1}(\mathbb{R};H)$. Then $\supp(f)\subseteq \mathbb{R}_{\leq t}$ if and only if $f\in \operatorname{dom}(P_{t})$ with $P_{t}f=0$.
\end{proposition}
Now we can replace the differential equation from problem \eqref{eq:IVP} with the formal projected version on the entire real line:
\begin{equation*}
  P_{0}\bigl(\partial_{\rho}M(\partial_{\rho}) + A\bigr)u = P_{0}F\text{.}
\end{equation*}
Note that $P_{0}F = F$ for $F \in L_{2,\rho}(0,\infty;H)$ by \cref{th:P_t_is_cutoff}. We can now further make an ansatz $u=v+Z$, where $v \in L_{2,\rho}(0,\infty;H)$ and $Z$ is an $H^{1}_{\rho}$-extension of the prehistory $\Phi \in H^{1}(-h,0;H)$ to the right.
$Z$ will take various forms later on, the crucial part however is $\supp(v) \subseteq [0,\infty)$. Plugging this ansatz into the projected equation, one obtains the problem
\begin{equation}
  \label{eq:DistribProblem}
  \bigl(\partial_{\rho}M(\partial_{\rho}) + A\bigr)v = F\bigl(\argdot, (v+Z)_{(\argdot)}\bigr) + f \qquad \text{on}\ H^{-1}_{\rho}(\mathbb{R};H)
\end{equation}
for some $f$ with $\supp (f)\subseteq [0,\infty)$ instead. The additional term $f$ contains information on the initial prehistory and incorporates behaviour of the material law at the origin. The calculation of $f$ is manageable for simple material laws, i.e., material laws of the form $M(z)=M_{0} + M_{1}z^{-1}$, because in this case, the cutoff only depends on the values at the origin. This is not the case for more complex material laws. For details and further reading we refer to \cref{sec:IVPs} and \cite[sec.~3.2]{Trostorff2018}. In any case, \cref{eq:DistribProblem} does not need the initial prehistory as a separate equation any more.\\
This subsection was our justification to pivot to evolutionary equations on the real line (albeit in distributional spaces in general) to prove well-posedness and forego giving meaning to evolutionary problems on the half-line only. The distributional
space $H^{-1}_{\rho}(\mathbb{R};H)$ is still too large to apply our tools though. We will pivot further to a formulation in $L_{2,\rho}(\mathbb{R};H)$ (and in fact $H^{1}_{\rho}$) later, i.e.,
\begin{equation}
  \label{eq:EvolProblem}
  \bigl(\overline{\partial_{\rho}M(\partial_{\rho}) + A}\bigr)v = F\bigl(\argdot, (v+Z)_{(\argdot)}\bigr) + f \qquad \text{on}\ L_{2,\rho}(\mathbb{R};H)\text{.}
\end{equation}
To demand that the operators and functions produce elements in $L_{2,\rho}$ is akin to the functions and operators in \cref{eq:DistribProblem} satisfying a consistency condition, which is an assumption occuring frequently in the study of DDEs. We point out that since $F(t,x)=0$ for $t<0$ and since $\supp (f) \subseteq [0,\infty)$, \cref{eq:EvolProblem} is essentially an equation on the half-line only, appealing to the causality of the Picard operator.

\subsection{Regularity theory in evolutionary equations}
\label{subsec:RegThy}
It will be critically important to have a regularity result for evolutionary equations. In particular, if a right-hand side $g$ in \cref{eq:EvolEq} satisfies $g\in \operatorname{dom}\bigl(\mathring{\partial}_{\rho}\bigr)$, the solution $S_{\rho}g$ should satisfy $S_{\rho}g \in \operatorname{dom}\bigl(\mathring{\partial}_{\rho}\bigr)$, where we regard $\operatorname{dom}\bigl(\mathring{\partial}_{\rho}\bigr) \subseteq L_{2,\rho}(\mathbb{R};H)$, extending functions trivially by $0$.
\begin{theorem}[regularity]
  \label{th:Commutation}
  For an $\mathrm{m}$-accretive operator $A$ and a material law $M$ satisfying the assumptions of \cref{th:Picard}, the operator $S_{\rho} = \bigl(\overline{\partial_{\rho}M(\partial_{\rho}) + A}\bigr)^{-1}$ satisfies
  \begin{equation*}
    S_{\rho}\mathring{\partial}_{\rho} \subseteq \mathring{\partial}_{\rho}S_{\rho} \quad \text{and}\quad
    S_{\rho}I_{\rho} = I_{\rho}S_{\rho}\text{.}
  \end{equation*}
  In particular, $g\in \operatorname{dom}\bigl(\mathring{\partial}_{\rho}\bigr) \implies S_{\rho}g \in \operatorname{dom}\bigl(\mathring{\partial}_{\rho}\bigr)$.
\end{theorem}
\begin{proof}
  For a proof of the inclusion $S_{\rho}\partial_{\rho} \subseteq \partial_{\rho}S_{\rho}$ we refer to \cite[rmk.~6.3.4]{Waurick2022}. The inclusion $S_{\rho}\mathring{\partial}_{\rho} \subseteq \mathring{\partial}_{\rho}S_{\rho}$ is a simple consequence of causality:\\
  Let $g \in \operatorname{dom}\bigl(\mathring{\partial}_{\rho}\bigr)$. Then $g \in \operatorname{dom}(\partial_{\rho})$ and \cref{th:Picard} implies $S_{\rho}g \in \operatorname{dom}(\partial_{\rho})$; in particular $S_{\rho}g$ is continuous appealing to \cref{th:Sobolev}. Moreover, since $\supp(g)\subseteq [0,\infty)$, causality of $S_{\rho}$ yields $\supp(S_{\rho}g)\subseteq [0,\infty)$. Hence, $0 = \bigl(S_{\rho}g\bigr)(0)$ and thus $S_{\rho}g\in \operatorname{dom}\bigl(\mathring{\partial}_{\rho}\bigr)$ by \cref{th:DomDerivativeFor0}.\\
  The equality $S_{\rho}I_{\rho} = I_{\rho}S_{\rho}$ is obtained by applying $I_{\rho}$ on both sides of the inclusion $S_{\rho}\partial_{\rho} \subseteq \partial_{\rho}S_{\rho}$, noting that $\partial_{\rho}I_{\rho}=\mathrm{id}_{L_{2,\rho}(\mathbb{R};H)}$.
\end{proof}

For the heat equation (cf.~\cref{subsec:Heat}) we can say even more:
\begin{definition}
  \label{def:ParabolicPair}
  Let $M\colon \C \supseteq \operatorname{dom}(M) \to H$ be a material law of the form
  \begin{equation*}
    M(z) \coloneq \begin{pmatrix} M_{00}(z) & 0 \\ 0 & 0 \end{pmatrix} + z^{-1}N(z)
  \end{equation*}
  on the decomposed space $H = H_{0}\oplus H_{0}^{\perp}$ for some closed linear subspace $H_{0}\subseteq H$ and let $C\colon H_{0} \supseteq \operatorname{dom}(C) \to H_{0}^{\perp}$ be closed and densely defined. Let $A \coloneq \begin{psmallmatrix} 0 & -C^{\ast} \\ C & 0 \end{psmallmatrix}$. The pair $(M,A)$ is called a {\em parabolic pair} if there exists $\rho > \max\{0,s_{\mathrm{b}}(M)\}$ such that
  \begin{equation*}
    \forall z \in \C_{\Re > \rho}\colon \Re M_{00}(z) \geq c \quad \wedge \quad \C_{\Re >\rho}\ni z \mapsto M_{00}(z)\in \mathcal{L}(H) \,\,\text{is bounded}.
  \end{equation*}
\end{definition}

\begin{theorem}[{\cite[thm.~15.2.3]{Waurick2022}}]
  \label{th:parabolicRegularity}
  Let $(M,A)$ be a parabolic pair such that for some $\rho>\max \{0, s_{\mathrm{b}}(M)\}$ for all $z\in \C_{\Re > \mu}$ ($\mu > \rho$):
  \begin{equation*}
    \Re \bigl(M_{00}(z) + z^{-1}N_{00}(z)\bigr) \geq c\text{,}
  \end{equation*}
  where $N_{00}(z) \coloneq \iota^{\ast}_{H_{0}}N(z)\iota_{H_{0}}$.\footnote{Here, $\iota_{H_{0}}$ is the embedding $H_{0}\hookrightarrow H$.} Let $f \in L_{2,\rho}\left(\mathbb{R};H_{0}\right)$ and $g\in H_{\rho}^{\nicefrac{1}{2}}\left(\mathbb{R};H_{0}^{\perp}\right)$.\\
  Then the solution $(u,v)\coloneq S_{\rho}(f,g)\in L_{2,\rho}\left(\mathbb{R};H\right)$ satisfies:\footnote{Here, $H^{\nicefrac{1}{2}}_{\rho} = \operatorname{dom}\bigl(\partial_{\rho}^{\nicefrac{1}{2}}\bigr)$. For a definition of $\partial_{\rho}^{\nicefrac{1}{2}}$ we refer to \cref{subsec:FracDer} or \cite[sec.~15.2]{Waurick2022}. We will not utilize this space for well-posedness theory.}
  \begin{align*}
    u &\in H^{1}_{\rho}(\mathbb{R}; H_{0})\cap H^{\nicefrac{1}{2}}_{\rho}\bigl(\mathbb{R}; \operatorname{dom}(C)\bigr)\text{,}\\
    v &\in H^{\nicefrac{1}{2}}_{\rho}(\mathbb{R}; H_{0}^{\perp})\cap L_{2,\rho}\bigl(\mathbb{R};\operatorname{dom}(C^{\ast})\bigr)\text{.}
  \end{align*}
\end{theorem}
We can proceed to investigate the intricacies of evolutionary equations with state-dependent delay.

\section{Preliminary Considerations}
\label{sec:Prelim}

\subsection{The delay operator}
\label{subsec:Delay}
We start with some observations regarding the delay operator. From here on, let $h>0$.
\begin{lemma}
  \label{th:EstimateForThetaT}
  Let $T>0$ and $t\in [0,T]$. The map
  \begin{align*}
    \Theta_{t}\colon \, L_{2,\rho}(-h,T;H)&\rightarrow L_{2}(-h,0;H)\text{,}\\
    u&\mapsto u_{(t)}
  \end{align*}
  satisfies $\norm{\Theta_{t}}\leq \e^{\rho t}$.
\end{lemma}
\begin{proof}
  For $u\in L_{2,\rho}(0,T;H)$:
  \begin{align*}
    \norm{\Theta_{t}u}_{L_{2}(-h,0;H)}^{2}
    &=\norm{u_{(t)}}_{L_{2}(-h,0;H)}^{2}
    = \medint\int_{-h}^0 \norm{u(t+s)}_{H}^{2}\dx[s]\\
    &=\medint\int_{t-h}^t \norm{u(s)}_{H}^{2}\dx[s]
    \leq \medint\int_{t-h}^t \norm{u(s)}_{H}^{2}\e^{(t-s)2\rho}\dx[s]\\
    &\leq \e^{2\rho t} \medint\int_{-h}^T \norm{u(s)}_{H}^{2}\e^{-2\rho s}\dx[s]
      = \e^{2\rho t}\norm{u}_{L_{2,\rho}(-h,T;H)}^{2}\text{.}\qedhere
  \end{align*}
\end{proof}
Investigating the time-dependence of the delay operator, we obtain the following result, that can already be found in \cite[Thm.~2.2]{Waurick2023} or \cite[Lem.~2.1]{Aigner2024}. 
\begin{lemma}
  \label{th:ThetaIsLipschitz}
  For $0<T\leq \infty$ let
  \begin{align*}
    \Theta\colon\,L_{2,\rho}(-h,T;H)&\rightarrow L_{2,\rho}\bigl(0,T;L_{2}(-h,0;H)\bigr)\text{,}\\
    u&\mapsto (t\mapsto \Theta_{t}u)\text{.}
  \end{align*}
  Then $\norm{\Theta}\leq \frac{1}{\sqrt{2\rho}}$.
\end{lemma}

\begin{remark}[$\Theta$ on $H^{1}_{\rho}\left(0,T;H\right)$]
  \label{rmk:ThetaIsLipschitz}
  We point out that the same estimate holds for
  \begin{align*}
    \tilde{\Theta}\colon\,H_{\rho}^{1}(-h,T;H)&\rightarrow H_{\rho}^{1}\bigl(0,T;L_{2}(-h,0;H)\bigr)\text{,}\\
    u&\mapsto (t\mapsto \Theta_{t}u)\text{.}
  \end{align*}
  This holds, since for $u\in H^{1}_{\rho}\left(-h,\infty;H\right)$ one has $\partial_{\rho}u_{(s)}=(\partial_{\rho}u)_{(s)}$, because $\partial_{\rho}u(\argdot + s)= (\partial_{\rho}u)(\argdot + s)$. Therefore, one can simply estimate:
  \begin{align*}
    \norm{\tilde{\Theta} u}_{H^{1}_{\rho}(0,T;L_{2}(-h,0;H))}^{2}
    &= \norm{\Theta u}_{L_{2,\rho}(0,T;L_{2}(-h,0;H))}^{2} + \norm{\partial_{\rho} \Theta u}_{L_{2,\rho}(0,T;L_{2}(-h,0;H))}^{2}\\
    &= \norm{\Theta u}_{L_{2,\rho}(0,T;L_{2}(-h,0;H))}^{2} + \norm{\Theta \partial_{\rho}u}_{L_{2,\rho}(0,T;L_{2}(-h,0;H))}^{2}\\
    &\leq \norm{\Theta}^{2} \norm{u}_{L_{2,\rho}(-h,T;H)}^{2} +  \norm{\Theta}^{2} \norm{\partial_{\rho} u}_{L_{2,\rho}(-h,T;H)}^{2}\\
    &= \norm{\Theta}^{2}\norm{u}_{H^{1}_{\rho}(-h,T;H)}^{2}\text{.}
  \end{align*}
  Similarly, one can show the same estimate for $\tilde{\Theta}$ as a map
  \begin{equation*}
    H_{\rho}^{1}(-h,T;H)\rightarrow L_{2,\rho}\bigl(0,T;H^{1}(-h,0;H)\bigr)\text{,} \quad
    u\mapsto (t\mapsto \Theta_{t}u)\text{.}
  \end{equation*}
\end{remark}
Furthermore, we observe that antiderivative and delay interchange:
\begin{proposition}
  \label{th:delayIntegralCommute}
  For $u \in L_{2,\rho}(-h,\infty;H)$ holds $\Theta_{t}I_{\rho}u = \smallint_{-h}^{\argdot}\Theta_{t}u$.\footnote{Here we understand $I_{\rho} \coloneq \smallint_{-\infty}^{\argdot}$.}
\end{proposition}
\begin{proof}
   For $s\in [-h,0]$, $t+s>0$ and $u\in L_{2,\rho}\left(-h,\infty;H\right)$ we compute:
   \begin{equation*}
     (I_{\rho}u)_{(t)}(s)= \medint\int_{-h}^{t+s} u(r)\dx[r]
        = \medint\int_{-h-t}^s u(t+r)\dx[r]
        = \medint\int_{-h}^s u_{(t)}(r)\dx[r]\text{.}\qedhere
   \end{equation*}
\end{proof}

\subsection{Auxiliary tools}
\label{subsec:AuxiliaryTools}
In order to prepare for \cref{sec:SolThy}, we need some assumptions --- in particular on the right-hand side of \cref{eq:EvolProblem}.

\subsubsection{Lipschitz-regularity}
\label{subsubsec:LipschitzRegularityOfRHS}
For the right-hand side $F$ of \cref{eq:EvolProblem}, we will need to demand some regularity property. Indeed, already for classical ODEs one cannot prove uniqueness of solutions for arbitrary continuous right-hand sides $F$ (i.e., in the case of Peano’s theorem). Motivated from the state-dependent ODE case (studied in \cite{Aigner2024,Waurick2023}), we use the following notion of Lipschitz-continuity:
\begin{definition}
  \label{def:AlmLC}
  A function $G\colon [0,\infty)\times L_{2,\rho}\left(-h,0;H\right) \to H$ is {\em almost uniformly Lipschitz-continuous} if it is continuous and for all $\alpha >0$ there exists some $L_{\alpha}>0$ such that
  \begin{equation*}
    \forall t\in [0,\infty)\,\forall \phi,\psi \in V_{\alpha}\colon \norm{G(t,\phi) - G(t,\psi)}_{H}\leq L_{\alpha}\norm{\phi - \psi}_{H^{1}(-h,0;H)}\text{,}
  \end{equation*}
  where we set $V_{\alpha}\coloneq \{u \in H^{1}(-h,0;H)\colon \norm{u'}_{\infty}\leq \alpha\}\subseteq L_{2}(-h,0;H)$.
\end{definition}
We follow up with an easy observation:
\begin{lemma}
  \label{th:V_alpha_convex}
  For $\alpha>0$, the space $V_{\alpha}\subseteq L_{2}(-h,0;H)$ is closed and convex. In particular, $V_{\alpha}$ is a closed and convex subspace of $H^{1}(-h,0;H)$.
\end{lemma}
\begin{proof}
  The convexity of $V_{\alpha}$ is evident. For closedness let $(u_{n})_{n}\in V_{\alpha}^{\N}$ with $u_{n}\to u$ in $L_{2}(-h,0;H)$ as $n\to \infty$ for some $u\in L_{2}(-h,0;H)$. We observe, that $(u_{n})_{n}$ is bounded in $H^{1}(-h,0;H)$:
  \begin{equation*}
    \norm{u_{n}'}_{2}^{2}= \medint\int_{-h}^{0} \norm{u_{n}'(t)}^{2}\dx[t] \leq \medint\int_{-h}^{0}\alpha^{2} \dx[t] = h\alpha^{2}\text{.}
  \end{equation*}
  Therefore, we obtain a weakly converging subsequence $(u_{n_{k}})_{k}\in H^{1}(-h,0;H)^{\N}$, converging necessarily to $u$, and hence $u\in H^{1}(-h,0;H)$. In addition, since $V_{\alpha}\subseteq H^{1}$ is convex, by Mazur's lemma: $\overline{V_{\alpha}}^{\!\!\phantom{.}_{\sigma (H^{1},(H^{1})')}}= \overline{V_{\alpha}}^{\!\!\phantom{.}_{H^{1}}}$. Since $V_{\alpha}$ is a closed subspace in $H^{1}(-h,0;H)$ by \cite[rem.~3.1]{Waurick2023}, $u\in V_{\alpha}$ follows. For the convenience of the reader, we repeat the short proof of closedness of $V_{\alpha}$ in $H^{1}$ here:\\
  Suppose $(u_{n})_{n} \in V_{\alpha}^{\N}$ with $u_{n}\to u$ in $H^{1}(-h,0;H)$. By the Fischer--Riesz theorem, $u_{n_{k}}'\to u'$ pointwise almost everywhere for a subsequence and thus $\norm{u'}_{\infty} \leq \alpha$.
\end{proof}

\subsubsection{Projection}
\label{subsubsec:Projection}
It will be critical to force the argument of $F$ from \cref{eq:EvolProblem} to be in $V_{\alpha}$ in order to apply almost uniform Lipschitz-continuity of $F$. Expanding upon the proof strategy from \cite[thm.~4.1]{Waurick2023}, we accomplish this by means of a projection:
\begin{definition}
  \label{def:MetricProjection}
  Let the metric projection on $V_{\alpha}$ be denoted by
  \begin{equation*}
    \pi_{\alpha}\colon \begin{Bmatrix}L_{2} \\ H^{1}\end{Bmatrix}
    \to V_{\alpha}\text{,}\qquad v\mapsto
    \begin{cases*}
      v &for $v\in V_{\alpha}$\text{,}\\
      \operatorname{arg min}_{u\in V_{\alpha}}\norm{u-v}_{L_{2}\,\text{or}\,H^{1}} &for $v\notin V_{\alpha}\text{.}$
    \end{cases*}
  \end{equation*}
\end{definition}
Since $V_{\alpha}$ is a closed convex subspace and because $H^{1}$ and $L_{2}$ are Hilbert spaces, the metric projection exists. In either case, the metric projection is Lipschitz-continuous with constant $1$ (cf.~\cite[thm.~V.3.2 \& ex.~V.6.15]{Werner2011}).

\subsubsection{Regularity preservation}
\label{subsubsec:RegularityPreservationOfF}
For technical reasons, we will have to impose the condition that the right-hand side of \cref{eq:EvolProblem} in some sense increases the regularity of the unknown $u$. We will achieve this by supposing that $F$ is of a specific form and contains a nonlinearity that preserves regularity.
\begin{definition}
  \label{def:RegPres}
  We call $F$ {\em regularity preserving} if for $u \in H^{1}_{\rho}$ with $\norm{u'}_{\infty}<\infty$, the function $t\mapsto F\bigl(t, u_{(t)}\bigr)$ defines an element of $H^{1}_{\rho}(0,\infty;H)$.
\end{definition}
\begin{example}
  Examples of such regularity-preserving functions include
  \begin{itemize}[leftmargin=4ex]
    \item autonomous functions $F\bigl(u_{(t)}\bigr) =\sigma\bigl(u(t)\bigr)$ for any (globally) Lipschitz-continuous function $\sigma\colon H\to H$:\\[1cm]
          For $t,s \geq 0$ we can estimate:
          \begin{align*}
            \norm{F(t,u_{(t)}) - F(s,u_{s})} &= \norm[\big]{\sigma\bigl(u(t)) - \sigma(u(s)\bigr)}\\
                                             &\leq L_{\sigma}\norm{u(t)-u(s)}\\
                                             &\leq L_{\sigma}\norm{u'}_{\infty}\abs{t-s}\text{,}
          \end{align*}
          where $L_{\sigma}$ denotes the Lipschitz-constant of $\sigma$. In particular, $t\mapsto F\bigl(t,u(t)\bigr)$ is differentiable almost everywhere and we can bound the derivative by the $L_{2,\rho}(0,\infty;H)$-integrable constant $L_{\sigma}\norm{u'}_{\infty}$.
    \item the ``standard'' example of state-dependent delay, $F\bigl(t,u_{(t)}\bigr) = u\bigl(t-\tau(u(t))\bigr)$ for any (globally) Lipschitz-continuous $\tau\colon H \to [0,h]$:\\
          Because $u' \in L_{\infty}$, $u$ is Lipschitz-continuous and we can calculate for $t,s \geq 0$:
          \begin{align*}
            \norm[\big]{F(t,u_{(t)}) - F(s,u_{(s)})}  &= \norm[\big]{u\bigl(t-\tau(u(t))\bigr) - u\bigl(s-\tau(u(s))\bigr)}\\
                                                &\leq \norm{u'}_{\infty}\abs{t-\tau(u(t)) - s+\tau(u(s))}\\
                                                &\leq \norm{u'}_{\infty}\bigl[\abs{t-s} + L_{\tau}\norm{u(t)-u(s)}\bigr]\\
                                                &\leq \norm{u'}_{\infty}\bigl[1 + L_{\tau}\norm{u'}_{\infty}\bigr]\abs{t-s}\text{,}
          \end{align*}
          where $L_{\tau}$ denotes the Lipschitz-constant of $\tau$. In particular, $t\mapsto F\bigl(t, u_{(t)}\bigr)$ is differentiable almost everywhere and we can bound the derivative by the constant $\norm{u'}_{\infty}\bigl[1 + L_{\tau}\norm{u'}_{\infty}\bigr]$, which is an $L_{2,\rho}(0,\infty;H)$-integrable function.
  \end{itemize}
\end{example}
The next result provides a more abstract class of examples.
\begin{example}
  \label{th:RegularityPreservation}
  Let $F\colon [0,\infty)\times L_{2}(-h,0;H) \to H$ be a function that is
  \begin{itemize}[leftmargin=4ex]
    \item an almost uniformly Lipschitz-continuous map and
    \item satisfies $\forall \alpha>0\exists L>0\forall\varphi\in V_{\alpha}\forall s,t\in [0,\infty)$:
          \begin{equation}
            \norm{F(t,\varphi)-F(s,\varphi)}_{H}\leq L\abs{t-s}\text{.}\tag{$\triangle$}
          \end{equation}
  \end{itemize}
  Then, the function $t \mapsto F\bigl(t, u_{(t)}\bigr)$ is an element of $H^{1}_{\rho}(0,\infty;H)$ for any $u\in H^{1}_{\rho}(-h,\infty;H)$ with bounded derivative.
\end{example}
\begin{proof}
  We let $0\leq s \leq t$ and calculate for $u \in H^{1}_{\rho}(-h,\infty;H)$ with $\norm{u'}_{\infty}\leq \alpha$:
  \begin{align*}
    \norm[\big]{F\bigl(t,u_{(t)}\bigr) - F\bigl(s,u_{(s)}\bigr)}_{H}
    &\leq \norm[\big]{F\bigl(t,u_{(t)}\bigr) - F\bigl(t,u_{(s)}\bigr)} + \norm[\big]{F\bigl(t, u_{(s)}\bigr) - F\bigl(s,u_{(s)}\bigr)}\\
    &\leq L_{\alpha}^{(1)}\norm{u_{(s)}-u_{(t)}}_{H^{1}(-h,0;H)} + L_{\alpha}^{(2)}\abs{t-s}\text{,}
  \end{align*}
  where $L_{\alpha}^{(1)}$ denotes the Lipschitz-constant w.r.t.~to the first variable provided by almost uniform Lipschitz-continuity and $L_{\alpha}^{(2)}$ w.r.t.~the second variable provided by ($\triangle$). We can further estimate the first summand in the following way:
  \begin{align*}
    \norm{u_{(s)}-u_{(t)}}_{L_{2}(-h,0;H)}^{2}
    &= \medint\int_{-h}^{0}\norm{u(s + r) - u(t + r)}^{2}_{H} \dx[r]\\
    &= \medint\int_{s-h}^{s} \norm{u(t-s+x)-u(x)}^{2}_{H} \dx[x]\\
    &= \medint\int_{s-h}^{s} \norm[\Big]{\medint\int_{0}^{t-s} u'(y) \dx[y] }^{2}_{H} \dx[x]\\
    &\leq \medint\int_{s-h}^{s} \abs{t-s}^{2} \alpha^{2} \dx[x]
      = \abs{t-s}^{2}\alpha^{2}h\text{.}
  \end{align*}
  Hence, the Lipschitz-seminorm is bounded (almost everywhere) and the claim follows.
\end{proof}

\subsubsection{Integrability}
\label{subsubsec:IntegrabilityOfF}
Since the theory of evolutionary equations establishes well-posedness of a problem by applying the Picard operator $S_{\rho}$ to the right-hand side of a problem, the associated solution theory requires an input on the entire (half-)line. Hence, we will also have to demand an integrability condition. It turns out that the following condition is sufficient:
\begin{lemma}
  \label{th:integrability}
  Let $v\in H^{1}_{\rho}(-h,\infty;H)$ and $F \colon [0,\infty)\times L_{2}(-h,0;H) \to H$ be almost uniformly Lipschitz-continuous. Suppose $F(\argdot,0)\in L_{2,\rho}(0,\infty;H)$. Then for all $\alpha>0$ holds $F\bigl(\argdot, \pi_{\alpha}(v_{(\argdot)})\bigr)\in L_{2,\rho}(0,\infty;H)$.
\end{lemma}
\begin{proof}
  We can simply estimate, using the simple inequality $(a+b)^{2}\leq 2(a^{2}+b^{2})$:
  \begin{align*}
    &\norm[\big]{F\bigl(\argdot,\pi_{\alpha}(v_{(\argdot)})\bigr)}_{L_{2,\rho}(0,\infty;H)}^{2}
      = \medint\int_0^{\infty} \norm[\big]{F\bigl(t,\pi_{\alpha}(v_{(t)})\bigr)}_{H}^{2}\e^{-2\rho t}\dx[t]\\
    &= \medint\int_0^{\infty} \norm[\big]{F(t,0)-F(t,0)+F\bigl(t,\pi_{\alpha}(v_{(t)})\bigr)}_{H}^{2}\e^{-2\rho t}\dx[t]\\
    &\leq \medint\int_0^{\infty} \bigl[2\norm{F(t,0)}_{H}^{2}
      + 2L_{\alpha} \norm{\pi_{\alpha}(0)-\pi_{\alpha}(v_{(t)})}_{H^{1}(-h,0;H)}^{2}\bigr]\e^{-2\rho t}\dx[t]\\
    &\leq \medint\int_0^{\infty} \bigl[2\norm{F(t,0)}_{H}^{2}
      + 2L_{\alpha} \norm{\underbrace{v_{(t)}}_{= \Theta_{t}(v)}}_{H^{1}(-h,0;H)}^{2}\bigr]\e^{-2\rho t}\dx[t]\text{.}
  \end{align*}
  All those integrals are finite, as our estimate for $\Theta$ from \cref{th:ThetaIsLipschitz} shows.
\end{proof}

\subsection{Reformulation of the problem}
\label{subsec:Reformulation}
Let $A$ and $M$ satisfy the assumptions of \cref{th:Picard}. As discussed in \cref{subsec:DistribSpaces_IVPs}, the problem we want to investigate is
\begin{equation}
  \tag{\ref{eq:EvolProblem}}
  \bigl(\overline{\partial_{\rho}M(\partial_{\rho}) + A}\bigr)v = F\bigl(\argdot, (v+Z)_{(\argdot)}\bigr) + f \qquad \text{on}\ L_{2,\rho}(\mathbb{R};H)\text{,}
\end{equation}
for some $f\in L_{2,\rho}(\mathbb{R};H)$ with support on $[0,\infty)$ and a continuation $Z$ of the prehistory $\Phi$. The closure of the operator on the left-hand side allows us to multiply this equation with the Picard operator $S_{\rho}$ (cf.~\cref{th:Picard}). This way, we obtain the (weak) formulation
\begin{equation}
  \label{eq:GenFPP}
  v = S_{\rho}\bigl[f + F\bigl(\argdot, (v+Z)_{(\argdot)}\bigr)\bigr]\qquad \text{in}\ L_{2,\rho} (\mathbb{R};H)\text{.}
\end{equation}
This version of the problem lends itself to a reformulation as a fixed point problem (FPP) on $L_{2,\rho}(0,\infty;H)$ via the mapping
\begin{equation}
  \label{eq:FPPinL2}
  L_{2,\rho}(0,\infty;H) \ni v \mapsto S_{\rho}f + S_{\rho}F\bigl(\argdot,(v+Z)_{(\argdot)}\bigr)\text{,}
\end{equation}
provided that the right-hand side of \cref{eq:FPPinL2} is regular enough to produce a function in $L_{2,\rho}(0,\infty;H)$ again. Note that the solution of the FPP is a weak solution in the sense of \cref{def:weakSol}.

\section{Well-posedness}
\label{sec:SolThy}

Throughout this section, let $A\colon H\supseteq \operatorname{dom}(A)\to H$ be an $\mathrm{m}$-accretive linear operator and let $M\colon \C \supseteq \operatorname{dom}(M)\to \C$ be a linear material law, both of which together satisfy the assumptions of \cref{th:Picard} for all $\rho>0$ large enough. Then, utilizing the metric projection $\pi_{\alpha}$ from \cref{def:MetricProjection}, we transition from \eqref{eq:FPPinL2} to
\begin{align}
  \begin{aligned}
    \label{eq:FPPproj}
    \Gamma_{\rho,\alpha}\colon \, L_{2,\rho}(0,\infty;H)&\rightarrow L_{2,\rho}(0,\infty;H)\text{,}\\
    v &\mapsto S_{\rho}f + S_{\rho}F\bigl(\argdot,\pi_{\alpha}((v+Z)_{(\argdot)})\bigr)\text{,}
  \end{aligned}
\end{align}
where we will assume that $F$ satisfies $F(\argdot,0)\in L_{2,\rho}$ and the function on the right is therefore in $L_{2,\rho}(0,\infty;H)$ by \cref{th:integrability}.
The tricky part of our well-posedness theory will be to remove the projection $\pi_{\alpha}$ again (at least locally in time). Suitable regularity assumptions on $f$ will prove necessary for that as well. Before that, we introduce our notion of solution.
\begin{definition}
  A {\em local} solution $v \in L_{2,\rho}(0,\infty;H)$ of \cref{eq:EvolProblem} is a fixed point of $\Gamma_{\rho,\alpha}$ for some $\alpha>0$, that satisfies $\pi_{\alpha}\bigl((v+Z)_{(t)}\bigr)=(v+Z)_{(t)}$ for $t\in [0,T]$ up to some positive $T>0$. A {\em global} solution of \cref{eq:EvolProblem} is a fixed point of \cref{eq:FPPinL2}.
\end{definition}
\begin{remark}[On the notion of solution]
Let $v$ be a local solution of \cref{eq:EvolProblem}, i.e., $v$ solves \cref{eq:GenFPP} up to some positive time $T>0$. Then
\begin{align*}
  \chi_{\mathbb{R}_{\leq T}}v
  &= \chi_{\mathbb{R}_{\leq T}} S_{\rho}\bigl[f + F\bigl(\argdot,\pi_{\alpha}((v+Z)_{(\argdot)})\bigr)\bigr]\\
  &= \chi_{\mathbb{R}_{\leq T}} S_{\rho} \chi_{\mathbb{R}_{\leq T}}\bigl[f + F\bigl(\argdot,\pi_{\alpha}((v+Z)_{(\argdot)})\bigr)\bigr]\\
  &= \chi_{\mathbb{R}_{\leq T}} S_{\rho}\bigl[\chi_{\mathbb{R}_{\leq T}}f + \chi_{\mathbb{R}_{\leq T}}F\bigl(\argdot,\pi_{\alpha}((v+Z)_{(\argdot)})\bigr)\bigr]\\
  &= \chi_{\mathbb{R}_{\leq T}} S_{\rho}\bigl[\chi_{\mathbb{R}_{\leq T}}f + \chi_{\mathbb{R}_{\leq T}}F\bigl(\argdot,(v+Z)_{(\argdot)}\bigr)\bigr]\\
  &= \chi_{\mathbb{R}_{\leq T}} S_{\rho}\bigl[\chi_{\mathbb{R}_{\leq T}}f + \chi_{\mathbb{R}_{\leq T}}F\bigl(\argdot,(\chi_{\mathbb{R}_{\leq T}}v+Z)_{(\argdot)}\bigr)\bigr]\text{,}
\end{align*}
where we used causality of the solution operator $S_{\rho}$, appealing to \cref{th:Picard}. In that sense, a local solution $v$ of \cref{eq:EvolProblem} solves the FPP
\begin{equation}
  \label{eq:localEq}
  L_{2,\rho}(0,T;H) \ni w \mapsto S_{\rho}\bigl[\chi_{\mathbb{R}_{\leq T}}f + \chi_{\mathbb{R}_{\leq T}}F\bigl(\argdot,(w+Z)_{(\argdot)}\bigr)\bigr]\in L_{2,\rho}(0,T;H)
\end{equation}
for some $T>0$. A global solution in particular solves the FPP above for all $T>0$.
\end{remark}
To obtain a local solution of \cref{eq:EvolProblem} and to not simply find a fixed point of \cref{eq:FPPproj} (which is necessarily defined on the entirety of $L_{2,\rho}(0,\infty;H)$), we will in addition have to impose a more specific shape of $F$. In order to do that, we introduce the short notation
\begin{align*}
  \widetilde{F}(v)(t) &\coloneq F\bigl(t, (v + Z)_{(t)}\bigr)\text{,}\\
  \widetilde{F}_{\alpha}(v)(t) &\coloneq F\bigl(t, \pi_{\alpha}\bigl((v + Z)_{(t)}\bigr)\bigr)\text{.}
\end{align*}
Note that the prehistory is integrated into the functions $\widetilde{F}$ and $\widetilde{F}_{\alpha}$. We confine to the following two cases (which are relevant in applications, cf.~\cref{sec:Examples}):

\subsubsection*{The case $\widetilde{F}(v)=I_{\rho}\widetilde{G}(v)$}\phantom{.}\\
We assume $F\bigl(t,(v + Z)_{(t)}\bigr)=I_{\rho}G\bigl(\argdot,(v+Z)_{(\argdot)}\bigr)(t)$, which gives rise to the FPP
\begin{equation}
  \label{eq:FPPIntOut}
  v = S_{\rho}\bigl[f + I_{\rho}G\bigl(\argdot,(v+Z)_{(\argdot)}\bigr)\bigr]
  \qquad\text{in}\,L_{2,\rho}(0,\infty;H)\text{.}
\end{equation}

\subsubsection*{The case $\widetilde{F}(v)=\widetilde{G}(I_{\rho}v)$}\phantom{.}\\
We assume $F\bigl(t,(v + Z)_{(t)}\bigr)=G\bigl(t,(I_{\rho}v + Z)_{(t)}\bigr)$, which gives rise to the FPP
\begin{equation}
  \label{eq:FPPIntIns}
  v = S_{\rho}\bigl[f + G\bigl(\argdot,(I_{\rho}v+Z)_{(\argdot)}\bigr)\bigr]
  \qquad\text{in}\,L_{2,\rho}(0,\infty;H)\text{.}
\end{equation}
In either case, we again point out that the right-hand side has support only on $[0,\infty)$, because $f$ and $G$ have and the solution operator $S_{\rho}$ is causal.

\subsection{Local Existence}
Now we can state local existence and uniqueness results for both of these formulations. We will make use of some recurring assumptions:
\begin{assumptions}
  \label{assumptions}
  \phantom{.}
  \begin{enumerate}[label=(\Alph*), leftmargin=5ex]
    \item \label{ass:A} Let $G\colon [0,\infty)\times L^{2}(-h,0;H)\to H$
          \begin{enumerate}[label=\roman*), leftmargin=5ex]
            \item be almost uniformly Lipschitz-continuous (cf.~\cref{def:AlmLC}),
            \item regularity preserving (cf.~\cref{def:RegPres})
            \item and satisfy $G(\argdot,0)\in L_{2,\rho}(0,\infty;H)$.
          \end{enumerate}
    \item \label{ass:B} Let $\Phi \in H^{1}(-h,0;H)$ with $\norm{\Phi'}_{\infty}<\infty$ and let $Z$ be an $H^{1}_{\rho}(-h,\infty;H)$-extension of $\Phi$ satisfying
          \begin{enumerate}[label=\roman*), leftmargin=5ex]
            \item $Z\vert_{(-h,0]}=\Phi$ and
            \item $Z\vert_{[0,\infty)}\in \mathcal{C}^{1}(0,\infty;H)$ with $\norm{Z'(0)}_{H}\leq \norm{\Phi'}_{\infty}$.
          \end{enumerate}
  \end{enumerate}
\end{assumptions}

\begin{theorem}[local existence and uniqueness for \cref{eq:FPPIntOut}]
  \label{th:localExistOut}
  Let $G$ satisfy assumption \ref{ass:A} and let $\Phi$ and $Z$ satisfy assumption \ref{ass:B}. Let $f \in H_{\tilde{\rho}}^{2}(0,\infty;H)\cap H^{1}_{0,\tilde{\rho}}(0,\infty;H)$ and $G(0,\Phi) = - f'(0)$ for some $\tilde{\rho} >0$. Then \cref{eq:FPPIntOut} has a unique local solution $u\in H^{2}(0,T;H)$.
\end{theorem}
The last condition can be viewed as a consistency condition, which is reminiscent of results for ODEs with state-dependent delay, cf. e.g., \cite{Walther2003}. Consistency conditions of this form will prove to be critical for our techniques throughout the article.
\begin{proof}
  We let $\alpha > \norm{\Phi'}_{\infty}$ and $\rho > \max\bigl\{\tfrac{L_{\alpha}^{2}}{2 c^{2}},\tilde{\rho}\bigr\}$, where $L_{\alpha}$ is the Lipschitz-constant of $G$ associated to the parameter $\alpha$, cf.~\cref{def:AlmLC}, and $c$ is from the bound of $\norm{S_{\tilde{\rho}}}$, cf.~\cref{th:Picard}.\\
  We first show the existence of a fixed point of
  \begin{equation*}
    v = S_{\rho}\bigl[f + I_{\rho}G\bigl(\argdot,\pi_{\alpha}((v+Z)_{(\argdot)})\bigr)\bigr] = S_{\rho}\bigl[f + I_{\rho}\widetilde{G}_{\alpha}(v)\bigr]\eqcolon \Gamma_{\rho,\alpha}(v)\text{,}\quad v\in H^{1}_{0,\rho}(0,\infty;H)\text{,}
  \end{equation*}
  by showing that $\Gamma_{\rho,\alpha}$ satisfies the requirements of the contraction mapping principle:
\begin{itemize}[leftmargin=3ex]
  \item $\Gamma_{\rho,\alpha}$ is a self-mapping:\\
        Since $\pi_{\alpha}$ maps into $V_{\alpha}$ and $G$ satisfies assumption \ref{ass:A} iii), \cref{th:integrability} ensures that $\widetilde{G}_{\alpha}(v)$ defines an element of $L_{2,\rho}(0,\infty;H)$. As $I_{\rho}\colon L_{2,\rho}\to H^{1}_{\rho}$, we have $I_{\rho}\widetilde{G}_{\alpha}(v) \in H^{1}_{\rho}(0,\infty;H)$. By \cref{th:Commutation}, the expression $\Gamma_{\rho,\alpha}(v)$ defines an element of $H^{1}_{\rho}(0,\infty;H)$. Furthermore, because $f\in H^{1}_{0,\rho}(0,\infty;H)$ we have
        \begin{equation*}
          0 = \bigl[f + I_{\rho}\widetilde{G}_{\alpha}(v)\bigr](0)\text{,}
        \end{equation*}
        which assures that $t\mapsto \bigl[f + I_{\rho}\widetilde{G}_{\alpha}(v)\bigr]$ defines an element of $H^{1}_{0,\rho}(0,\infty;H)$ by \cref{th:DomDerivativeFor0} and appealing to causality of the solution operator $S_{\rho}$ we infer $\Gamma_{\rho,\alpha}(v)\in H^{1}_{0,\rho}(0,\infty;H)$.
  \item $\Gamma_{\rho,\alpha}$ is Lipschitz-continuous:\\
        Taking the norm $\norm{\partial_{\rho}\argdot}_{L_{2,\rho}}$ as the norm on $H^{1}_{0,\rho}(0,\infty;H)$, we can calculate for $v,w \in H^{1}_{0,\rho}(0,\infty;H)$:
        \begin{align*}
          &\norm{\Gamma_{\rho,\alpha}v-\Gamma_{\rho,\alpha}w}_{H^{1}_{0,\rho}(0,\infty;H)}^{2}\\
          &\quad=\norm[\big]{\mathring{\partial}_{\rho}S_{\rho}I_{\rho}\widetilde{G}_{\alpha}(v)
            - \mathring{\partial}_{\rho}S_{\rho}I_{\rho}\widetilde{G}_{\alpha}(w)}_{L_{2,\rho}(0,\infty;H)}^{2}\\
          \intertext{Using the fact that antiderivative and $S_{\rho}$ commute (cf.~\cref{th:Commutation}), this expression simplifies to}
          &\quad= \norm[\big]{S_{\rho}\widetilde{G}_{\alpha}(v)
            - S_{\rho}\widetilde{G}_{\alpha}(w)}_{L_{2,\rho}(0,\infty;H)}^{2}\\
          &\quad\leq \norm{S_{\rho}}^{2}\norm[\big]{\widetilde{G}_{\alpha}(v)
            - \widetilde{G}_{\alpha}(w)}_{L_{2,\rho}(0,\infty;H)}^{2}\\
          \intertext{Using almost uniform Lipschitz-continuity of $G$ (note that $\alpha > \norm{\Phi'}_{\infty}$) and the Lipschitz-continuity of the projection $\pi_{\alpha}$, we can estimate further:}
          &\quad\leq \tfrac{1}{c^{2}}\medint\int_0^\infty \norm[\big]{G\bigl(t,\pi_{\alpha}((v+Z)_{(t)})\bigr)
            -G\bigl(t,\pi_{\alpha}((w+Z)_{(t)})\bigr)}_{H}^{2}\e^{-2\rho t}\dx[t]\\
          &\quad\leq \tfrac{L_{\alpha}^{2}}{c^{2}}\medint\int_0^\infty \norm[\big]{\pi_{\alpha}\bigl((v+Z)_{(t)}\bigr)
            -\pi_{\alpha}\bigl((w+Z)_{(t)}\bigr)}_{H^{1}(-h,0;H)}^{2}\e^{-2\rho t}\dx[t]\\
          &\quad\leq \tfrac{L_{\alpha}^{2}}{c^{2}}\medint\int_{0}^{\infty}\norm[\big]{(v+Z)_{(t)}
            -(w+Z)_{(t)}}_{H^{1}(-h,0;H)}^{2}\e^{-2\rho t}\dx[t]\\
          &\quad= \tfrac{L_{\alpha}^{2}}{c^{2}}\medint\int_0^\infty \norm{v_{(t)}
            -w_{(t)}}_{H^{1}(-h,0;H)}^{2}\e^{-2\rho t}\dx[t]\\
          &\quad=\tfrac{L_{\alpha}^{2}}{c^{2}}\norm{\Theta v
            - \Theta w}_{L_{2,\rho}(0,\infty;H^{1}(-h,0;H))}^{2}\\
          \intertext{An application of the Lipschitz-continuity of the delay operator $\Theta$ in the form of \cref{rmk:ThetaIsLipschitz} yields the desired result:}
          &\quad\leq \tfrac{L_{\alpha}^{2}}{c^{2}} \tfrac{1}{2\rho}
            \norm{v-w}_{H_{\rho}^{1}(-h,\infty;H)}^{2}\text{.}
        \end{align*}
  \item $\Gamma_{\rho,\alpha}$ is a contraction:\\
        This follows from the last estimate and the assumption $\rho > \tfrac{L_{\alpha}^{2}}{2 c^{2}}$.
\end{itemize}
The contraction mapping principle hence assures a solution $w$ of
\begin{equation*}
  w = \Gamma_{\rho,\alpha}(w)\text{,} \quad w\in H^{1}_{0,\rho}(0,\infty;H)\text{.}
\end{equation*}
To obtain a local solution of \cref{eq:FPPIntOut}, we need to argue that at least up to some $T>0$, we have $\pi_{\alpha}\bigl((w + Z)_{(\argdot)}\bigr) = (w + Z)_{(\argdot)}$.\\
We first observe that under the assumption $f\in H^{2}_{\rho}\left(0,\infty;H\right)$, we can write $f(t) = f(0) + \int_{0}^{t}f'(s)\dx[s]$ with $f'\in H^{1}_{\rho}\left(0,\infty;H\right)$. Using $f(0)=0$, we obtain $f = I_{\rho}f'$ and therefore by \cref{th:Commutation}:
\begin{equation}
  \label{eq:aux}
  \mathring{\partial}_{\rho}w = S_{\rho}\bigl[f' + G\bigl(\argdot, \pi_{\alpha}((w+Z)_{(\argdot)})\bigr)\bigr]\text{.}
\end{equation}
Under the assumption $f'(0) = -G(0,\Phi)$, the function $t\mapsto f'(t) + G\bigl(t,\pi_{\alpha}((u + Z)_{(t)})\bigr)$ is continuous as composition of continuous functions (in particular because the delay operator $\Theta$ is continuous in time, cf.~\cite[prop.~2.4]{Waurick2023}) and takes the value $0$ at $0$, since
\begin{align*}
  \lim_{t\downarrow 0} \bigl[f'(t) + G\bigl(t,\pi_{\alpha}((u + Z)_{(t)})\bigr)\bigr]
  &= f'(0) + G\bigl(0, \lim_{t\downarrow 0}[\pi_{\alpha}((u + Z)_{(t)})]\bigr)\\[-1.2ex]
  &= f'(0) + G\bigl(0, (u+Z)_{(0)}\bigr)\\
  &= f'(0) + G(0,\Phi)
  = 0\text{.}
\end{align*}
Because $\pi_{\alpha}$ projects onto $H^{1}$ and $G$ is regularity-preserving appealing to assumption \ref{ass:A} ii), $t\mapsto G\bigl(\argdot, \pi_{\alpha}((w + Z)_{(\argdot)})\bigr)$ is an $H^{1}_{\rho}$-function. Hence, the right-hand side of \cref{eq:aux} is in $H^{1}_{0,\rho}(0,\infty;H)$. By virtue of \cref{th:Commutation} we obtain that $w' \in H^{1}_{0,\rho}(0,\infty;H)$.\\
Since $w = 0$ on $(-\infty,0]$ and $\Phi \in V_{\alpha}$, we only need to show that $w' + Z'$ is bounded in $\norm{\argdot}_{\infty}$-norm by $\alpha$ up to some positive time $T>0$. Since $Z'$ is bounded on any interval $[0,c]$, $c>0$, and $\norm{Z'(0)}_{H}< \alpha$ by assumption \ref{ass:B} ii), $\norm{Z'}_{\mathcal{C}(0,\tilde{T};H)}<\alpha$ holds up to some positive $\tilde{T}>0$. Since $w' \in H^{1}_{0,\rho}(0,\infty;H)$, $w'$ is continuous appealing to the embedding  \cref{th:Sobolev} and $w'(0)=0$ by definition of $H^{1}_{0,\rho}$. Hence, up to some positive $0<T\leq \tilde{T}$, $w'+Z'$ remains bounded in $\norm{\argdot}_{\infty}$-norm by $\alpha$. This establishes existence of a local solution.\\
For uniqueness, we let $v \in L_{2,\rho}(0,\infty;H)$ be a local solution of \cref{eq:FPPIntOut}, that is by definition, a fixed point of $\Gamma_{\hat{\rho},\alpha}$ from \eqref{eq:FPPproj} for some $\hat{\rho}>0$, i.e.,
\begin{align*}
  \begin{aligned}
    \Gamma_{\hat{\rho},\alpha}\colon \, L_{2,\hat{\rho}}(0,\infty;H)&\rightarrow L_{2,\hat{\rho}}(0,\infty;H)\text{,}\\
    v &\mapsto S_{\hat{\rho}}f + S_{\hat{\rho}}I_{\hat{\rho}}\widetilde{G}_{\alpha}(v)\text{.}
  \end{aligned}
\end{align*}
As proven at the start of the proof, a fixed point $v$ is an element of $H^{1}_{0,\hat{\rho}}(0,\infty;H)$ by virtue of the shape of the right-hand side. Uniqueness of the fixed point now follows from the contraction mapping principle and eventual independence of the weight parameters $\hat{\rho}$ and $\rho$, appealing to \cref{th:Picard}.
\end{proof}

\begin{remark}[maximal existence interval]
  \label{rmk:maxExistIntOut}
  A closer inspection of the proof reveals that the unique local solution $u = v + Z$ of \cref{eq:FPPIntOut} exists up to $T>0$ at which $\lim_{t\uparrow T}\norm{u}_{\mathcal{C}(0,t;H)}=+\infty$:\\
  Indeed, in the proof of \cref{th:localExistOut}, the solution of the FPP containing $\pi_{\alpha}$
  \begin{equation*}
    v = \Gamma_{\rho, \alpha}(v) = S_{\rho}\bigl[f + I_{\rho}G\bigl(\argdot, \pi_{\alpha}\bigl((v+Z)_{(\argdot)}\bigr)\bigr)\bigr]\text{,} \qquad v\in H^{1}_{0,\rho}(0,\infty;H)
  \end{equation*}
  is a local solution of \cref{eq:FPPIntOut} for as long as the $\norm{\argdot}_{\infty}$-norm of $v' + Z'$ remains bounded by $\alpha$, say up to $T_{0}$. Increasing the parameter $\alpha$ to $\alpha + 1$ yields a solution $w$ of $w = \Gamma_{\rho, \alpha +1}(w)$, such that $w'+Z'$ remains bounded by $\alpha +1$ up to $T_{1}>T_{0}$, by virtue of the continuity of $w'+Z'$ (cf. proof of \cref{th:localExistOut}). We observe that the solutions $v$ (corresponding to $\alpha$) and $w$ (corresponding to $\alpha + 1$) agree up to $T_{0}$, by virtue of unique solvability. The claim follows by iteration and the fact, that $\norm{Z'}_{\mathcal{C}(0,t;H)}<\infty$ for all $t>0$.
\end{remark}

\begin{theorem}[local existence and uniqueness for \cref{eq:FPPIntIns}]
  \label{th:localExistIns}
  Let $G$ satisfy assumption \ref{ass:A} and let $\Phi$ and $Z$ satisfy  assumption \ref{ass:B}. Let $f\in H^{1}_{\tilde{\rho}}(0,\infty;H)$ for some $\tilde{\rho}>0$ and $f(0) = -G\bigl(0, \Phi\bigr)$. Then \cref{eq:FPPIntIns} has a unique local solution $u \in H^{1}(0,T;H)$.
\end{theorem}
The proof of this theorem follows in large parts the proof of \cref{th:localExistOut}, but there are some technical differences in how to remove the projection $\pi_{\alpha}$. For the convenience of the reader, we present the entire argument again.
\begin{proof}
  Let $\alpha > \norm{\Phi'}_{\infty}$ and $\rho > \max \bigl\{1, \tilde{\rho}, \tfrac{2 c^{2}}{L_{\alpha}^{2}(h^{2}+1)}\bigr\}$, where $L_{\alpha}$ is the Lipschitz-constant of $G$ associated to the parameter $\alpha$, cf.~\cref{def:AlmLC}, and $c$ is from the bound for $\norm{S_{\tilde{\rho}}}$, cf.~\cref{th:Picard}. We start with a preliminary observation:\\
  Let $-\infty < a < b < \infty$ and $f \in L_{2}(a,b;H)$. Then we can estimate:
  \begin{align}
    \begin{aligned}
      \medint\int_{a}^{b} \norm[\Big]{\medint\int_{a}^{t}f(s)\dx[s]}^{2}\dx[t]
      &\leq \medint\int_{a}^{b} (t-a) \medint\int_{a}^{t}\norm{f(s)}^{2}\dx[t]
      = \medint\int_{a}^{b} \medint\int_{s}^{b}\underbrace{(t-a)}_{\leq b-a} \norm{f(s)}^{2}\dx[t]\dx[s]\\[-1.5ex]
      &\leq (b-a)^{2}\medint\int_{a}^{b}\norm{f(s)}^{2}\dx[s]\text{.}
    \end{aligned} \tag{$\square$}
  \end{align}
  We begin the actual proof by first showing the existence of a fixed point of
  \begin{equation*}
    v = S_{\rho}\bigl[f + G\bigl(\argdot,\pi_{\alpha}\bigl((I_{\rho}v+Z)_{(\argdot)}\bigr)\bigr)\bigr] = S_{\rho}\bigl[f+ \widetilde{G}_{\alpha}(I_{\rho}v)\bigr]\eqcolon \Gamma_{\rho,\alpha}(v)
    \quad v\in L_{2,\rho}(0,\infty;H)\text{,}
  \end{equation*}
  by verifying that $\Gamma_{\rho,\alpha}$ satisfies the requirements of the contraction mapping principle:
\begin{itemize}[leftmargin=3ex]
  \item $\Gamma_{\rho,\alpha}$ is a self-mapping:\\
        This follows from \cref{th:integrability}, the assumptions on $f$ and $G$ and that the Picard operator is a bounded linear operator on $L_{2,\rho}$ (cf.~\cref{th:Picard}).
  \item $\Gamma_{\rho,\alpha}$ is Lipschitz-continuous:\\
        For $v,w \in L_{2,\rho}(0,\infty;H)$ we can compute:
        \begin{align*}
          &\norm{\Gamma_{\rho,\alpha}v-\Gamma_{\rho,\alpha}w}_{L_{2,\rho}(0,\infty;H)}^{2}\\
          &\quad= \norm{S_{\rho}\widetilde{G}_{\alpha}(I_{\rho}v) -S_{\rho}\widetilde{G}_{\alpha}(I_{\rho}w)}_{L_{2,\rho}(0,\infty;H)}^{2}\\
          &\quad\leq \norm{S_{\rho}}^{2} \medint\int_0^\infty \norm{G\bigl(t,\pi_{\alpha}\bigl((I_{\rho}v+Z)_{(t)}\bigr)\bigr) -G\bigl(t,\pi_{\alpha}\bigl((I_{\rho}w+Z)_{(t)}\bigr)\bigr)}_{H}^{2}\e^{-2\rho t}\dx[t]\\
          \intertext{Using almost uniform Lipschitz-continuity of $G$ and Lipschitz-continuity of $\pi_{\alpha}$ we estimate further:}
          &\quad\leq \tfrac{L_{\alpha}^{2}}{c^{2}}\medint\int_0^\infty \norm[\big]{\pi_{\alpha}\bigl((I_{\rho}v+Z)_{(t)}\bigr) -\pi_{\alpha}\bigl((I_{\rho}w+Z)_{(t)}\bigr)}_{H^{1}(-h,0;H)}^{2}\e^{-2\rho t}\dx[t]\\
          &\quad\leq \tfrac{L_{\alpha}^{2}}{c^{2}}\medint\int_{0}^{\infty}\norm[\big]{(I_{\rho}v+Z)_{(t)} -(I_{\rho}w+Z)_{(t)}}_{H^{1}(-h,0;H)}^{2}\e^{-2\rho t}\dx[t]\\
          &\quad=  \tfrac{L_{\alpha}^{2}}{c^{2}}\medint\int_{0}^{\infty}\norm[\big]{\Theta_{t}\bigl(I_{\rho}v+Z\bigr) -\Theta_{t}\bigl(I_{\rho}w+Z\bigr)}_{H^{1}(-h,0;H)}^{2}\e^{-2\rho t}\dx[t]\\
          \intertext{We recall that by \cref{th:delayIntegralCommute}, $I_{\rho}$ and the delay operator interchange, which allows us to equate:}
          &\quad=  \tfrac{L_{\alpha}^{2}}{c^{2}} \medint\int_{0}^{\infty}\norm[\Big]{\smallint_{-h}^{\argdot}\bigl(\Theta_{t}v\bigr) -\smallint_{-h}^{\argdot}\bigl(\Theta_{t}w\bigr)}_{H^{1}(-h,0;H)}^{2}\e^{-2\rho t}\dx[t]\\
          &\quad=  \tfrac{L_{\alpha}^{2}}{c^{2}} \medint\int_{0}^{\infty}\Big[\norm[\big]{v_{(t)} -w_{(t)}}_{L_{2}(-h,0;H)}^{2} +\norm[\Big]{\smallint_{-h}^{\argdot}\bigl(v_{(t)} -w_{(t)}\bigr)}_{L_{2}(-h,0;H)}^{2}\Big]\e^{-2\rho t}\dx[t]\\
          \intertext{We make use of the preliminary observation ($\square$) and obtain:}
          &\quad\leq  \tfrac{L_{\alpha}^{2}}{c^{2}} \medint\int_{0}^{\infty}\Big[\norm[\big]{v_{(t)} -w_{(t)}}_{L_{2}(-h,0;H)}^{2} +h^{2}\norm[\big]{v_{(t)} -w_{(t)}}_{L_{2}(-h,0;H)}^{2}\Big]\e^{-2\rho t}\dx[t]\\
          &\quad= \tfrac{L_{\alpha}^{2}}{c^{2}} (h^{2} + 1) \medint\int_{0}^{\infty} \norm{v_{(t)}-w_{(t)}}_{L_{2}(-h,0;H)}^{2}\e^{-2\rho t}\dx[t]\\
          &\quad= \tfrac{L_{\alpha}^{2}}{c^{2}}(h^{2} + 1) \norm{\Theta v -\Theta w}_{L_{2,\rho}(0,\infty;L_{2}(-h,0;H))}^{2}\\
          \intertext{Now we can appeal to Lipschitz-continuity of $\Theta$, cf.~\cref{th:ThetaIsLipschitz}:}
          &\quad\leq \tfrac{L_{\alpha}^{2}}{c^{2}}(h^{2}+1)\tfrac{1}{2\rho} \norm{v-w}_{L_{2,\rho}(-h,\infty;H)}^{2}\text{.}
        \end{align*}
  \item $\Gamma_{\rho,\alpha}$ is a contraction:\\
        This follows from the last estimate and the assumption $\rho > \max \bigl\{1,\tfrac{2 c^{2}}{L_{\alpha}^{2}(h^{2}+1)}\bigr\}$.
\end{itemize}
The contraction mapping principle therefore provides a solution $w$ of
\begin{equation*}
  w = \Gamma_{\rho,\alpha}(w)\text{,}\qquad w\in L_{2,\rho}(0,\infty;H)\text{.}
\end{equation*}
We again argue, that up to some $T>0$, we have $\pi_{\alpha}\bigl((I_{\rho}w + Z)_{(t)}\bigr) = \bigl(I_{\rho}w + Z\bigr)_{(t)}$.\\
Because of the integral term $I_{\rho}$, it suffices to verify that the term $w + Z'$ is bounded in $\norm{\argdot}_{\infty}$-norm by $\alpha$ up to some positive time $T>0$:
Appealing to the regularity preserving property of $G$ and the fact that $f\in H^{1}_{\rho}(0,\infty;H)$ by assumption, we know $f + G\bigl(\argdot,\pi_{\alpha}\bigl((I_{\rho}w + Z)_{(\argdot)}\bigr)\bigr) \in H^{1}_{\rho}(0,\infty;H)$. We further observe:
\begin{align*}
  \lim_{t\downarrow 0} G\bigl(t,\pi_{\alpha}\bigl((I_{\rho}w+Z)_{(t)}\bigr)\bigr)
  &= G\bigl(0, \pi_{\alpha}\bigl((I_{\rho}w + Z)_{(0)}\bigr)\bigr)\\[-1.5ex]
  &= G\bigl(0, \pi_{\alpha}\bigl((0 + Z)_{(0)} \bigr)\bigr)\\
  &= G(0, \Phi )\text{.}
\end{align*}
The consistency condition therefore assures
\begin{equation*}
  f + G\bigl(\argdot,\pi_{\alpha}\bigl((I_{\rho}w+Z)_{(\argdot)}\bigr)\bigr) \in H^{1}_{0,\rho}(0,\infty;H)\text{.}
\end{equation*}
By virtue of \cref{th:Commutation}, the solution $w$ of $\Gamma_{\rho,\alpha}(w) = w$ is in $H^{1}_{0,\rho}(0,\infty;H)$ and hence $\norm{w(t)}_{H}< \alpha - \norm{\Phi'}_{\infty}$ up to some positive time $\tilde{T}>0$, where we make use of \cref{th:Sobolev}. Since $Z' \in \mathcal{C}(0,\tilde{T};H)$ and $\norm{Z'(0)}_{H}\leq \norm{\Phi'}_{\infty}$ by assumption, continuity provides a positive $0<T\leq \tilde{T}$ such that $\norm{w + Z'}_{\mathcal{C}(-h,T;H)}\leq \alpha$.\\
Uniqueness of the solution follows similarly as in the proof of \cref{th:localExistOut} from the uniqueness of the fixed point of $\Gamma_{\rho,\alpha}$ by appealing to the contraction mapping principle and eventual independence of the weight parameter $\rho$ appealing to \cref{th:Picard}.
\end{proof}

\begin{remark}[maximal existence interval]
  \label{rmk:maxExistIntIns}
  A closer inspection of the proof reveals that the unique local solution $u = I_{\rho}v + Z$ of \cref{th:localExistIns} exists up to $T>0$ at which $\lim_{t\uparrow T}\norm{u}_{\mathcal{C}(0,t;H)}=+\infty$:\\
  Indeed, in the proof of \cref{th:localExistIns}, the solution of the FPP containing $\pi_{\alpha}$
  \begin{equation*}
    v = \Gamma_{\rho, \alpha}(w) = S_{\rho}\bigl[f + G\bigl(\argdot, \pi_{\alpha}\bigl((I_{\rho}v +Z)_{(\argdot)}\bigr)\bigr)\bigr]\text{,} \qquad v \in L_{2,\rho}(0,\infty;H)\text{,}
  \end{equation*}
  is a local solution of \cref{eq:FPPIntIns} for as long as the $\norm{\argdot}_{\infty}$-norm of $v$ remains bounded by $\alpha$. The claim follows as in the case of \cref{eq:FPPIntOut}, cf.~\cref{rmk:maxExistIntOut}, replacing $v'$ and $w'$ with $v$ and $w$ respectively.
\end{remark}
While increased regularity and a consistency condition proved critical for the proofs of \cref{th:localExistOut} and \cref{th:localExistIns}, one can lower these assumptions if the equation itself provides an increase in regularity for its solution. For example, we can make use of parabolic regularity (in the form of \cref{th:parabolicRegularity}), whenever available. We consider the problem
\begin{align}
  \label{eq:ParabolicProblem}
  \begin{aligned}
    \biggl[\partial_{\rho}
    \begin{pmatrix}
      M_{00}(\partial_{\rho}) & 0 \\ 0 & 0
    \end{pmatrix}
    + \begin{pmatrix}
      N_{00}(\partial_{\rho}) & N_{01}(\partial_{\rho}) \\ N_{10}(\partial_{\rho}) & N_{11}(\partial_{\rho})
    \end{pmatrix}
    &+ \begin{pmatrix}
      0 & -C\adjun \\ C & 0
    \end{pmatrix}
      \biggr]\begin{pmatrix} v \\ w \end{pmatrix}\\
    &=
    \begin{pmatrix}
      f + G\bigl(\argdot, (I_{\rho}v + Z)_{(\argdot)}\bigr) \\ 0
    \end{pmatrix}
    \text{,}
  \end{aligned}
\end{align}
where $(M,A)$ is a parabolic pair according to \cref{def:ParabolicPair}. For the right-hand side we assume the form from \cref{eq:FPPIntIns}, and in addition, that $f+G$ only takes values in $H_{0}$.\footnote{We again use the notation of \cref{def:ParabolicPair}.} We will make use of the following result in \cref{subsec:Heat}.
\begin{theorem}
  \label{th:localExistParabolic}
  Let $(M,A)$ be a parabolic pair. Let $G$ satisfy assumptions \ref{ass:A} i) and iii) and let $\Phi$ and $Z$ satisfy assumption \ref{ass:B}. Let further $G$ only take values in $H_{0}$ and $f \in H^{1}_{\tilde{\rho}}(0,\infty;H_{0})$ for some $\tilde{\rho}>0$. Then \cref{eq:ParabolicProblem} has a unique local solution $\begin{psmallmatrix} v \\ w \end{psmallmatrix} \in H^{1}(0,T;H_{0})\times H^{\nicefrac{1}{2}}(0,T;H_{0}^{\perp})$, provided that $f(0) = -G(0,\Phi)$.
\end{theorem}
\begin{proof}
  We first point out that the left-hand side of \cref{eq:ParabolicProblem} satisfies the assumptions of \cref{th:Picard} and therefore fits the setting of evolutionary equations. We can follow the steps of \cref{th:localExistIns} (using the same notation as there) to obtain a fixed point $(v,w) \in  L_{2,\rho}(0,\infty;H_{0})\times L_{2,\rho}(0,\infty;H_{0}^{\perp})$ of the system (\ref{eq:ParabolicProblem}) for some large enough $\rho>\tilde{\rho}$, but with right-hand side
  \begin{equation*}
    \begin{pmatrix}
      f + G\bigl(\argdot, \pi_{\alpha}\bigl((I_{\rho}v + Z)_{(\argdot)}\bigr)\bigr) \\ 0
    \end{pmatrix}\text{.}
  \end{equation*}
  Again, we simply need to prove that the projection $\pi_{\alpha}$ is locally the identity up to some positive time $T >0$. For this we can make use of parabolic regularity by means of \cref{th:parabolicRegularity}, since the second component of the right-hand side is $0$. Parabolic regularity then assures that $v \in H^{1}_{\rho}(0,\infty;H_{0})$ and $w\in H^{\nicefrac{1}{2}}_{\rho}(0,\infty;H_{0}^{\perp})$. Because of the integral term $I_{\rho}$, as in the proof of \cref{th:localExistIns}, we only need to verify local boundedness of $v+Z'$ in $\norm{\argdot}_{\infty}$-norm by $\alpha$. For this we observe that since the right-hand side vanishes for $t\leq 0$ (appealing to the consistency condition), by virtue of causality of $S_{\rho}$, we infer that $v\vert_{\mathbb{R}_{\leq 0}}=0$. This in turn implies $v\in \operatorname{dom}\bigl(\mathring{\partial}_{\rho}\bigr)$, in particular $v(0)=0$. Together with the embedding \cref{th:Sobolev}, the claim follows for some positive time $T>0$, since $\alpha > \norm{\Phi'}_{\infty}$.
\end{proof}

\subsection{Global existence}
\label{subsec:GlobalExist}
Since our approach follows \cite{Waurick2023} and we extend the techniques for ODEs to the PDE case, it is no surprise that we can also prove the usual maximal existence result one expects from classical ODE theory:
\begin{proposition}[maximal existence interval]
  \label{th:maximalExist}
  Under the assumptions of
  \begin{itemize}[leftmargin=4ex]
    \item \cref{th:localExistOut}, the unique local solution of \cref{eq:FPPIntOut} exists up to $T>0$ at which $\lim_{t\uparrow T}\norm{u'}_{\mathcal{C}(0,t;H)}=+\infty$.
    \item \cref{th:localExistIns}, the unique local solution of \cref{eq:FPPIntIns} exists up to $T>0$ at which $\lim_{t\uparrow T}\norm{u}_{\mathcal{C}(0,t;H)}=+\infty$.
  \end{itemize}
\end{proposition}
\begin{proof}
  The first claim is \cref{rmk:maxExistIntOut}, the second one \cref{rmk:maxExistIntIns}.
\end{proof}
Providing a criterion for global existence is not easy. As in the ODE case, one can aim to apply Gr\"onwall’s lemma, but it turns out that the notion of almost uniform Lipschitz-continuity only provides an estimate in $H^{1}$-norm, which does not suffice. Hence, we do have to assume a stronger condition:
\begin{definition}
  \label{def:AlmLCinL2}
  Let $F\colon [0,\infty)\times L_{2}(-h,0;H) \to H$ be continuous. $F$ is called {\em almost uniformly Lipschitz-continuous in $L_{2}$} if $\forall \alpha > 0\exists L_{\alpha}>0\forall t\geq 0\forall \varphi,\psi \in V_{\alpha}$:
  \begin{equation*}
    \norm{F(t,\varphi)-F(t,\psi)}_{H}\leq L_{\alpha}\norm{\varphi-\psi}_{L_{2}(-h,0;H)}\text{,}
  \end{equation*}
  where $V_{\alpha}$ is defined as in \cref{def:AlmLC}.
\end{definition}
\begin{remark}
  We point out that many practical examples do not exhibit this property. In particular, not the ``standard'' example
  \begin{equation*}
    F(t,\varphi) = \varphi\bigl(-\tau(\phi(0))\bigr) \quad \leftrightsquigarrow
    \quad F\bigl(t,u_{(t)}\bigr) = u\bigl(t-\tau(u(t))\bigr)
  \end{equation*}
for some Lipschitz-continuous $\tau$. Not even constant delay, i.e., the function $F(t,\varphi)=\varphi(-\tau)$ for a constant $\tau$ does exhibit this property. It is possible to define such delays however, for instance the integral mean
\begin{equation*}
  F(t,\varphi) = \frac{1}{h}\int_{-h}^{0}\varphi(s)\dx[s]
\end{equation*}
qualifies, as a trivial calculation quickly shows:
\begin{align*}
  \norm{F(t,\varphi)-F(t,\phi)}_{H}
  &= \frac{1}{h}\norm[\Big]{\medint\int_{-h}^{0}\varphi(s) - \phi(s)\dx[s]}\\
  &\leq \Bigl(\frac{1}{h}\medint\int_{-h}^{0}\norm{\varphi(s)-\phi(s)}^{2}\dx[s]\Bigr)^{-\nicefrac{1}{2}}
    =h^{-\frac{1}{2}}\norm{\varphi -\phi}_{L_{2}(-h,0;H)}\text{.}
\end{align*}
Examples for such an integral mean type of delay are provided/referenced in \cite{Khasawneh2011}, coming from car following models, wheel shimmy and machining dynamics, although not all examples there meet \cref{def:AlmLCinL2}. An example of a delay matching \cref{def:AlmLCinL2} originating from mathematical biology is discussed in \cite[rmk.~6.3]{Waurick2023}.
\end{remark}
\begin{theorem}[global existence for \cref{eq:FPPIntIns}]
  \label{th:GlobalExistForIntIns}
  The solution of \cref{eq:FPPIntIns} is global if in addition to the assumptions of \cref{th:localExistIns}, $G$ is
  \begin{itemize}[leftmargin=4ex]
    \item (globally) Lipschitz-continuous w.r.t. the first argument and
    \item almost uniformly Lipschitz-continuous in $L_{2}$ and the Lipschitz-constant $L_{\alpha}$ in \cref{def:AlmLCinL2} is uniform with respect to $\alpha$.
  \end{itemize}
\end{theorem}
\begin{proof}
  We begin with an initial observation: For $t,s\in \mathbb{R}_{\geq 0}$ we can estimate
  \begin{align*}
    &\norm{\widetilde{G}_{\alpha}(I_{\rho}v)(t) - \widetilde{G}_{\alpha}(I_{\rho}v)(s)}_{H}\\
    &\quad\leq \norm[\big]{G\bigl(t,\pi_{\alpha}\bigl((I_{\rho}v+Z)_{(t)}\bigr)\bigr) -G\bigl(s,\pi_{\alpha}\bigl((I_{\rho}v+Z)_{(t)}\bigr)\bigr)}_{H}\\
    &\qquad +\norm[\big]{G\bigl(s,\pi_{\alpha}\bigl((I_{\rho}v+Z)_{(t)}\bigr)\bigr) -G\bigl(s,\pi_{\alpha}\bigl((I_{\rho}v+Z)_{(s)}\bigr)\bigr)}_{H}\\
    &\quad\leq L_{1}\abs{t-s} + L_{2}\norm[\big]{\pi_{\alpha}\bigl((I_{\rho}v+Z)_{(t)}\bigr) -\pi_{\alpha}\bigl((I_{\rho}v+Z)_{(s)}\bigr)}_{L_{2}(-h,0;H)}\text{,}\\
    \intertext{where $L_{1}$ and $L_{2}$ denote the Lipschitz-constants w.r.t. the first and second argument respectively. Using Lipschitz-continuity of $\pi_{\alpha}$, we estimate further:}
    &\quad\leq L_{1}\abs{t-s} + L_{2}\norm[\big]{(I_{\rho}v+Z)_{(t)} -(I_{\rho}v+Z)_{(s)}}_{L_{2}(-h,0;H)}\text{.}
  \end{align*}
  We can therefore infer for $B(t)\coloneq \widetilde{G}_{\alpha}(I_{\rho}v)(t)$:
  \begin{align}
    \begin{aligned}
    \frac{\norm{B(t)-B(s)}_{H}}{\abs{t-s}}
    &\leq L_{1} + L_{2}\frac{\norm[\big]{(I_{\rho}v+Z)_{(t)} - (I_{\rho}v+Z)_{(s)}}_{L_{2}(-h,0;H)}}{\abs{t-s}}\\
    &\underset{t\to s}{\longrightarrow} L_{1} + L_{2}\norm{(v+Z')_{(s)}}_{L_{2}(-h,0;H)}\text{.}
    \end{aligned}\tag{$\circ$}
  \end{align}
  This equation holds for almost all $s\geq 0$. Hence, $B$ is differentiable (almost everywhere) on the entire right half-line.\\
  We start with the actual proof: According to \cref{th:localExistIns}, the local solution $v$ to \cref{eq:FPPIntIns} is in $H^{1}_{0,\rho}(0,T;H)$ up to some positive time $T$. The local solution $v$ of \cref{eq:FPPIntIns} satisfies
  \begin{equation*}
    v = S_{\rho}\bigl[f + \widetilde{G}_{\alpha}(I_{\rho}v)\bigr] = S_{\rho}[f + B]\text{.}
  \end{equation*}
  Since $f$ and $B$ are $H^{1}_{\rho}$-functions (the latter because $G$ is regularity preserving, cf.~\cref{def:RegPres}) and the consistency condition $-f(0) = G(0,\Phi)=B(0)$ holds, we can interchange $\partial_{\rho}$ and $S_{\rho}$ according to \cref{th:Commutation} and obtain for $v'$:
  \begin{equation*}
    \partial_{\rho}v = \mathring{\partial}_{\rho}v = \mathring{\partial}_{\rho}S_{\rho}[f + B] = S_{\rho}[f' + B']\text{.}
  \end{equation*}
  We can apply the fundamental theorem of calculus and write for $0\leq t\leq T$:
  \begin{align*}
    v(t) &= \underbrace{v(0)}_{=0} + \medint\int_{0}^{t}v'(s)\dx[s] = \medint\int_{0}^{\infty}\chi_{[0,t]}S_{\rho}[f' + B'](s)\dx[s]\text{.}\\[-1ex]
    \intertext{We appeal to causality to further obtain}
    v(t) &= \medint\int_{0}^{\infty}\chi_{[0,t]}S_{\rho}\chi_{[0,t]}[f' + B'](s) \dx[s]\text{.}
  \end{align*}
  This allows us to estimate
  \begin{align*}
    &\norm{v(t)}_{H} \leq \norm[\Big]{\medint\int_{0}^{\infty}\chi_{[0,t]}S_{\rho}\chi_{[0,t]}[f' + B'](s)\dx[s]}_{H}\\
    &\quad\leq \medint\int_{0}^{\infty}\chi_{[0,t]} \norm[\big]{S_{\rho}\chi_{[0,t]}[f' + B'](s)}_{H} \dx[s]\\
    &\quad\leq \medint\int_{0}^{\infty}\e^{\rho s}\chi_{[0,t]}\e^{-\rho s} \norm[\big]{S_{\rho}\chi_{[0,t]}[f' + B'](s)}_{H} \dx[s]\\
    &\quad\leq \Bigl[\medint\int_{0}^{t}\e^{2\rho s}\dx[s]\Bigr]^{\nicefrac{1}{2}}
      \Bigl[\medint\int_{0}^{\infty} \norm[\big]{S_{\rho}\chi_{[0,t]}[f' + B'](s)}_{H}^{2}\e^{-2\rho s} \dx[s]\Bigr]^{\nicefrac{1}{2}}\\
                    &\quad\leq \tfrac{1}{\sqrt{2\rho}}\e^{\rho t}\norm{S_{\rho}\chi_{[0,t]}[f' + B']}_{L_{2,\rho}(0,\infty;H)}\\
                    &\quad\leq \tfrac{1}{\sqrt{2\rho}}\e^{\rho t} \norm{S_{\rho}}\Bigl[\medint\int_{0}^{\infty}\norm[\big]{\chi_{[0,t]}[f' + B'](s)}_{H}^{2}\e^{-2\rho s} \dx[s]\Bigr]^{\nicefrac{1}{2}}\\
                    &\quad\leq \tfrac{1}{\sqrt{2\rho}}\e^{\rho t} \norm{S_{\rho}}\Bigl[\norm{f'}_{L_{2,\rho}(0,\infty;H)} + \Bigl(\medint\int_{0}^{\infty} \norm[\big]{\chi_{[0,t]}B'(s)}_{H}^{2}\e^{-2\rho s} \dx[s]\Bigr)^{\nicefrac{1}{2}}\Bigr]\\
    \intertext{Now we can use the estimate for $B'$ from ($\circ$) to further estimate:}
                    &\quad\leq \tfrac{1}{\sqrt{2\rho}}\e^{\rho t} \norm{S_{\rho}}\Bigl[\norm{f'}_{L_{2,\rho}} + \Bigl(\medint\int_{0}^{t}\bigl(L_{1} +L_{2}\norm{(v+Z')_{(s)}}_{L_{2}(-h,0;H)}\bigr)^{2}\e^{-2\rho s} \dx[s]\Bigr)^{\nicefrac{1}{2}}\Bigr]\\
                    &\quad\leq \tfrac{1}{\sqrt{2\rho}}\e^{\rho t}\norm{S_{\rho}} \Bigl[\norm{f'}_{L_{2,\rho}}\!+ \!\sqrt{\tfrac{2}{2\rho}}L_{1} \!+ \!\Bigl(2L_{2}^{2}\medint\int_{0}^{t}2\bigl(\norm{v_{(s)}}^{2}_{L_{2}} \!+ \!\norm{Z'_{(s)}}_{L_{2}}^{2}\bigr)\e^{-2\rho s} \!\dx[s]\Bigr)^{\nicefrac{1}{2}}\Bigr]\\
    &\quad\leq \tfrac{1}{\sqrt{2\rho}}\e^{\rho t}\norm{S_{\rho}} \Bigl[\norm{f'}_{L_{2,\rho}} \!+ \!\tfrac{L_{1}}{\sqrt{\rho}} \!+ \!\Bigl(4L_{2}^{2}\medint\int_{0}^{t} h^{2}\norm{v}_{\mathcal{C}(0,s;H)}^{2}\e^{-2\rho s} \!\dx[s] \\
    &\qquad\qquad\qquad\qquad\qquad\qquad\qquad\quad+ 4L_{2}^{2}\norm{\Theta Z'}_{L_{2,\rho}(0,t;L_{2}(-h,0;H))}^{2}\Bigr)^{\nicefrac{1}{2}}\Bigr]\\
                    &\!\leq \!\tfrac{\e^{\rho t}}{\sqrt{2\rho}}\norm{S_{\rho}}\! \Bigl[\!\norm{f'}_{L_{2,\rho}}\!\!\!+ \!\tfrac{L_{1}}{\sqrt{\rho}} \!+ \!\sqrt{\!\tfrac{2}{\rho}}L_{2} \norm{Z'}_{L_{2,\rho}(-h,\infty;H)} \!+ \!2L_{2}h \!\Bigl(\medint\int_{0}^{t}\norm{v}_{\mathcal{C}(0,s;H)}^{2}\e^{-2\rho s} \!\dx[s]\!\Bigr)^{\!\frac{1}{2}}\Bigr]\text{.}
  \end{align*}
  In other words, the estimate shows (collecting the respective terms into constants $C_{1}$ and $C_{2}$):
  \begin{equation*}
    \e^{-\rho t}\norm{v(t)}_{H}\leq \Bigl[C_{1} + C_{2}\medint\int_{0}^{t}\bigl(\norm{v}_{\mathcal{C}(0,s;H)}\e^{-\rho s}\bigr)^{2}\dx[s]\Bigr]^{\frac{1}{2}}\text{.}
  \end{equation*}
  Consequently (altering the constants),
  \begin{equation*}
    \e^{-2\rho t}\norm{v(t)}_{H}^{2}\leq C_{1} + C_{2}\medint\int_{0}^{t}\norm{v}_{\mathcal{C}(0,s;H)}^{2}\e^{-2\rho s}\dx[s]\text{.}
  \end{equation*}
  Since the mapping $s\mapsto \norm{v(s)}_{H}$ is continuous, there exists $t^{\ast}\in [0,t]$ such that $\norm{v(t^{\ast})}_{H}=\norm{v}_{\mathcal{C}(0,t;H)}$. Hence,
  \begin{align*}
    \e^{-2\rho t}\norm{v}^{2}_{\mathcal{C}(0,t;H)}
    &\leq \e^{-2\rho t^{\ast}}\norm{v(t^{\ast})}^{2}\\
    &\leq C_{1} + C_{2}\medint\int_{0}^{t^{\ast}}\norm{v}^{2}_{\mathcal{C}(0,s;H)}\e^{-2\rho s}\dx[s]\\
    &\leq C_{1} + C_{2}\medint\int_{0}^{t}\norm{v}^{2}_{\mathcal{C}(0,s;H)}\e^{-2\rho s}\dx[s]\text{.}
  \end{align*}
  Thus, Gr\"onwall's lemma applied to $\gamma(t)\coloneq \e^{-2\rho t}\norm{v}_{\mathcal{C}(0,t;H)}^{2}$ implies $\gamma(t)\leq C_{1} \e^{C_{2}t}$. Consequently for $\rho_{1} > \tfrac{C_{2}}{2} + \rho$,
  \begin{align*}
    \norm{v}^{2}_{L_{2,\rho_{1}}(0,\infty;H)} &= \medint\int_{0}^{\infty}\norm{v(s)}_{H}^{2}\e^{-2\rho_{1} s}\dx[s]
                                              \leq \medint\int_{0}^{\infty}\e^{2\rho s}\gamma(s)\e^{-2\rho_{1} s}\dx[s]\\
    &\leq \medint\int_{0}^{\infty}C_{1}\e^{(C_{2} + 2(\rho - \rho_{1}))s}\dx[s] < \infty\text{.}
  \end{align*}
  Hence, the local solution $v$ of \cref{eq:FPPIntIns} defined on successive time-intervals exists as an $L_{2,\rho_{1}}(0,\infty;H)$-function. It remains to verify that $v$ is a fixed point of \eqref{eq:FPPinL2}. For that, it suffices to prove that $t\mapsto G\bigl(t,(I_{\rho_{1}}v + Z)_{(t)}\bigr) \in L_{2,\rho_{1}}(0,\infty;H)$.
\begin{align*}
  &\norm[\big]{G\bigl(\argdot,(v+Z)_{(\argdot)}\bigr)}_{L_{2,\rho_{1}}(0,T;H)}
  = \norm[\big]{G\bigl(\argdot,\pi_{\alpha}\bigl((v+Z)_{(\argdot)}\bigr)\bigr)}_{L_{2,\rho_{1}}(0,T;H)}\\
  &\quad\leq \norm{G(\argdot,0)}_{L_{2,\rho_{1}}(0,T;H)} + \norm[\big]{G(\argdot,0) - G\bigl(\argdot,\pi_{\alpha}\bigl((I_{\rho}v+Z)_{(\argdot)}\bigr)\bigr)}_{L_{2,\rho_{1}}(0,T;H)}\\
  \intertext{Use of almost uniform Lipschitz-continuity in $L_{2}$ with uniform constant $L$ gives:}
  &\quad\leq \norm{G(\argdot,0)}_{L_{2,\rho_{1}}(0,T;H)} + L\Bigl[\medint\int_{0}^{T}\norm[\big]{\pi_{\alpha}\bigl((I_{\rho_{1}}v+Z)_{(\argdot)}\bigr)}_{L_{2}(-h,0;H)}^{2}\e^{-2\rho_{1}s}\dx[s]\Bigr]^{\nicefrac{1}{2}}\\
  &\quad\leq \norm{G(\argdot,0)}_{L_{2,\rho_{1}}(0,T;H)} + L\Bigl[\medint\int_{0}^{T}\norm{(I_{\rho_{1}}v+Z)_{(s)}}_{L_{2}(-h,0;H)}^{2}\e^{-2\rho_{1}s}\dx[s]\Bigr]^{\nicefrac{1}{2}}\\
  &\quad\leq \norm{G(\argdot,0)}_{L_{2,\rho_{1}}(0,T;H)} + L\norm{\Theta(I_{\rho_{1}}v + Z)}_{L_{2,\rho_{1}}(0,T;L_{2}(-h,0;H))}\\
  &\quad\leq \norm{G(\argdot,0)}_{L_{2,\rho_{1}}(0,T;H)} + \tfrac{L}{\sqrt{2\rho_{1}}}\bigl(\tfrac{1}{\rho_{1}}\norm{v}_{L_{2,\rho_{1}}(-h,T;H)} + \norm{Z}_{L_{2,\rho_{1}}(-h,T;H)}\bigr)\text{.}
\end{align*}
  The claim follows letting $T\to \infty$. Hence, $v$ is a global solution of \cref{eq:FPPIntIns}.
\end{proof}

\begin{remark}
  The difficult part of the proof of \cref{th:GlobalExistForIntIns} is establishing an exponential bound for $\norm{v}_{\mathcal{C}(0,t;H)}$ in order to prove that $v$ globally solves \cref{eq:FPPIntIns}. Once $v \in L_{2,\rho}(0,\infty;H)$ for some large enough $\rho >0$ is established, verifying that $v$ is a fixed point of \cref{eq:FPPinL2} is easy.\\
  Conversely, an investigation of the proof of \cref{th:GlobalExistForIntIns} reveals that a function $v \in L_{2,\rho}(0,\infty;H)$ that solves \cref{eq:localEq} for all $T>0$ is a global solution, i.e., $v$ is a fixed point of \cref{eq:FPPinL2} --- provided that the requirements of \cref{th:localExistIns} are met and that the Lipschitz-constant $L_{\alpha}$ is global (independent of $\alpha$). The stronger notion of almost uniform Lipschitz-continuity in $L_{2}$ and Lipschitz-continuity w.r.t. the first argument of $G$ are not required for this.
\end{remark}

\begin{theorem}[Global existence for \cref{eq:FPPIntOut}]
  \label{th:GlobalExistForIntOut}
  The solution of \cref{eq:FPPIntOut} is global if in addition to the assumptions of \cref{th:localExistOut}, $G$ is
  \begin{itemize}[leftmargin=4ex]
    \item (globally) Lipschitz-continuous w.r.t. the first argument and
    \item almost uniformly Lipschitz-continuous in $L_{2}$ and the Lipschitz-constant $L_{\alpha}$ in \cref{def:AlmLCinL2} is uniform with respect to $\alpha$.
  \end{itemize}
\end{theorem}
\begin{proof}
  The proof is completely analoguous to the proof of \cref{th:GlobalExistForIntIns}: Instead of $v$ as in the proof of \cref{th:GlobalExistForIntIns} one estimates $v'$ instead (note that no further derivatives of $G$ are required because of the leading integral term $I_{\rho}$).
\end{proof}

\subsection{Continuous dependence on data}
\label{subsec:ContDepOnData}
The remaining part in the discusstion of well-posedness is continuous dependence on data. We will focus on \cref{eq:FPPIntIns} and start with the dependence on the right-hand sides.

\begin{theorem}[continuous dependence on the right-hand sides]
  Under the assumptions of \cref{th:localExistIns} let $u,v$ be local solutions of \cref{eq:FPPIntIns} corresponding to right-hand sides $f_{1} + G_{1}$ and $f_{2} + G_{2}$ respectively up to some $T>0$ such that $u_{(t)},v_{(t)}$ remain in $V_{\alpha}$. Then there exist $\rho>0$ and $C>0$ such that:
  \begin{equation*}
    \norm{u-v}_{L_{2,\rho}(0,T;H)}^{2}\leq C\Bigl[\norm{f_{1}-f_{2}}_{L_{2,\rho}(0,T;H)}^{2} \!+\!\sup_{t\in [0,T]} \norm[\big]{\widetilde{G_{1}}(u)(t) - \widetilde{G_{2}}(v)(t)}_{H}^{2}\Bigr]\text{.}
  \end{equation*}
\end{theorem}
\begin{proof}
  We estimate, utilizing the fixed point equation \cref{eq:localEq}
  \begin{align*}
    &\norm{u-v}_{L_{2,\rho}(0,T;H)}^{2}
      =\norm[\big]{S_{\rho}\bigl[f_{1} - f_{2} + \widetilde{G_{1}}(u) - \widetilde{G_{2}}(v)\bigr]}_{L_{2,\rho}(0,T;H)}^{2}\\
    &\quad\leq \norm{S_{\rho}}^{2}\norm[\big]{f_{1} - f_{2} + \widetilde{G_{1}}(u) -\widetilde{G_{2}}(v)}_{L_{2,\rho}(0,T;H)}^{2}\\
    &\quad\leq 2\norm{S_{\rho}}^{2}\Bigl[\norm{f_{1}-f_{2}}_{L_{2,\rho}(0,T;H)}^{2}\\[-1ex]
    &\qquad\qquad\qquad+ \medint\int_0^T \norm[\big]{G_{1}\bigl(t,(I_{\rho}u+Z)_{(t)}\bigr)
      -G_{2}\bigl(t,(I_{\rho}v+Z)_{(t)}\bigr)}_{H}^{2}\e^{-2\rho t} \dx[t]\Bigr]\\
    &\quad\leq 2\norm{S_{\rho}}^{2} \Bigl[\norm{f_{1}-f_{2}}_{L_{2,\rho}}^{2}
      + 2\medint\int_0^T\bigl[\norm[\big]{G_{1}\bigl(t,(I_{\rho}u+Z)_{(t)}\bigr)                                   -G_{1}\bigl(t,(I_{\rho}v+Z)_{(t)}\bigr)}_{H}^{2}\\[-1ex]
    &\qquad\qquad\qquad+ \norm[\big]{G_{1}\bigl(t,(I_{\rho}v+Z)_{(t)}\bigr)
      -G_{2}\bigl(t,(I_{\rho}v+Z)_{(t)}\bigr)}^{2}\bigr] \e^{-2\rho t} \dx[t]\Bigr]\text{.}
  \end{align*}
  The first integral term we can estimate using almost uniform Lipschitz-continuity of $G_{1}$ as follows:
  \begin{align*}
    &\medint\int_0^T \norm[\big]{G_{1}\bigl(t,(I_{\rho}u+Z)_{(t)}\bigr)
      -G_{1}\bigl(t,(I_{\rho}v+Z)_{(t)}\bigr)}_{H}^{2} \e^{-2\rho t}\dx[t]\\[-1ex]
    &\quad \leq L_{G_{1}}^{2} \medint\int_{0}^{T} \norm[\big]{(I_{\rho}u+Z)_{(t)}
      - (I_{\rho}v+Z)_{(t)}}_{H^{1}(-h,0;H)}^{2}\e^{-2\rho t} \dx[t]\\
    &\quad= L_{G_{1}}^{2} \medint\int_{0}^{T} \norm[\big]{(I_{\rho}(u-v))_{(t)}}_{H^{1}(-h,0;H)}^{2}\e^{-2\rho t} \dx[t]\\
    &\quad= L_{G_{1}}^{2} \norm{\Theta(I_{\rho}(u-v))}_{L_{2,\rho}(0,T;H^{1}(-h,0;H))}^{2}\\
    \intertext{We can estimate this term appealing to \cref{rmk:ThetaIsLipschitz}:}
    &\quad\leq L_{G_{1}}^{2}\tfrac{1}{2\rho} \norm{I_{\rho}(u-v)}_{H^{1}_{\rho}(-h,T;H)}^{2}\\
    &\quad\leq L_{G_{1}}^{2}\tfrac{1}{2\rho}\bigl(1+\tfrac{1}{\rho^{2}}\bigr) \norm{u-v}_{L_{2,\rho}(0,T;H)}^{2}\text{.}
  \end{align*}
  For the second integral term we simply observe that since $I_{\rho}\varphi$ is continuous for $\varphi \in L_{2}$, the function $t\mapsto (I_{\rho}\varphi)_{(t)}$ is continuous. Hence, the supremum of $t\mapsto G_{1/2}\bigl(t,(I_{\rho}\varphi)_{(t)}\bigr)$ exists on $[0,T]$ and the second term can be estimated with
  \begin{equation*}
    \tfrac{1}{2\rho}\sup_{t\in [0,T]} \norm[\big]{G_{1}\bigl(t,(I_{\rho}v+Z)_{(t)}\bigr)
      -G_{2}\bigl(t,(I_{\rho}v+Z)_{(t)}\bigr)}_{H}\text{.}
  \end{equation*}
  Making $\rho$ large enough and rearranging gives the inequality
  \begin{align*}
    &\norm{u-v}_{L_{2,\rho}(0,T;H)}^{2} \\
    &\quad\leq \!C\Bigl[\norm{f_{1}\!-\!f_{2}}_{L_{2,\rho}(0,T;H)}^{2} \!+\!\sup_{t\in [0,T]} \norm[\big]{G_{1}\bigl(t,(I_{\rho}v \!+ \!Z)_{(t)}\bigr)\!- \!G_{2}\bigl(t,(I_{\rho}v\!+\!Z)_{(t)}\bigr)}_{H}^{2}\Bigr]\text{.}
      \qedhere
  \end{align*}
\end{proof}

The discussion about the specific continuous dependence on initial datum begs the consideration of the special case of simple material laws, because as outlined in \cref{subsec:DistribSpaces_IVPs}, the function $f$ in \cref{eq:GenFPP} arising from incorporating the prehistory into the equation can be messy. For simple material laws it is easier to recover information on the prehistory. We postpone the corresponding result (\cref{th:ContDepOnIV}) to \cref{sec:SimMatLaws}.

\section{Simple material laws}
\label{sec:SimMatLaws}
In this section, we take a closer look at the case of simple material laws, i.e., material laws of the form
\begin{equation*}
  M(z) = M_{0} + z^{-1}M_{1}\text{,}
\end{equation*}
where $M_{0},M_{1}\in \mathcal{L}(H)$. From \cite[prop.~3.2.10]{Trostorff2018} (cf.~\cref{th:IV_for_simple_MatLaws}) it follows that for a prehistory $\Phi \in \chi_{\mathbb{R}_{\leq 0}}H^{1}_{\rho}(\mathbb{R};H)$, the space of admissible prehistories reduces to the space of admissible initial values. In particular, the projection $P_{0}$ (cf.~\cref{def:cutoff}) behaves nicely for such operators and we can explicitly spell out what $f$ in \cref{eq:EvolProblem} is and make some conditions from \cref{sec:SolThy} explicit.\\
To do this, we specifically choose the extension $Z$ of the prehistory $\Phi \in H^{1}(-h,0;H)$ with $\Phi'\in L_{\infty}(-h,0;H)$. To that end, we will make repeated use of the assumption $\Phi(-h)=0$ (and $\Phi^{(k)}(-h)=0$ for $k\geq 0$ if required). This is only a formal restriction though:
\begin{remark}
  \label{rmk:PrehistoryVanishes}
  If $\Phi(-h)\neq 0$, we can simply consider $\tilde{\Phi}\coloneq \Phi - \Phi(-h)$ and solve the equation with a right-hand side $\hat{F}(t,\varphi)\coloneq F(t, \varphi + \Phi(-h))$. Assuming that $F$ satisfies assumption \ref{ass:A}, so does $\hat{F}$.\\
  Alternatively, one can also connect the point $(-h, \Phi(-h))$ to $(-h-1,0)$ via an affine segment, preserving regularity of $\Phi$ and ``force'' this condition for $-h-1$ instead of $-h$. The transition from $-h$ to $-h-1$ is irrelevant, as long as the function $F$ only depends on the histories in $[-h,0]$.\\
  The same argument applies to higher derivatives (if available) by using a suitably regular connection of the aforementioned points.
\end{remark}
As in \cref{sec:SolThy}, in this section we assume that $A\colon H\supseteq \operatorname{dom}(A)\to H$ is an $\mathrm{m}$-accretive operator and that $A$ and $M$ satisfy the assumptions of \cref{th:Picard}. Rather than trying to compute the projection $P_{0}$ of the differential operator $\partial_{\rho}M_{0} + M_{1} + A$, we will derive the problem formulations directly. To that end, we will make repeated use of the following notation for a given $\Psi \in \mathcal{C}(-h,0;H)$:
\begin{equation*}
  \hat{\Psi}\colon\,[-h,\infty)\rightarrow H\text{,}\qquad
  t\mapsto\begin{cases*}\Psi(t) &for $-h\leq t\leq 0$\text{,} \\\Psi(0) &for \qquad\,\,\,\,$t>0$\text{.}\end{cases*}
\end{equation*}
We repeat that for $\Phi \in H^{1}(-h,0;H)$ with $\Phi' \in L_{\infty}(-h,0;H)$, the continuation $\hat{\Phi}$ is an $H_{\rho}^{1}(-h,\infty;H)$-function by virtue of the integration by parts formula and continuity in $0$. By induction, $I_{\rho}^{k}\widehat{\Phi^{(k)}}\in H^{k}_{\rho}$ for $\Phi\in \mathcal{C}^{k}(-h,0;H)$,  where we understand $I_{\rho}= \smallint_{-h}^{\argdot}$.\\
We aim to solve the ``shifted'' problem with trivial prehistory, making the ansatz $u = v + Z$ for some $v$ with $\supp(v)\subseteq [0,\infty)$:
\begin{alignat}{3}
  \begin{aligned}
    \label{eq:FDESimpleMatLaws}
    \bigl(\partial_{\rho}M_{0}+M_{1}+A\bigr)\bigl(v+Z\bigr)   &=F\bigl(\argdot,(v+Z)_{(\argdot)}\bigr)&&\quad\text{for}\qquad\quad\, t>0\text{,}\\
    v(t)&= 0 &&\quad\text{for} -h\leq t\leq 0\text{.}
  \end{aligned}
\end{alignat}
As in \cref{sec:SolThy}, we distinguish between two different cases depending on the shape of the right-hand side of \cref{eq:EvolProblem}. We will choose $Z$ appropriately.

\subsection{The case $\widetilde{F}(v) = I_{\rho}\widetilde{G}(v)$}
We make the ansatz $u = v + I_{\rho}I_{\rho}\widehat{\Phi''}$ for some $v \in H^{1}_{0,\rho}(0,\infty;H)$, assuming that $\Phi'' \in \mathcal{C}^{2}(-h,0;H)$. Applied to \cref{eq:FDESimpleMatLaws}, this ansatz produces
\begin{equation*}
  \bigl(\mathring{\partial}_{\rho}M_{0}\!+\!M_{1}\!+\!A\bigr)v = \!-\!M_{0}I_{\rho}\Phi''(0) \!- \!(M_{1}\!+\!A)I_{\rho}\bigl(\Phi'(0) \!+ \!I_{\rho}\Phi''(0)\bigr)\!+ \!F\bigl(\argdot,(v \!+\!I_{\rho}I_{\rho}\widehat{\Phi''})_{(\argdot)}\bigr)\text{,}
\end{equation*}
because $v\in \operatorname{dom}\bigl(\mathring{\partial}_{\rho}\bigr)$. Assuming that $F$ can be written as in \cref{eq:FPPIntOut}, we obtain
\begin{align}
  \begin{aligned}
  \label{eq:SimplMatLawIntOut}
    \bigl(\mathring{\partial}_{\rho}M_{0}+M_{1}+A\bigr)v = &-M_{0}I_{\rho}\Phi''(0) - (M_{1}+A)I_{\rho}\bigl(\Phi'(0) + I_{\rho}\Phi''(0)\bigr) \\
    &+ I_{\rho}G\bigl(\argdot,(v +I_{\rho}I_{\rho}\widehat{\Phi''})_{(\argdot)}\bigr)\text{.}
  \end{aligned}
\end{align}
An application of the Picard operator $S_{\rho}$ now produces an equation in the style of \cref{eq:FPPIntOut} and we can prove a corresponding well-posedness theorem:
\begin{theorem}[local existence and uniqueness]
  Under assumption \ref{ass:A} on $G$, \cref{eq:SimplMatLawIntOut} has a unique local solution up to some $T>0$ provided that
  \begin{itemize}[leftmargin=4ex]
    \item $\Phi \in \mathcal{C}^{2}(-h,0;H)$ with $\Phi(0),\Phi'(0),\Phi''(0)\in \operatorname{dom}(A)$ and
    \item the consistency condition $M_{0}\Phi''(0) + (M_{1}+A)\Phi'(0) = G(0,\Phi)$ holds.
  \end{itemize}
\end{theorem}
\begin{proof}
  We simply apply \cref{th:localExistOut}, validating the assumptions there:
  \begin{itemize}[leftmargin=5ex]
    \item $I_{\rho}M_{0}\Phi''(0) - I_{\rho}(M_{1}+A)\bigl(\Phi'(0) + I_{\rho}\Phi''(0)\bigr) \in H^{1}_{0,\rho}\cap H^{2}_{\rho}$ holds trivially and
          \item the consistency condition reads:
  \begin{equation*}
    M_{0}\Phi''(0) + (M_{1}+A)\Phi'(0) = G\bigl(0, (I_{\rho}I_{\rho}\widehat{\Phi''})_{(0)}\bigr) = G(0, \Phi)\text{.}\qedhere
  \end{equation*}
  \end{itemize}
\end{proof}

\subsection{The case $\widetilde{F}(v) = \widetilde{G}(I_{\rho}v)$}\phantom{.}\\
Here we make the ansatz $u = v + I_{\rho}\widehat{\Phi'}$ with some $v \in H^{1}_{0,\rho}(0,\infty;H)$, assuming that $\Phi \in \mathcal{C}^{1}(-h,0;H)$. Reformation of the differential equation \eqref{eq:FDESimpleMatLaws} (for $t>0$) produces
\begin{align*}
  \bigl(\partial_{\rho}M_{0}\!+\!M_{1}\!+\!A\bigr)v &=\!-\bigl(\partial_{\rho}M_{0}\!+\!M_{1}\!+\!A\bigr)\bigl(I_{\rho}\Phi'(0) \!+ \!\Phi(0)\bigr)\!+\!F\bigl(\argdot,(v\!+\!I_{\rho}\widehat{\Phi'})_{(\argdot)}\bigr)\\
  &= \!-M_{0}\Phi'(0) \!- \!(M_{1}\!+\!A)\bigl(\Phi(0) \!+ \!I_{\rho}\Phi'(0)\bigr)\!+ \!F\bigl(\argdot,(v\!+\!I_{\rho}\widehat{\Phi'})_{(\argdot)}\bigr)\text{.}
\end{align*}
Since $v\in H^{1}_{0,\rho}(0,\infty;H)$, $v \in \operatorname{dom}\bigl(\mathring{\partial}_{\rho}\bigr)$ follows. We obtain:
\begin{equation*}
  \bigl(\mathring{\partial}_{\rho}M_{0}\!+\!M_{1}\!+\!A\bigr)v  \!=\!-\!M_{0}\Phi'(0) \!- \!(M_{1}\!+\!A)\bigl(\Phi(0) \!+ \!I_{\rho}\Phi'(0)\bigr)\!+ \!F\bigl(\argdot,(v\!+\!I_{\rho}\widehat{\Phi'})_{(\argdot)}\bigr)\text{.}
\end{equation*}
Assuming that $F$ has the same form as in \cref{eq:FPPIntIns}, the equation becomes
\begin{equation}
  \label{eq:SimplMatLawIntIns}
  \bigl(\mathring{\partial}_{\rho}M_{0}\!+\!M_{1}\!+\!A\bigr)v  \!=\!-\!M_{0}\Phi'(0) \!- \!(M_{1}\!+\!A)\bigl(\Phi(0) \!+ \!I_{\rho}\Phi'(0)\bigr)\!+ \!G\bigl(\argdot,(I_{\rho}(v\!+\widehat{\Phi'})_{(\argdot)}\bigr)\text{.}
\end{equation}
Again, we obtain a corresponding well-posedness theorem:
\begin{theorem}[local existence and uniqueness]
  \label{th:localExistInsSimple}
  Under assumption \ref{ass:A} on $G$, \cref{eq:SimplMatLawIntIns} has a unique local solution up to some $T>0$ provided that
  \begin{itemize}[leftmargin=4ex]
    \item $\Phi \in \mathcal{C}^{1}(-h,0;H)$ with $\Phi(0),\Phi'(0) \in \operatorname{dom}(A)$ and
    \item the consistency condition $M_{0}\Phi'(0) + (M_{1}+A)\Phi(0) = G(0,\Phi)$ holds.
  \end{itemize}
\end{theorem}
\begin{proof}
  We simply apply \cref{th:localExistIns} and translate the consistency condition there:
  \begin{equation*}
    f(0) = M_{0}\Phi'(0) + (M_{1}+A)\Phi(0) = G\bigl(0, (I_{\rho}\widehat{\Phi'})_{(0)}\bigr) = G(0, \Phi)\text{.}\qedhere
  \end{equation*}
\end{proof}

For this formulation and the special case of simple material laws we also prove continuous dependence on the initial data:
\begin{theorem}[continuous dependence on initial data]
  \label{th:ContDepOnIV}
  Let $G$ satisfy assumption \ref{ass:A} and let $\Phi,\Psi\in H^{2}(-h,0;H)$ be two prehistories satisfying the assumptions of \cref{th:localExistInsSimple}. Let $\alpha > \operatorname{max}\{\norm{\Phi'}_{\infty},$ $\norm{\Psi'}_{\infty}\}$. Then there exists $C>0$ such that for any solutions $u,v$ to \cref{eq:SimplMatLawIntIns} with initial prehistories $\Phi,\Psi$ respectively, that exist up to $T>0$ such that both $(I_{\rho}(u + \widehat{\Phi'}))_{(t)}$ and $(I_{\rho}(u + \widehat{\Psi'}))_{(t)}$ remain in $V_{\alpha}$ for all $t\leq T$:
  \begin{equation*}
    \norm{u-v}_{L_{2}(-h,T;H)}\leq C \bigl[\norm{\Phi-\Psi}_{H^{2}(-h,0;H)} + \norm{\Phi(0)-\Psi(0)}_{\operatorname{dom}(A)}\bigr]\text{.}
  \end{equation*}
\end{theorem}
\begin{proof}
  We recall that the solutions are given by
  \begin{equation*}
    u = w_{u} + I_{\rho}\widehat{\Phi'} \quad \text{and} \quad v = w_{v} + I_{\rho}\widehat{\Psi'}\text{,}
  \end{equation*}
  where $w_{u}$ and $w_{v}$ are local solutions to the FPP \eqref{eq:SimplMatLawIntIns}, which $w_{u}$ and $w_{v}$ solve up to the common time $T$. We start by estimating the $L_{2,\rho}(0,T;H)$-difference ($\rho$ large enough) of $w_{u}$ and $w_{v}$ by using the formula provided by \cref{eq:SimplMatLawIntIns}. We estimate the nonlinearities and the additional terms separately.
  \begin{enumerate}[leftmargin=3ex, label=\roman*)]
  \item We start with the nonlinearity:
  \begin{align*}
    &\norm[\big]{S_{\rho}\bigl[G\bigl(\argdot,(I_{\rho}(w_{u}+\widehat{\Phi'}))_{(\argdot)}\bigr)
      -G\bigl(\argdot,(I_{\rho}(w_{v}+\widehat{\Psi'}))_{(\argdot)}\bigr)\bigr]}_{L_{2,\rho}(0,T;H)}^{2}\\
    &\leq \norm{S_{\rho}}^{2}\norm[\big]{G\bigl(\argdot,(I_{\rho}(w_{u}+\widehat{\Phi'}))_{(\argdot)}\bigr)
      -G\bigl(\argdot,(I_{\rho}(w_{v}+\widehat{\Psi'}))_{(\argdot)}\bigr)}_{L_{2,\rho}(0,T;H)}^{2}\\
    &= \norm{S_{\rho}}^{2} \medint\int_0^T \norm[\big]{G\bigl(t,(I_{\rho}(w_{u}+\widehat{\Phi'}))_{(t)}\bigr)
      -G\bigl(t,(I_{\rho}(w_{v}+\widehat{\Psi'}))_{(t)}\bigr)}_{H}^{2}\e^{-2\rho t} \dx[t]\\
    \intertext{Using almost uniform Lipschitz continuity of $G$ we obtain:}
    &\leq \norm{S_{\rho}}^{2} L^{2}\medint\int_0^T \norm[\big]{\bigl(I_{\rho}(w_{u}+\widehat{\Phi'})\bigr)_{(t)}
      -\bigl(I_{\rho}(w_{v}+\widehat{\Psi'})\bigr)_{(t)}}_{H^{1}(-h,0;H)}^{2}\e^{-2\rho t} \dx[t]\text{.}
  \end{align*}
  We obtain two integrals because of the $H^{1}$-norm, that we have to estimate:
  \begin{enumerate}[leftmargin=3.5ex, label=\alph*)]
  \item We start with the derivative:
  \begin{align*}
    \;\;\;\;\;&\medint\int_0^T\norm[\big]{\bigl(w_{u}+\widehat{\Phi'}\bigr)_{(t)}
      -\bigl(w_{v}+\widehat{\Psi'}\bigr)_{(t)}}_{L_{2}(-h,0;H)}^{2}\e^{-2\rho t}\dx[t]\\
    &= \!\medint\int_0^T \!\medint\int_{-h}^0
      \norm[\big]{w_{u}(t \!+ \!s) \!- \!w_{v}(t \!+ \!s) \!+ \!\widehat{\Phi'}(t \!+ \!s) \!- \!\widehat{\Psi'}(t+s)}_{H}^{2}
      \dx[s] \e^{-2\rho t} \dx[t]\\
    \intertext{Using $(a+b)^{2}\leq 2(a^{2}+b^{2})$ we continue:}
    &\leq 2\medint\int_{-h}^T \medint\int_{s}^{s+h}
      \bigl(\norm{w_{u}(s)-w_{v}(s)}_{H}^{2} + \norm{\widehat{\Phi'}(s)-\widehat{\Psi'}(s)}_{H}^{2}\bigr)\e^{-2\rho t} \dx[t] \dx[s]\\
    &\leq \tfrac{1}{\rho} \medint\int_{-h}^T
      \bigl(\norm{w_{u}(s)-w_{v}(s)}_{H}^{2} +\norm{\widehat{\Phi'}(s)-\widehat{\Psi'}(s)}_{H}^{2}\bigr)\e^{-2\rho s} \dx[s]\\
    &= \tfrac{1}{\rho} \biggl[\norm{w_{u}\!-\!w_{v}}_{L_{2,\rho}(-h,T;H)}^{2}
      \!+ \!\norm{\Phi'\!-\!\Psi'}_{L_{2,\rho}(-h,0;H)}^{2} \!+ \!\medint\int_0^T \!\norm{\Phi'(0) \!-\!\Psi'(0)}^{2}\e^{-2\rho t}\!\dx[t]\biggr]\\
    &\leq \tfrac{1}{\rho}
      \bigl[\norm{w_{u} \!- \!w_{v}}_{L_{2,\rho}(-h,T;H)}^{2}
      \!+ \!\norm{\Phi' \!- \!\Psi'}_{L_{2,\rho}(-h,0;H)}^{2} \!+ \!\tfrac{1}{2\rho}\norm{\Phi' \!- \!\Psi'}_{\infty}^{2}\bigr]\\
    \intertext{The embedding \cref{th:Sobolev} provides a constant $\tilde{C}(\rho) >0$ such that}
    &\leq \tfrac{1}{\rho}
      \bigl[\tilde{C}(\rho)^{2}\norm{\Phi-\Psi}_{H^{2}_{\rho}(-h,0;H)}^{2}+\norm{w_{u}-w_{v}}_{L_{2,\rho}(0,T;H)}^{2}\bigr]\text{.}
  \end{align*}
  \item For the term with an additional integral term we can make a similar estimate:
  \begin{align*}
    \;\;\;\;\;&\medint\int_0^T\norm[\big]{\bigl(I_{\rho}(w_{u}+\widehat{\Phi'})\bigr)_{(t)}
      -\bigl(I_{\rho}(w_{v}+\widehat{\Psi'})\bigr)_{(t)}}_{L_{2}(-h,0;H)}^{2}\e^{-2\rho t}\dx[t]\\
    &\leq \tfrac{1}{\rho} \medint\int_{-h}^T
      \bigl(\norm{I_{\rho}(w_{u}-w_{v})(s)}_{H}^{2} +\norm{I_{\rho}\bigl(\widehat{\Phi'}-\widehat{\Psi'}\bigr)(s)}_{H}^{2}\bigr)\e^{-2\rho s} \dx[s]\\
    \intertext{Using that $I_{\rho}$ is a bounded linear operator, we continue as before:}
    &\leq \tfrac{1}{\rho^{3}} \biggl[\norm{w_{u}\!-\!w_{v}}_{L_{2,\rho}(-h,T;H)}^{2}
      \!+ \!\norm{\Phi'\!-\!\Psi'}_{L_{2,\rho}(-h,0;H)}^{2} \!+ \!\medint\int_0^T \!\norm{\Phi'(0) \!-\!\Psi'(0)}^{2}\e^{-2\rho t}\!\dx[t]\biggr]\\
    &\leq \tfrac{1}{\rho^{3}}
      \bigl[\norm{w_{u} \!- \!w_{v}}_{L_{2,\rho}(-h,T;H)}^{2}
      \!+ \!\norm{\Phi' \!- \!\Psi'}_{L_{2,\rho}(-h,0;H)}^{2} \!+ \!\tfrac{1}{2\rho}\norm{\Phi' \!- \!\Psi'}_{\infty}^{2}\bigr]\\
    \intertext{The embedding \cref{th:Sobolev} provides a constant $\tilde{C}(\rho)\!>\!0$ such that}
    &\leq \tfrac{1}{\rho^{3}}
      \bigl[\tilde{C}(\rho)^{2}\norm{\Phi-\Psi}_{H^{2}_{\rho}(-h,0;H)}^{2}+\norm{w_{u}-w_{v}}_{L_{2,\rho}(0,T;H)}^{2}\bigr]\text{.}
  \end{align*}
  \end{enumerate}
  Alltogether this implies (assuming $\rho >1$ if necessary),
  \begin{equation*}
    \norm{u-v}_{L_{2,\rho}(0,T;H)} \leq \norm{S_{\rho}} L \sqrt{\tfrac{2}{\rho}} \bigl[\norm{w_{u}-w_{v}}_{L_{2,\rho}(0,T;H)}+\tilde{C}(\rho)\norm{\Phi-\Psi}_{H^{2}_{\rho}(-h,0;H)}\bigr]\text{.}
  \end{equation*}
  The estimate $\norm{S_{\rho}}\leq \frac{1}{c}$ from \cref{th:Picard} yields for some finite $C>0$,
  \begin{equation*}
    \norm{u-v}_{L_{2,\rho}(0,T;H)} \leq \tfrac{C}{\sqrt{\rho}} \bigl[\norm{u-v}_{L_{2,\rho}(0,T;H)}+ \tilde{C}(\rho)\norm{\Phi-\Psi}_{H^{2}_{\rho}(-h,0;H)}\bigr]\text{.}
  \end{equation*}
  This takes care of the nonlinearity.
  \item For the additional terms in \cref{eq:SimplMatLawIntIns} we have to estimate:
  \begin{align*}
    &\norm[\big]{S_{\rho}\bigl[M_{0}\Phi'(0) \!+ \!(M_{1} \!+ \!A)\bigl(\Phi(0) \!+ \!I_{\rho}\Phi'(0)\bigr) \!- \!M_{0}\Psi'(0) \!- \!(M_{1}\! + \!A)\bigl(\Psi(0) \!+ \!I_{\rho}\Psi'(0)\bigr)\bigr]\!}_{\!L_{2,\rho}}^{2}\\
    &\leq \norm{S_{\rho}}^{2}\!\!\medint\int_{0}^{T} \!\norm{M_{0}(\Phi'(0) \!- \!\Psi'(0)) \!+ \!(M_{1}\!+\!A)(\Phi(0)\!-\!\Psi(0)) \!+ \!t(\Phi'(0)\!-\!\Psi'(0))}_{H}^{2}\e^{-2\rho t}\dx[t]\\
    &\leq \tfrac{4}{c^{2}}\bigl[ \norm{M_{0}}^{2}\norm{\Phi'-\Psi'}_{\infty}^{2}\tfrac{1}{2\rho} + \norm{M_{1}}^{2}\norm{\Phi-\Psi}_{\infty}^{2}\tfrac{1}{2\rho} \\
    &\qquad\quad+ \norm{A(\Phi(0)-\Psi(0))}_{H}^{2}\tfrac{1}{2\rho} + \tfrac{1}{4\rho^{3}}\norm{\Phi'-\Psi'}_{\infty}^{2}\bigr]\text{.}
  \end{align*}
  The term $\norm{A(\Phi(0)-\Psi(0))}_{H}$ can be estimated with $\norm{\Phi(0)-\Psi(0)}_{\operatorname{dom}(A)}$ and thus we have proven a sufficient estimate for $\norm{w_{u}-w_{v}}_{L_{2,\rho}(0,T;H)}$.
  \end{enumerate}
  To estimate the difference of $u - v = w_{u} - w_{v} + I_{\rho}\widehat{\Phi'} - I_{\rho}\widehat{\Psi'}$, it remains to estimate the difference $\norm{I_{\rho}\widehat{\Phi'} - I_{\rho}\widehat{\Psi'}}_{L_{2,\rho}(-h,T;H)}^{2}$, but this can be done as above. Using the previous estimates and the fact that $\norm{\Phi-\Psi}_{H^{2}_{\rho}(-h,0;H)}\leq \e^{\rho h}\norm{\Phi -\Psi}_{H^{2}(-h,0;H)}$, we have for $\rho \gg 0$ a constant $\hat{C}(\rho)>0$ such that:
  \begin{equation*}
    \norm{u-v}_{L_{2,\rho}(-h,T;H)} \leq \e^{\rho h}\hat{C}(\rho)\norm{\Phi-\Psi}_{H^{2}(-h,0;H)} + \norm{\Phi(0)-\Psi(0)}_{\operatorname{dom}(A)}\text{.}
  \end{equation*}
  The equivalence of the norms $\norm{\argdot}_{L_{2}(0,T;H)}$ and $\norm{\argdot}_{L_{2,\rho}(0,T;H)}$ produces the desired constant.
\end{proof}

\section{Examples}
\label{sec:Examples}
In this section, we visit examples highlighting the applicability of our results. In the process thereof, we will make use of several (weak) differential operators and their symmetrized versions. We provide their definitions in \cref{sec:DiffOp} and take the notation introduced there for granted subsequently. Furthermore, whenever convenient, we will assume that the prehistory $\Phi$ (and if necessary, its derivatives) vanishes at $-h$, appealing to \cref{rmk:PrehistoryVanishes}. We begin with three classical PDEs in which $A$ is a skew-selfadjoint operator.

\subsection{Heat Equation}
\label{subsec:Heat}
Let $\Omega\subseteq \mathbb{R}^{d}$ be open and let $a\colon \, \Omega\rightarrow \mathbb{C}^{d\times d}$ be measurable, bounded and $\operatorname{Re}a\geq c>0$. The heat equation with state-dependent delay is given as
\begin{align}
  \label{eq:Heat}
  \begin{aligned}
    \bigl(\partial_{\rho}-\operatorname{div}_{0}a\operatorname{grad}\bigr)u&=F\bigl(\argdot,u_{(\argdot)}\bigr)\text{,}\\
    u_{0}&=\Phi\text{.}
  \end{aligned}
\end{align}
Here, we simply took the evolutionary equation formulation of the heat equation from \cite[thm.~6.2.4]{Waurick2022} and added a state-dependent right-hand side. Structurally, such a system has been studied e.g. in \cite{Khusainov2009} (although for constant delay), \cite{Camacho2018} (in the context of control theory), \cite{Braik2019} (for classical solutions) and \cite{Schnaubelt2004} (for a semigroup approach).\\
Employing the ansatz $u = v + I_{\rho}\widehat{\Phi'}$ and immediately incorporating the prehistory into the equation in the style of \cref{eq:EvolProblem} in form of a function $f$ on the right-hand side, we can write the heat equation as a system in the following form:
\begin{equation}
  \label{eq:HeatEq}
  \biggr[\partial_{\rho} \begin{pmatrix}1 & 0 \\ 0 & 0\end{pmatrix} + \begin{pmatrix}0 & 0 \\ 0 & a^{-1}\end{pmatrix}
  + \begin{pmatrix}0 & \operatorname{div}_{0} \\ \operatorname{grad} & 0\end{pmatrix} \biggl] \begin{pmatrix}v \\ w\end{pmatrix} = \begin{pmatrix}f + F\bigl(\argdot, (v+I_{\rho}\widehat{\Phi'})_{(\argdot)}\bigr) \\ 0\end{pmatrix}\text{.}
\end{equation}
This system fits the setting of evolutionary equations (cf.~\cite[ch.~6.2]{Waurick2022}); in fact, we have a simple material law, that in addition satisfies the assumptions of \cref{th:parabolicRegularity}. We can prove well-posedness of the heat equation in both cases $\widetilde{F}(v) = I_{\rho}\widetilde{G}(v)$ and $\widetilde{F}(v) = \widetilde{G}(I_{\rho}v)$, but we will focus on the latter case.

\subsubsection{Supposing $\widetilde{F}(v) = \widetilde{G}(I_{\rho}v)$}
If we simply suppose that the right-hand side is of the desired form, we can make use of the parabolic structure, cf.~\cref{def:ParabolicPair}:
\begin{theorem}
  \Cref{eq:HeatEq} is (locally) solvable under assumption \ref{ass:A} i) and iii) on $G$ and assumption \ref{ass:B} on $\Phi$, provided that $f\in H^{1}_{\rho}(0,\infty;H)$ and $f(0) = -G(0,\Phi)$.
\end{theorem}
\begin{proof}
  We observe that the left-hand side of \cref{eq:HeatEq} adheres to the conditions of \cref{th:parabolicRegularity} and we can therefore simply apply \cref{th:localExistParabolic}.
\end{proof}
We remark that it is critical in our approach to delay equations that the equation is structurally regularity increasing. We can demand this by virtue of a right-hand side with built-in increase of regularity or an equation that structurally produces more regular solutions, e.g., by virtue of parabolic regularity.

\subsubsection{Generating the structure}
Alternatively, we can also generate the desired formulation from \cref{eq:Heat}, showcasing that the assumptions on the form of the right-hand side are not as big of a restriction as they might seem at first glance. To this end, we make the ansatz $u = I_{\rho}(v + \widehat{\Phi'})$ for a $v \in L_{2,\rho}(0,\infty;H)$ and obtain from \cref{eq:Heat}:
\begin{equation*}
  \bigl(1 - \operatorname{div}_{0}a \operatorname{grad}I_{\rho}\bigr)v + \Phi'(0) - \operatorname{div}_{0}a \operatorname{grad}\bigl(\Phi(0) + I_{\rho}\Phi'(0)\bigr)
  = F\bigl(\argdot, (I_{\rho}(v+\widehat{\Phi'}))_{(\argdot)}\bigr)\text{.}
\end{equation*}
We substitute $w= I_{\rho}a \operatorname{grad}v \in H^{1}_{0,\rho}(0,\infty;H)$ and obtain:
\begin{equation*}
  v - \operatorname{div}_{0} w = -\Phi'(0) + \operatorname{div}_{0}a \operatorname{grad}\bigl(\Phi(0) + I_{\rho}\Phi'(0)\bigr) + F\bigl(\argdot, (I_{\rho}(v+\widehat{\Phi'}))_{(\argdot)}\bigr)\text{,}
\end{equation*}
which can be written as the following system:
\begin{align}
  \label{eq:HeatEq2}
  \begin{aligned}
  &\biggl[\mathring{\partial}_{\rho}
  \begin{pmatrix}
    0 & 0 \\ 0 & a^{-1}
  \end{pmatrix}
  +
  \begin{pmatrix}
    1 & 0 \\ 0 & 0
  \end{pmatrix}
  -
  \begin{pmatrix}
    0 & \operatorname{div}_{0} \\ \operatorname{grad} & 0
  \end{pmatrix}
  \biggr]\begin{pmatrix} v \\ w \end{pmatrix}\\
  &\qquad= \begin{pmatrix}
    - \Phi'(0) + \operatorname{div}_{0}a \operatorname{grad}\bigl(\Phi(0) + I_{\rho}\Phi'(0)\bigr) + F\bigl(\argdot, (I_{\rho}(v+\widehat{\Phi'}))_{(\argdot)}\bigr)\\ 0
  \end{pmatrix}\text{.}
  \end{aligned}
\end{align}
We therefore generated a formulation of the type (\ref{eq:FPPIntIns}), assuming we have a prehistory $\Phi \in H^{2}(-h,0;H)$ (in particular, the first derivative is bounded in $\norm{\argdot}_{\infty}$-norm). This formulation is reminiscent of \cite[cor.~4.2]{Waurick2017}, where stochastic problems were discussed.
We arrive at the following well-posedness theorem:
\begin{theorem}
  \Cref{eq:HeatEq2} is locally uniquely solvable in $H^{1}_{0,\rho}$ under assumption \ref{ass:A} on $F$ and assumption \ref{ass:B} on $\Phi$, provided that
  \begin{itemize}[leftmargin=4ex]
    \item $a$ is Hermitian (i.e., $a = a^{\ast}$),
    \item $\Phi \in H^{2}(0,\infty;H)$ with $\Phi(0),\Phi'(0) \in \operatorname{dom}\bigl(\operatorname{div}_{0} a \operatorname{grad}\bigr)$ and
    \item the consistency condition $\bigl(\operatorname{div}_{0} a \operatorname{grad}\bigr)\Phi(0) + F(0,\Phi) = \Phi'(0)$ holds.
  \end{itemize}
\end{theorem}
\begin{proof}
  We note that the left-hand side of \cref{eq:HeatEq2} fits the setting of evolutionary equations, because $a = a^{\ast}$ and we can simply apply \cref{th:localExistIns}, translating the consistency condition there:
  \begin{equation*}
    -\Phi'(0) + \bigl(\operatorname{div}_{0} a \operatorname{grad}\bigr)\Phi(0)
    + F(0,\Phi) = 0\text{.}\qedhere
  \end{equation*}
\end{proof}

\subsection{Wave Equation}
\label{subsec:Wave}
Let $\Omega\subseteq \mathbb{R}^{d}$ be open and let $T\colon\,\Omega\rightarrow \mathbb{R}^{d\times d}$ be bounded, measurable and Hermitian, satisfying $T\geq c>0$. The wave equation is given as
\begin{align}
  \begin{aligned}
    \label{eq:WaveNaive}
    \bigl(\partial_{\rho}^{2}-\operatorname{div} T \operatorname{grad}_{0}\bigr)u&=F\bigl(\argdot,u_{(\argdot)}\bigr)\text{,}\\
    u_{(0)}&=\Phi\text{.}
  \end{aligned}
\end{align}
Again, we use the formulation from \cite[thm.~6.2.6]{Waurick2022} for the wave equation in evolutionary equations and supplant the conventional righ-hand side with one that includes state-dependent delay as before. This system fits the formulation of many articles employing the theory of semigroups (for a very general resource we reference \cite{Batkai2005} or refer to \cref{subsec:SG2ndOrder} for an example), but there are other sources as well, e.g. \cite{Jornet2021} (where an actual solution is calculated) or \cite{Xiong2023} (in the context of attractors). It should be pointed out that the delayed wave equation has also been studied extensively by C.~Pignotti and S.~Nicaise, although in the context of delayed feedback on the boundary of a domain (e.g. \cite{Pignotti2006}) or with the delay in the derivative of the state variable (e.g. \cite{Pignotti2015}).\\
Making the usual ansatz $u = v + I_{\rho}\widehat{\Phi'}$ with $v \in H^{1}_{0,\rho}(0,\infty;H)$, we can model the equation as the following system, where we include the prehistory as a function $f$ into the equation as in \cref{eq:EvolProblem}:
\begin{equation}
  \label{eq:WaveEq}
  \biggl[\partial_{\rho}
  \begin{pmatrix}
    1 & 0 \\ 0 & T^{-1}
  \end{pmatrix}
  - \begin{pmatrix}
    0 & \operatorname{div} \\ \operatorname{grad}_{0} & 0
  \end{pmatrix}\biggr]
  \begin{pmatrix}v \\ w\end{pmatrix}
  = \begin{pmatrix}
    f + F\bigl(\argdot, (v + I_{\rho}\widehat{\Phi'})_{(\argdot)}\bigr) \\ 0
  \end{pmatrix}\text{.}
\end{equation}
Again, this equation fits the setting of evolutionary equations (cf. \cite[ch.~6.2]{Waurick2022}) and the material law is a simple material law. For both studied cases $\widetilde{F}(v) = I_{\rho}\widetilde{G}(v)$ and $\widetilde{F}(v) = \widetilde{G}(I_{\rho}v)$ we obtain a well-posedness theorem:
\begin{theorem}
  Let $f = I_{\rho}g$ for some $g\in H^{1}_{\rho}(0,\infty;H)$ and $F\bigl(\argdot, (v+I_{\rho}\widehat{\Phi'})_{(\argdot)}\bigr) = I_{\rho}G\bigl(\argdot, (v+I_{\rho}\widehat{\Phi'})_{(\argdot)}\bigr)$ and let $G$ satisfy assumption \ref{ass:A} and $\Phi$ assumption \ref{ass:B}. Then \cref{eq:WaveEq} has a unique local solution provided that $g(0) = -G(0,\Phi)$.
\end{theorem}
\begin{proof}
Immediate from \cref{th:localExistOut}.
\end{proof}
\begin{theorem}
  Let $f \in H^{1}_{0,\rho}(0,\infty;H)$ and $F\bigl(\argdot, (v+I_{\rho}\widehat{\Phi'})_{(\argdot)}\bigr) = G\bigl(\argdot, (I_{\rho}v+I_{\rho}\widehat{\Phi'})_{(\argdot)}\bigr)$ and let $G$ satisfy assumption \ref{ass:A} and $\Phi$ assumption \ref{ass:B}. Then \cref{eq:WaveEq} has a unique local solution provided that $f(0) = -G(0,\Phi)$.
\end{theorem}
\begin{proof}
Immediate from \cref{th:localExistIns}.
\end{proof}
Because we are working with a simple material law, we can also be more explicit and use the following approach, generating the case $\widetilde{F}(v) = I_{\rho}\widetilde{G}(v)$:\\
We make the ansatz $u = v + I_{\rho}I_{\rho}\widehat{\Phi''}$ with $v\in H^{2}_{0,\rho}(0,\infty;H)$ and obtain from \cref{eq:WaveEq}
\begin{equation*}
  \bigl(\mathring{\partial}_{\rho}^{2}- \operatorname{div}T\operatorname{grad}_{0}\bigr)v
  = - \Phi''(0) + \operatorname{div}T\operatorname{grad}_{0}I_{\rho}I_{\rho}\widehat{\Phi''} + F\bigl(\argdot,(v+I_{\rho}I_{\rho}\widehat{\Phi''})_{(\argdot)}\bigr)\text{.}
\end{equation*}
Integration over the entire equation gives
\begin{equation*}
  \bigl(\mathring{\partial}_{\rho}- \operatorname{div}T\operatorname{grad}_{0}I_{\rho}\bigr)v
  = - I_{\rho}\Phi''(0) + \operatorname{div}T\operatorname{grad}_{0}I_{\rho}I_{\rho}I_{\rho}\widehat{\Phi''} + I_{\rho}F\bigl(\argdot,(v+I_{\rho}I_{\rho}\widehat{\Phi''})_{(\argdot)}\bigr)\text{.}
\end{equation*}
Substituting $w=T\operatorname{grad}_{0}I_{\rho}v$, we obtain the following system:
\begin{align*}
  \mathring{\partial}_{\rho}v-\operatorname{div}w&= I_{\rho}\bigl[-\Phi''(0) + \operatorname{div}T\operatorname{grad}_{0}I_{\rho}I_{\rho}\widehat{\Phi''}\bigr] + I_{\rho}F\bigl(\argdot,(v+I_{\rho}I_{\rho}\widehat{\Phi''})_{(\argdot)}\bigr)\text{,}\\
  \mathring{\partial}_{\rho}T^{-1}w&=\operatorname{grad}_{0}v\text{.}
\end{align*}
Written in matrix form, the system reads:
\begin{align}
  \label{eq:WaveEq2}
  \begin{aligned}
  &\biggl[\mathring{\partial}_{\rho} \begin{pmatrix}
    1 & 0 \\ 0 & T^{-1}
  \end{pmatrix}
    - \begin{pmatrix}
      0 & \operatorname{div} \\ \operatorname{grad}_{0} & 0
    \end{pmatrix}\biggr]
    \begin{pmatrix} v \\ w \end{pmatrix}\\
  &\qquad\qquad= \begin{pmatrix}
    I_{\rho}\bigl[-\Phi''(0) + \operatorname{div}T\operatorname{grad}_{0}I_{\rho}I_{\rho}\widehat{\Phi''}\bigr]
    + I_{\rho}F\bigl(\argdot,(v+I_{\rho}I_{\rho}\widehat{\Phi''})_{(\argdot)}\bigr) \\ 0
  \end{pmatrix}\text{.}
  \end{aligned}
\end{align}
The left-hand side of this equation is unchanged from before and satisfies the requirements of Picard's \cref{th:Picard}. We thus produced an equation fitting the formulation \eqref{eq:FPPIntOut}. We also point out that this version of the delayed wave equation is reminiscent to the formulation in \cite[thm.~5.16]{Waurick2018}, where stochastic problems were discussed.
\begin{theorem}
  \Cref{eq:WaveEq2}  admits a unique local solution under assumption \ref{ass:A} on $F$ and assumption \ref{ass:B} on $\Phi$, provided that
  \begin{itemize}[leftmargin=4ex]
    \item $\Phi \in H^{2}(-h,0;H)$ with $\Phi(0),\Phi'(0) \in \operatorname{dom}\bigl(\operatorname{div}T\operatorname{grad}_{0}\bigr)$ and
    \item the consisteny condition $\Phi''(0) - \operatorname{div}T\operatorname{grad}_{0}\Phi(0) = F(0,\Phi)$ holds.
  \end{itemize}
\end{theorem}
\begin{proof}
  The left-hand side of \cref{eq:WaveEq2} satisfies the assumptions of \cref{th:Picard} and we can simply appeal to \cref{th:localExistOut}, translating the consistency condition there:
  \begin{align*}
    F(0,\Phi) &= F(0, \bigr(I_{\rho}I_{\rho}\widehat{\Phi''}\bigr)_{(0)}) = \Phi''(0) - \operatorname{div}T \operatorname{grad}_{0}(I_{\rho}I_{\rho}\widehat{\Phi''}\bigr)(0)\\
    &= \Phi''(0) - \operatorname{div}T\operatorname{grad}_{0}\Phi(0)\text{.}\qedhere
  \end{align*}
\end{proof}
It is also possible to generate the case $\widetilde{F}(v) = \widetilde{G}(I_{\rho}v)$, cf.~\cref{subsec:SG2ndOrder}.

\subsection{Maxwell's equations}
\label{subsec:Maxwell}
Let $\Omega\subseteq \mathbb{R}^{3}$ be open and let $j_{c}\colon\,\mathbb{R}\times \Omega \rightarrow \mathbb{R}^{3}$. Let further $\epsilon\colon \,\Omega\rightarrow \mathbb{R}^{3\times 3}$ and $\mu\colon\,\Omega\rightarrow \mathbb{R}^{3\times 3}$ be bounded, measurable and symmetric. Also let $\sigma\colon\,\Omega\rightarrow \mathbb{R}^{3\times 3}$ be bounded and measurable. Then the delayed Maxwell's equations are given by
\begin{alignat*}{4}
  &\partial_{\rho}\epsilon E+\sigma E &-\operatorname{curl}H&=F_{1}\bigl(\argdot,(E+\Phi)_{(\argdot)},(H+\Psi)_{(\argdot)}\bigr)\text{,}\\
  &\partial_{\rho}\mu H &+ \operatorname{curl}_{0}E&=F_{2}\bigl(\argdot,(E+\Phi)_{(\argdot)},(H+\Psi)_{(\argdot)}\bigr)\text{,}\\
  &&E_{(0)}&=\Phi\text{,}\\
  &&H_{(0)}&=\Psi\text{,}
\end{alignat*}
where again, we took the evolutionary equation formulation of Maxwell's equation (cf.~\cite[thm.~6.2.8]{Waurick2022}) and supplanted the right-hand side with functions incorporating state-dependent delay. We note that Maxwell's equations are usually completed with conditions on the divergence of $H$ and $E$. It is an easy calculation to show, that the system
\begin{equation*}
  \biggl[\partial_{\rho}
  \begin{pmatrix}
    \epsilon & 0 \\ 0 & \mu
  \end{pmatrix}
  + \begin{pmatrix}
    \sigma & 0 \\ 0 & 0
  \end{pmatrix}
  + \begin{pmatrix}
    0 & - \operatorname{curl} \\ \operatorname{curl}_{0} & 0
  \end{pmatrix}\biggr]
  \begin{pmatrix}E \\ H\end{pmatrix}
  = \begin{pmatrix}j_{0} \\ 0\end{pmatrix}\text{,}
\end{equation*}
where $j_{0}$ is a given electrical current density, is equivalent to the ``standard'' Maxwell's equations due to the prescribed initial conditions (in the case without delay). The calculations can be found in \cite[rem.~6.2.9]{Waurick2022}. In the same source, it is also verified, that the left-hand side of the above system fits the theory of evolutionary equations and again we even have a simple material law.\\
This time, we take the more abstract point of view and investigate the problem in the formulation (\ref{eq:EvolProblem}), incorporating the prehistories into the equation and demanding that the functions $F_{1},F_{2}$ are of the same form as in (\ref{eq:FPPIntIns}), that is we study
\begin{align}
  \begin{aligned}
  \label{eq:Maxwell}
    \biggl[\partial_{\rho}\begin{pmatrix}
      \epsilon & 0 \\ 0 & \mu
    \end{pmatrix}
    &+ \begin{pmatrix}
        \sigma & 0 \\ 0 & 0
      \end{pmatrix}
    + \begin{pmatrix}
        0 & - \operatorname{curl} \\ \operatorname{curl}_{0} & 0
      \end{pmatrix}\biggr]
      \begin{pmatrix}E \\ H\end{pmatrix}\\
    &\quad= \begin{pmatrix}
      f_{1} + G_{1}\bigl(\argdot, (I_{\rho}(E + \widehat{\Phi'}))_{(\argdot)}, (I_{\rho}(H + \widehat{\Psi'}))_{(\argdot)}\bigr)\\
      f_{2} + G_{2}\bigl(\argdot, (I_{\rho}(E + \widehat{\Phi'}))_{(\argdot)}, (I_{\rho}(H + \widehat{\Psi'}))_{(\argdot)}\bigr)
    \end{pmatrix}\text{.}
  \end{aligned}
\end{align}
This formulation is also reminiscent of \cite[def.~4.13]{Waurick2017}, where stochastic problems were discussed. In that article, the authors arrived at a system with a right-hand side in the style of (\ref{eq:FPPIntIns}). By appealing to \cref{th:localExistIns} we obtain:
\begin{theorem}
  \Cref{eq:Maxwell} is locally uniquely solvable under assumption \ref{ass:A} on $G_{1},G_{2}$ and assumption \ref{ass:B} on $\Phi,\Psi$, provided that $f_{1},f_{2}\in H^{1}_{\rho}(0,\infty;H)$ satisfy $f_{1}(0) = -G_{1}(0,\Phi,\Psi)$ and $f_{2}(0) = -G_{2}(0,\Phi,\Psi)$ respectively.
\end{theorem}
\begin{remark}[Mixed-type problems]
  One strength of the theory of evolutionary equations is the ability to treat mixed-type problems: For instance, on a split domain one may have hyperbolic behaviour on one part, on the other part parabolic behaviour, with both parts linked together by interface conditions. An example is the following eddy current system
  \begin{equation*}
    \biggl[\partial_{t}\begin{pmatrix}
    \epsilon & 0\\0 & \mu
    \end{pmatrix}
    + \begin{pmatrix}
    \sigma & 0\\0 & 0
    \end{pmatrix}
    + \begin{pmatrix}
    0 & -\operatorname{curl}\\\operatorname{curl}_{0} & 0
    \end{pmatrix}\biggr]
    \begin{pmatrix}E\\H\end{pmatrix}
    = \begin{pmatrix}-J\\0\end{pmatrix}
  \end{equation*}
  on the space $(0,1)^{3} \eqcolon \Omega = C\dot{\cup}S$, where $\mu \equiv 1$, $\epsilon=1 - \chi_{C}$, $\sigma=10^{3}\chi_{C}$, $J(x,t)=\sin(\pi t)\chi_{S}(x)$. This setting still fits the framework of evolutionary equations. For a precise treatment of mixed-type eddy current problems in the setting of evolutionary problems (without delay), we refer to \cite{Waurick2021}.\\
  From our perspective of (generalized) initial value problems, such an increase in difficulty provides no problems for the solution theory (as long as we remain in the setting of evolutionary equations). Hence, we obtain well-posedness of the system above for the case of a state-dependent delay on the right-hand side in the vein of \cref{th:localExistOut} or \cref{th:localExistIns}.
\end{remark}
The first three subsections were devoted to some classical PDEs to show the range of examples for our theory. The next two examples are dedicated to examples from semigroup theory. We shall demonstrate exactly how our theory can handle some typical problems.

\subsection{Reaction-Diffusion equations}
\label{subsec:SG1stOrder}
We study the following problem from \cite{Chueshov2015p2}:
\begin{align}
  \label{eq:ReactDiffNaive}
  \begin{aligned}
  u'(t)+Au(t)+F\bigl(u_{(t)}\bigr)+G(u(t))&=h\text{,}\\
  u_{(0)}&=\Phi\text{,}
  \end{aligned}
\end{align}
where $A$ is a strictly positive-definite linear operator and is densely defined in a separable Hilbert space $H$; $F$ is of the form
\begin{equation*}
  F\bigl(u_{(t)}\bigr)=F_{0}\bigl(u(t-\eta (u_{(t)}))\bigr)
\end{equation*}
for some globally Lipschitz-continuous $F_{0}\colon\,H_{\alpha}\to H_{\alpha}$ for $\alpha = 0$ and $\alpha = -\frac{1}{2}$, where $H_{\alpha}$ denotes the completion of $\operatorname{dom}(A^{\alpha})$ w.r.t.\ $\norm{A^{\alpha}\argdot}$ in $H$ and a globally Lipschitz $\eta\colon \mathcal{C}(-r,0;H)\to [0,r]$; $h\in H$ is a constant (in time) and $G\colon H_{\frac{1}{2}}\to H$ is locally Lipschitz, satisfying some additional assumptions, we refer to \cite[ass.~2.1]{Chueshov2015p2} for details.\\
The main example/application in \cite{Chueshov2015p2} is the problem
\begin{equation*}
  \partial_{t}u(t,x) - \Delta u(t,x) + b\bigl(B[u\bigl(t-\eta (u_{(t)}),\argdot\bigr)](x)\bigr) + g\bigl(u(t,x)\bigr) = h(x)\text{,}
\end{equation*}
for a bounded linear $B\in \mathcal{L}\bigl(L^{2}(\Omega)\bigr)$ on a bounded domain $\Omega$; $b\colon \mathbb{R}\to \mathbb{R}$ and $\eta \colon \mathcal{C}\bigl(-h,0;L^{2}(\Omega)\bigr) \to [0,h]$ are Lipschitz-continuous; $g\in \mathcal{C}^{1}$ represents a nonlinear non-delayed reaction term and $h$ is a source term, for instance $h(x) = c_{1}x\e^{-c_{2}x}$ with $c_{1},c_{2}>0$ (motivated from population dynamics). For further motivation and specific examples we refer to the references in \cite{Chueshov2015p2}.\\
For our purposes, we start from the basic IVP (\ref{eq:ReactDiffNaive}) and  make the ansatz for a formulation in the spirit of \cref{eq:FPPIntIns}: $u=I_{\rho}(v+\widehat{\Phi'})$ with $v \in L_{2,\rho}(0,\infty;H)$. For $\phi \in \mathcal{C}^{1}(-h,0;H)$ we have $u'(t) = v(t) + \Phi'(0)$ for $t>0$ and hence the differential equation reads:
\begin{equation*}
  v + \Phi'(0) + AI_{\rho}(v + \widehat{\Phi'}) + F\bigl(I_{\rho}(v+\widehat{\Phi'})_{(\argdot)}\bigr)+G\bigl(I_{\rho}(v+\widehat{\Phi'})\bigr)=h\text{.}
\end{equation*}
Using the substitution $q\coloneq I_{\rho}A^{\nicefrac{1}{2}}v$, we obtain:
\begin{align}
  \begin{aligned}
  \label{eq:ReactionDiffusion}
  &\biggl[\mathring{\partial}_{\rho}\begin{pmatrix}
    1 & 0 \\ 0 & 0
  \end{pmatrix}
  + \begin{pmatrix}
    0 & 0 \\ 0 & 1
  \end{pmatrix}
  + \begin{pmatrix}
    0 & -A^{\nicefrac{1}{2}} \\ A^{\nicefrac{1}{2}} & 0
  \end{pmatrix}\biggr]
  \begin{pmatrix}q \\ v\end{pmatrix}\\
      &\quad= \begin{pmatrix}
        0 \\ -\Phi'(0) - AI_{\rho}\Phi(0) - F\bigl((I_{\rho}(v+\widehat{\Phi'}))_{(\argdot)}\bigr) -G\bigl(I_{\rho}(v+\widehat{\Phi'})\bigr) + h
      \end{pmatrix}\text{.}
  \end{aligned}
\end{align}
The left-hand side satisfies the assumptions of Picard's \cref{th:Picard}. We can now appeal to our well-posedness theory from \cref{sec:SolThy}:
\begin{theorem}
  Under assumption \ref{ass:A} on $F,G$ and assumption \ref{ass:B} on $\Phi\in \mathcal{C}^{1}(-h,0;H)$, \cref{eq:ReactionDiffusion} admits a unique local solution, provided that
  \begin{itemize}[leftmargin=4ex]
    \item $\Phi(0)\in \operatorname{dom}(A)$,
    \item $h\in H^{1}_{\rho}(0,\infty;H)$ and\footnote{In \cref{eq:ReactDiffNaive} (corresponding to the formulation from \cite{Chueshov2015p2}) it is assumed that $h$ is constant (in time), but we do not require this assumption for well-posedness. Of course, due to the exponential weight, constants are integrable functions in $L_{2,\rho}(0,\infty;H)$.}
    \item the consistency condition $\Phi'(0) + A\Phi(0) + F(\Phi) + G(\Phi(0)) = h(0)$ holds.
  \end{itemize}
\end{theorem}
\begin{proof}
  We note that the left-hand side of \cref{eq:ReactionDiffusion} fits the setting of evolutionary equations, i.e., satisfies the assumptions of \cref{th:Picard}. We appeal to \cref{th:localExistIns}, translating the consistency condition there:
  \begin{align*}
    -\Phi'(0) -A\Phi(0) + h(0)
    &= F\Bigl(0, \displaystyle\medint\int_{-h}^{\argdot}\Phi'(s)\dx[s]\Bigr) + G(\Phi(0))\\
    &= F(0,\Phi) + G(\Phi(0))\text{.}\qedhere
  \end{align*}
\end{proof}
\begin{remark}[Comparison of results]
  \label{rmk:Comparison1stOrderEvP}
  We point out that the assumptions in \cite{Chueshov2015p2} differ slightly from ours. Here, the Hilbert space $H$ does not need to be separable (although that point hardly ever comes up in examples). Similarly to here, some notion of Lipschitz-continuity of the prehistory and $\Phi(0)\in \operatorname{dom}(A^{\nicefrac{1}{2}})$ are assumed as well, cf.~\cite[thm.~3.2]{Chueshov2015p2}. As far as Lipschitz-continuity of the right-hand side is concerned, it is verified in \cite[thm.~3.2]{Waurick2023} that functions of the form $F_{0}\bigl(u(t-\tau (u(t)))\bigr)$ produce an almost uniformly Lipschitz-continuous map; in particular under the assumptions of \cite{Chueshov2015p2}. The same holds for $G$, which easily follows from the fact that the subspaces $V_{\alpha}$ contain functions with bounded derivative and that we obtain solutions in $H^{1}_{0,\rho}$. Hence, our Lipschitz assumptions are weaker than considered in \cite{Chueshov2015p2}. In \cite[thm.~3.3]{Chueshov2015p2}, the authors prove existence and uniqueness of a local solution $u$ satisfying
  \begin{equation*}
    u' \in \mathcal{C}(0,T;H_{\frac{1}{2}})\cap L_{2}(0,T;H)
  \end{equation*}
  and they include a result on continuous dependence on initial datum, cf.~\cite[prop.~3.4]{Chueshov2015p2}. For time continuity of the evolution of the prehistory, the authors consider initial values $\Phi \in \mathcal{C}^{1}(-h,0;H_{-\frac{1}{2}})$ satisfying a consistency condition
  \begin{equation*}
    \Phi'(0) + A\Phi(0) + F(\Phi) + G(\Phi(0)) = 0\text{,}
  \end{equation*}
  cf.~\cite[rem.~3.6]{Chueshov2015p2}. We point out that solutions in this formulation from \cite{Chueshov2015p2} can only be continuous across $0$ if $h=0$.
\end{remark}

\subsection{Second order evolution problems}
\label{subsec:SG2ndOrder}
Next, we consider a second-order problem from \cite{Chueshov2015p1}:
\begin{align*}
  u''(t)+ku'(t)+Au(t)+F(u(t))+M(u_{(t)})&=0\text{,}\\
  u_{(0)}&=\Phi\text{,}
\end{align*}
where $A$ is a positive operator with discrete spectrum on a separable Hilbert space $H$; $k$ is a constant; $F\colon \operatorname{dom}(A^{\nicefrac{1}{2}})\to H$ is locally Lipschitz-continuous and so is $M\colon W\to H$, where
\begin{equation*}
  W \coloneq \mathcal{C}\bigl(-h,0;\operatorname{dom}(A^{\nicefrac{1}{2}})\bigr)\cap \mathcal{C}^{1}(-h,0;H)\text{.}
\end{equation*}
The main example/application of the model in \cite{Chueshov2015p1} is a nonlinear plate equation of the form
\begin{equation*}
  \partial_{t}^{2}u(t,x) + k\partial_{t}u(t,x) + \Delta^{2} u(t,x) + F(u(t,x)) + au\bigl(t - \tau(u(t)),x\bigr) = 0
\end{equation*}
in a smooth bounded domain $\Omega \subseteq \mathbb{R}^{2}$ with suitable boundary conditions on $\partial\Omega$; $\tau$ is a mapping defined on solutions with values in some interval $[0,h]$, satisfying a Lipschitz-like condition (cf.~\cite[ass.~M4]{Chueshov2015p1}); $k$ and $a$ are constants. The plate is assumed to be placed on some foundation; the term $au\bigl(t - \tau(u(t)),x\bigr)$ models effects of the foundation with a delayed response; $F$ is a potentially nonlinear force. The model in particular covers the wave equation with state-dependent delay. For more details and applications we refer to \cite{Chueshov2015p1} and the references therein.\\
Since we want to solve a second order equation, we apply the usual ansatz to the derivative of $u$: $\partial_{\rho}u = I_{\rho}(v+\widehat{\Phi''})$, assuming that $\Phi \in \mathcal{C}^{2}(-h,0;H)$. This means $u = I_{\rho}I_{\rho}(v + \widehat{\Phi''})$ and consequently for $t>0$ hold:
\begin{align*}
  u''(t) &= v(t) + \Phi''(0)\text{,}\\
  u'(t) &= \medint\int_{0}^{t}v(s)\dx[s] + \medint\int_{-h}^{0}\Phi''(s)\dx[s] + \medint\int_{0}^{t}\Phi''(0)\dx[s] = \medint\int_{0}^{t}v(s)\dx[s] + \Phi'(0) + t\Phi''(0)\text{,}\\
  u(t) &= I_{\rho}(I_{\rho}v)(s) + \medint\int_{-h}^{t}\medint\int_{-h}^{s}\widehat{\Phi''}(r)\dx[r]\dx[s]\\
         &= I_{\rho}(I_{\rho}v)(s) + \tfrac{t^{2}}{2}\Phi''(0) + t\Phi'(0) + \Phi(0)\text{.}\\
\intertext{The initial condition is trivially satisfied and substituting $b \coloneq I_{\rho}v$ and $c \coloneq I_{\rho}A^{\nicefrac{1}{2}}b$, we can rewrite these equations:}
  u''(t) &= (\mathring{\partial}_{\rho}b)(t) + \Phi''(0)\text{,}\\
  u'(t) &= b(t) + \Phi'(0) + t\Phi''(0)\text{,}\\
  u(t) &= (I_{\rho}b) (s) + \tfrac{t^{2}}{2}\Phi''(0) + t\Phi'(0) + \Phi(0) = I_{\rho}(b + \Phi'(0)) + I_{\rho}I_{\rho}\Phi''(0) + \Phi(0)\text{.}
\end{align*}
This allows us to write the differential equation as
\begin{align*}
  0&=\mathring{\partial}_{\rho}b + \Phi''(0) + k\bigl(b + \Phi'(0) + I_{\rho}\Phi''(0)\bigr) + A^{\nicefrac{1}{2}} I_{\rho}A^{\nicefrac{1}{2}}b \\
  &\quad + A\bigl(I_{\rho}I_{\rho}\Phi''(0) + I_{\rho}\Phi'(0) + \Phi(0)\bigr)
  + F\bigl(I_{\rho}(b + I_{\rho}\widehat{\Phi''})\bigr) + M\bigl((I_{\rho}(b + I_{\rho}\widehat{\Phi''}))_{(\argdot)}\bigr)\text{,}
\end{align*}
which produces the system
\begin{align}
  \begin{aligned}
  \label{eq:2ndOrderEvEq}
  \biggl[\mathring{\partial}_{\rho}\begin{pmatrix}
    1 & 0 \\ 0 & 1
  \end{pmatrix}
  &+ \begin{pmatrix}
    k & 0 \\ 0 & 0
  \end{pmatrix}
  + \begin{pmatrix}
    0 & A^{\nicefrac{1}{2}}\\ -A^{\nicefrac{1}{2}} & 0
  \end{pmatrix}\biggr]
  \begin{pmatrix} b \\ c \end{pmatrix}\\
  &\quad=\begin{pmatrix}
    f - F\bigl(I_{\rho}(b + I_{\rho}\widehat{\Phi''})\bigr) - M\bigl((I_{\rho}(b + I_{\rho}\widehat{\Phi''}))_{(\argdot)}\bigr)\\
    0\end{pmatrix}\text{,}
  \end{aligned}
\end{align}
where
\begin{equation*}
  f \coloneq - \Phi''(0) - k\bigl(\Phi'(0) + I_{\rho}\Phi''(0)\bigr)
    - A\bigl(I_{\rho}I_{\rho}\Phi''(0) + I_{\rho}\Phi'(0) + \Phi(0)\bigr)\text{.}
\end{equation*}

We can prove well-posedness of the problem:
\begin{theorem}
  Under assumption \ref{ass:A} on $F,M$ and assumption \ref{ass:B} on $\Phi \in \mathcal{C}^{2}(-h,0;H)$, \cref{eq:2ndOrderEvEq} is locally uniquely solvable (in $H^{2}$), provided that the consistency condition $\Phi''(0)+k \Phi'(0)+A\Phi(0)+F(\Phi(0))+M(\Phi)=0$ holds.
\end{theorem}
\begin{proof}
  We note that the left-hand side of \cref{eq:2ndOrderEvEq} satisfies the assumptions of \cref{th:Picard} and therefore fits the framework of evolutionary equations. We can therefore simply appeal to \cref{th:localExistIns}, translating the consistency condition there, which is:
  \begin{align*}
    0 &= f(0) - F\Bigl(\displaystyle\medint\int_{-h}^{0}\underbrace{b(s)}_{=0}+ \Phi'(s)\dx[s]\Bigr) - M\Bigl(\Bigl(\displaystyle\medint\int_{-h}^{\argdot}\underbrace{b(s)}_{=0}+\Phi'(s)\dx[s]\Bigr)_{(0)}\Bigr)\\
    &= -\Phi''(0) - k \Phi'(0) - A\Phi(0) - F(\Phi(0)) - M(\Phi)\text{.}\qedhere
  \end{align*}
\end{proof}
\begin{remark}[Comparison of results]
  In \cite{Chueshov2015p1}, the authors consider prehistories in
  \begin{equation*}
    W = \mathcal{C}\bigl(-h,0;\operatorname{dom}(A^{\nicefrac{1}{2}})\bigr)\cap \mathcal{C}^{1}(-h,0;H)\text{,}
  \end{equation*}
  which is a reasonably similar assumption compared to ours and prove existence of a local mild solution
  \begin{equation*}
    u\in \mathcal{C}\bigl(-h,T;\operatorname{dom}(A^{\nicefrac{1}{2}})\bigr)\cap \mathcal{C}^{1}(-h,T;H)\text{,}
  \end{equation*}
  cf.~\cite[prop.~2.3]{Chueshov2015p1}. Under additional regularity assumptions on $F$ and a growth bound on $M$, they can prove continuous dependence on initial datum and an energy equation, cf.~\cite[thm.~2.4]{Chueshov2015p1}. For strong solutions, that is, solutions $u$ satisfying
  \begin{equation*}
    u \in \mathcal{C}\bigl(0,\infty;\operatorname{dom}(A)\bigr)\text{,} \quad
    u' \in \mathcal{C}\bigl(0,\infty;\operatorname{dom}(A^{\nicefrac{1}{2}})\bigr)\text{,} \quad
    u''\in \mathcal{C}(0,\infty;H)\text{,}
  \end{equation*}
  the authors require $M$ to be {\em ``locally almost Lipschitz''}, i.e.,
  \begin{equation*}
    \forall \rho > 0 \exists C_{\rho}>0 \forall \varphi,\psi \in W \,\,\text{with}\,\,\norm{\varphi}_{W},\norm{\psi}_{W}\leq \rho\colon \,\norm{M(\varphi)\!-\!M(\psi)} \leq C_{\rho}\norm{\varphi \!- \!\psi}\text{,}
  \end{equation*}
  as well as additional regularity on $F$ and a consistency condition of the form
  \begin{equation*}
    \Phi''(0) + k\Phi'(0) + A \Phi(0) + F(\Phi(0)) + M(\Phi) = 0\text{,}
  \end{equation*}
  cf.~\cite[rem.~2.7]{Chueshov2015p1}. The difference in the notion of Lipschitz-continuity of $F$ and $M$ is commensurable to the one in the previous example, discussed in \cref{rmk:Comparison1stOrderEvP}.
\end{remark}
Evolutionary equations are particularly useful when dealing with nonlocal operators and can accomodate operators that are nonlocal in time. The next two examples highlight this aspect:

\subsection{Fractional time-derivatives}
\label{subsec:FracDer}
The Fourier--Laplace transform (cf.~definition \eqref{def:FourierLaplace}) allows for a convenient definition of $\partial_{t}^{-\alpha}$ for $0<\alpha<1$ (on the entire real line) via
\begin{equation*}
  \partial_{\rho}^{-\alpha} \coloneq \mathcal{L}_{\rho}^{\ast} \tfrac{1}{(\iu \mathrm{m} +  \rho)^{\alpha}}\mathcal{L}_{\rho}\text{.}
\end{equation*}
Subsequently, one can define for $\alpha \in \mathbb{R}$: $\partial_{\rho}^{\alpha} \coloneq \partial_{\rho}^{\lceil \alpha \rceil }\partial_{\rho}^{\alpha - \lceil \alpha \rceil}$. For details see \cite{Trostorff2015} and also \cite{Trostorff2022}, where conditions are discussed, under which the Caputo fractional derivative or the Riemann--Liouville fractional derivative can be realized via the Fourier--Laplace transform.\\
Fractional time-derivatives have a number of applications, in particular in models for viscosity. We only mention two examples:
\begin{itemize}[leftmargin=4ex]
  \item Fractional Fokker--Planck equation (cf.~\cite[sec.~4.1]{Trostorff2015}).
        \begin{equation}
          \label{eq:FokkerPlank}
          \;\;\biggl[\partial_{\rho}^{1-\alpha} \begin{pmatrix}
            \kappa_{\alpha} & 0 \\ 0 & 0
          \end{pmatrix}
          \!+ \!\begin{pmatrix}
            \mu_{00} & \mu_{01} \\ \mu_{10} & \mu_{11}
          \end{pmatrix}
          \!+ \!\begin{pmatrix}
            0 & \operatorname{div} \\ \operatorname{grad}_{0} & 0
          \end{pmatrix} \biggr]
          \!\begin{pmatrix} u \\ v \end{pmatrix}
          \!= \!\begin{pmatrix}
            f_{1} \!+ \!F_{1}\bigl(\argdot,u_{(\argdot)}\bigr)\\
            f_{2} \!+ \!F_{2}\bigl(\argdot,v_{(\argdot)}\bigr)
          \end{pmatrix}\text{,}
        \end{equation}
        where we immediately plugged in a right-hand side in the vein of formulation \eqref{eq:EvolProblem}. For a description of the parameters involved in this system we refer to \cite{Trostorff2015}, for background on the equation to        \cite{Metzler1999}.
  \item Kelvin--Voigt model for visco-elastic solids (cf.~\cite[sec.~4.2]{Trostorff2015}).
        \begin{equation}
          \label{eq:KelvinVoigt}
          \biggl[\partial_{\rho} \begin{pmatrix}
            \eta & 0 \\ 0 & (C + D\partial_{\rho}^{\alpha})^{-1}
          \end{pmatrix}
          - \begin{pmatrix}
            0 & \operatorname{Div} \\ \operatorname{Grad}_{0} & 0
          \end{pmatrix}\biggr]
          \begin{pmatrix} u \\ T \end{pmatrix}
          = \begin{pmatrix}
            f + F\bigl(\argdot,u_{(\argdot)}\bigr) \\ 0
          \end{pmatrix}\text{,}
        \end{equation}
        where $\operatorname{Div}$ and $\operatorname{Grad}$ denote the symmetrized divergence and gradient respectively, which are defined in \cref{subsec:SymDiffOp}. For more details we refer to \cite{Waurick2013}, for background on the equation to \cite{Bodonolte2011}. Once again, we included a right-hand side befitting formulation \eqref{eq:EvolProblem}.
\end{itemize}
In the aforementioned reference \cite{Trostorff2015}, it is also argued that the left-hand sides of both examples fit the framework of evolutionary equations and hence our well-posedness results from \cref{sec:SolThy} apply, provided that the right-hand sides are of suitable form. For instance, in the variant (\ref{eq:FPPIntIns}), we can state the following by applying \cref{th:localExistIns}:

\begin{theorem}
  For a right-hand side
  \begin{equation*}
    \begin{pmatrix}
      f_{1} + G_{1}\bigl(\argdot, (I_{\rho}(u+\widehat{\Phi'}))_{(\argdot)}\bigr)\\
      f_{2} + G_{2}\bigl(\argdot, (I_{\rho}(v+\widehat{\Psi'}))_{(\argdot)}\bigr)
    \end{pmatrix}\text{,}
  \end{equation*}
  with $f_{1},f_{2} \in H^{1}_{\rho}(0,\infty;H)$ and $G_{1},G_{2}$ satisfying assumption \ref{ass:A} and $\Phi,\Psi$, satisfying assumption \ref{ass:B}, \cref{eq:FokkerPlank} has a unique local solution $(u,v)\in H^{1}_{0}(0,T;H)^{2}$, provided that the following two consistency conditions hold:
  \begin{equation*}
    f_{1}(0) = - G_{1}(0, \Phi) \quad\text{and}\quad
    f_{2}(0) = - G_{2}(0,\Psi)\text{.}
  \end{equation*}
\end{theorem}
\begin{theorem}
  For a right-hand side
  \begin{equation*}
    \begin{pmatrix}
      f + G\bigl(0, (I_{\rho}(u+\widehat{\Phi'}))_{(\argdot)}\bigr)\\
      0
    \end{pmatrix}\text{,}
  \end{equation*}
  with $f \in H^{1}_{\rho}(0,\infty;H)$ and $G$ satisfying assumption \ref{ass:A} and $\Phi$ satisfying assumption \ref{ass:B}, \cref{eq:KelvinVoigt} has a unique local solution $(u,T)\in H^{1}_{0}(0,T;H)^{2}$, provided that the consistency condition $f(0) = - G(0, \Phi)$ holds.
\end{theorem}

\subsection{Convolutions with bounded operators}
\label{subsec:ViscoElasticity}
We briefly describe an example with a nonlocal material law: For some $\mu \in \mathbb{R}$ let $L_{1,\mu}\bigl(0,\infty;\mathcal{L}(H)\bigr)$ be the space of weakly measurable functions $B\colon \mathbb{R}_{\geq 0} \to \mathcal{L}(H)$, i.e.,
\begin{equation*}
  \forall x,y \in H\colon t\mapsto \dualprod{B(t)x}{y} \,\,\text{is measurable,}
\end{equation*}
such that the map $t\mapsto \norm{B(t)}$ is measurable and the norm
\begin{equation*}
  \norm{B}_{L_{1,\mu}(0,\infty;H)}\coloneq \medint\int_{0}^{\infty} \e^{-\mu t}\norm{B(t)}\dx[t]
\end{equation*}
is finite. The convolution $B\ast$ of an operator $B \in L_{1,\mu}\bigl(0,\infty;\mathcal{L}(H)\bigr)$ is defined as the extension of the convolution defined on the space of simple functions $S(\mathbb{R};H)$ with values in $H$
\begin{equation*}
  B\ast \colon L_{2,\rho}(\mathbb{R};H)\supseteq S(\mathbb{R};H) \to L_{2,\rho}(0,\infty;H)\text{,} \qquad u \mapsto \medint\int_{\mathbb{R}}B(\argdot - s)u(s)\dx[s]
\end{equation*}
to the space $L_{2,\rho}(\mathbb{R};H)$. The integral here is understood in the weak sense. That this definition is well-defined is verified in \cite[lem.~3.1]{Trostorff2015-2}.\\
With the help of this convolution operator, one can pose the equations of visco-elasticity: Let $\Omega \subseteq \mathbb{R}^{3}$ be open and let $L_{2,\mathrm{sym}}(\Omega)$ be the subspace of symmetric matrices of $L_{2}(\Omega)^{3\times 3}$ (equipped with the Frobenius norm). The equations of visco-elasticity read
\begin{equation}
  \label{eq:ViscoElasticity}
  \biggl[\partial_{t}\!\!\begin{pmatrix}
    \rho & 0 \\ 0 & C^{-\nicefrac{1}{2}}(1 \!- \!C^{-\nicefrac{1}{2}} (B\ast) C^{-\nicefrac{1}{2}})^{-1}C^{-\nicefrac{1}{2}}
  \end{pmatrix}
  \!- \!\begin{pmatrix}
    0 & \operatorname{Div} \\ \operatorname{Grad}_{0} & 0
  \end{pmatrix}\biggr]
  \!\!\begin{pmatrix} v \\ T \end{pmatrix}
  \!= \!\begin{pmatrix} f \\ 0 \end{pmatrix}
\end{equation}
and are posed on the state-space $H= L_{2}(\Omega)^{3}\oplus L_{2,\mathrm{sym}}(\Omega)$. Here $\operatorname{Grad}_{0}$ denotes the symmetric gradient with Dirichlet boundary condition; $\operatorname{Div}$ is the (row-wise taken) weak divergence, cf.~\cref{sec:DiffOp} for definitions. The unknown $u$ is the displacement field and $T$ is the stress tensor; $\rho \in L_{\infty}(\Omega)$ is a density that is strictly positive and bounded below (by a constant $>0$) and $C \in \mathcal{L}\bigl(L_{2,\mathrm{sym}}(\Omega)\bigr)$ is selfadjoint and strictly positive-definite and $B \in L_{1,\mu}\bigl(0,\infty;\mathcal{L}(L_{2,\mathrm{sym}}(\Omega))\bigr)$ as above. Well-posedness of this system is proven in \cite[thm.~4.2]{Trostorff2015-2} under the assumption that $C$ and $B(t)$ commute (for all $t$) and additional assumptions on $B$.\\
The proof hinges on a generalized version of Picard’s theorem, cf.~\cite[thm.~3.7]{Trostorff2013}. In this context, it suffices to show the condition $\Re z^{-1}M(z)\geq c$ for all $z\in \mathcal{B}_{\C}(r,r)$. This can be characterized via the conditions
\begin{equation*}
  \Re z^{-1}\bigl(1+\sqrt{2\pi}\hat{C} (\iu z^{-1})\bigr) \geq c \qquad \text{and}\qquad
  \Re z^{-1}\bigl(1-\sqrt{2\pi}\hat{B} (\iu z^{-1})\bigr) \geq c\text{,}
\end{equation*}
where $\hat{B}$ and $\hat{C}$ denote the Fourier-transforms of $B$ and $C$ respectively. Conditions for these two inequalities to be satisfied are given in \cite[lem.~3.9 \& lem.~3.10]{Trostorff2015-2}.\\
It is possible to study materials with memory effects, expressed via state-dependent delay on the right-hand side. This was already done in \cite{Trostorff2015-2}, although by extending Picard’s theorem to the $H^{-1}_{\rho}$-case and proving existence of solutions there. With our tools and under stricter assumptions, we can pose the problem in $H^{1}_{0,\rho}$ and prove existence and uniqueness in a stronger sense by appealing to \cref{th:localExistOut} or \cref{th:localExistIns}. For instance, in the framework of \cref{th:localExistIns}, we can formulate:
\begin{theorem}
  For a right-hand side
  \begin{equation*}
    \begin{pmatrix}
      f + G\bigl(\argdot, (I_{\rho}(v+\widehat{\Phi'}))_{(\argdot)}\bigr)\\
      0
    \end{pmatrix}\text{,}
  \end{equation*}
  with $f \in H^{1}_{\rho}(0,\infty;H)$,  $G$ satisfying assumption \ref{ass:A} and $\Phi$ satisfying assumption \ref{ass:B}, \cref{eq:ViscoElasticity} has a unique local solution $(v,T)\in H^{1}_{0}(0,T;H)^{2}$, provided that the consistency condition $f(0) = - G(0, \Phi)$ holds.
\end{theorem}

\subsection{port-Hamiltonian systems}
Port-Hamiltonian systems arise in a wide variety of applications and are particularly useful in the context of control theory. For an introduction to port-Hamiltonian systems we refer to \cite{Schaft2014, Jacob2012}. For finite-dimenisonal port-Hamiltonian systems, the ODE theory for state-dependent delay systems developed in \cite{Aigner2024,Waurick2023} already applies. We will here take a glimpse at infinite-dimensional port-Hamiltonian systems. We use the formulation (7.8) from the standard reference \cite{Jacob2012}. Similar situations are also covered in \cite[sec.~5.1]{Waurick2012} and \cite{Waurick2023-2}.
\begin{definition}
  Let $P_{1}\in \mathbb{K}^{n\times n}$ be selfadjoint, let $P_{0}\in \mathbb{K}^{n\times n}$ be skew-selfadjoint and let $\mathcal{H} \in L_{\infty}(a,b;\mathbb{K}^{n\times n})$ satisfy $\mathcal{H}(\zeta)^{\ast} = \mathcal{H}(\zeta)$ and $mI \leq \mathcal{H}(\zeta) \leq MI$ for a.e. $\zeta \in [a,b]$ and $m,M > 0$. Let $X \coloneq L_{2}(a,b;\mathbb{K}^{n})$ be equipped with the inner product
  \begin{equation*}
    \dualprod{f}{g}_{X} = \tfrac{1}{2} \medint\int_{a}^{b}g(\zeta)^{\ast}\mathcal{H}(\zeta)f(\zeta) \dx[\zeta]\text{.}
  \end{equation*}
  Then the diﬀerential equation
  \begin{equation*}
    \tfrac{\partial x}{\partial t}(\zeta , t) = P_{1}\tfrac{\partial}{\partial \zeta}\bigl(\mathcal{H}(\zeta)x(\zeta , t)\bigr)
    + P_{0}\bigl(\mathcal{H}(\zeta)x(\zeta , t)\bigr)
  \end{equation*}
  is called a {\em linear, first order port-Hamiltonian system}.
\end{definition}
We want to examine the corresponding version of \cref{eq:EvolProblem}:
\begin{equation*}
  \bigl(\partial_{t}\mathcal{H}^{-1} - P_{0} - P_{1}\partial_{1}\bigr)v
  = F\bigl(\argdot, (v+I_{\rho}\widehat{\Phi'})_{(\argdot)}\bigr) + f\text{,}
\end{equation*}
gained via the usual ansatz $u=v+I_{\rho}\widehat{\Phi'}$ for a prescribed prehistory $\Phi$. Note that $\mathcal{H}$ is included into the inner product via $\dualprod{\mathcal{H}\argdot}{\argdot}_{X}$. We restrict ourselves
to the case where $P_{1}\partial_{1}$ is maximally dissipative. This is usually achieved by incorporating the boundary conditions into the operator and corresponds to the case where $P_{1}\partial_{1}$ is the infinitesimal generator of a $\mathcal{C}_{0}$-semigroup of contractions, appealing to the Lumer\textendash Philips theorem. In this case, we can provide a suitable well-posedness theory for our considered setting, e.g., for the equation
\begin{equation}
  \label{eq:pHS}
  \bigl(\partial_{t}\mathcal{H}^{-1} - P_{0} - P_{1}\partial_{1}\bigr)v
  = G\bigl(\argdot, (I_{\rho}(v+\widehat{\Phi'}))_{(\argdot)}\bigr) + f\text{,}
\end{equation}

\begin{theorem}
  \Cref{eq:pHS} is (locally) uniquely solvable under assumption \ref{ass:A} on $G$ and assumption \ref{ass:B} on $\Phi$, provided that $f \in H_{\rho}^{1}(0,\infty;H)$ for some $\rho>0$ and the consistency condition $f(0)=-G(0,\Phi)$ holds.
\end{theorem}

\begin{proof}
  To apply \cref{th:localExistIns}, we have to verify the assumptions of \cref{th:Picard}. Since $P_{1}\partial_{1}$ is maximally dissipative, this reduces to the condition
  \begin{equation*}
    \forall x\in X \forall z \in \C_{\Re \geq \nu}\colon \quad \Re \dualprod{x}{zM(z)x}\geq c\dualprod{x}{x}\text{,}
  \end{equation*}
  where $M(z)=\mathcal{H}^{-1}z + P_{0}$ (and $\C_{\Re \geq \nu}$ is some right half-plane). Since $\mathcal{H}^{-1}$ is pointwise selfadjoint, this is equivalent to
  \begin{equation*}
    \forall x \in X\colon\quad \rho_{0}\dualprod{\mathcal{H}^{-1}x}{x} = \rho_{0}\dualprod{\mathcal{H}^{-1}x}{x} + \Re \dualprod{P_{0}x}{x} \geq c \dualprod{x}{x}
  \end{equation*}
  for some $\rho_{0}\geq 0$ and $c>0$, because $P_{0}$ is skew-selfadjoint. From the assumptions on $\mathcal{H}$ we infer that for any $x\in \operatorname{dom}(\mathcal{H}^{\nicefrac{1}{2}})=X$:
  \begin{equation*}
    \dualprod{\mathcal{H}x}{\mathcal{H}x}
    = \dualprod{\mathcal{H}\mathcal{H}^{\nicefrac{1}{2}}x}{\mathcal{H}^{\nicefrac{1}{2}}x}
    \leq M \dualprod{\mathcal{H}^{\nicefrac{1}{2}}x}{\mathcal{H}^{\nicefrac{1}{2}}x}
    = M \dualprod{x}{\mathcal{H}x}\text{,}
  \end{equation*}
  where we used that $\mathcal{H}^{\nicefrac{1}{2}}$ is selfadjoint. The desired inequality follows.
\end{proof}

\section{Summary and Comparison}
In this article, we considered ``generalized'' initial value problems in the form of \cref{eq:EvolProblem}, that principally arise in a distributional formulation (motivated in \cref{subsec:DistribSpaces_IVPs}). Evolutionary equations in distributional spaces stemming from initial value problems have been studied before by S.~Trostorff in \cite{Trostorff2018} (cf. also the older publications \cite{Kalauch2014} and \cite{Waurick2012-2} for ODEs). As mentiond in \cref{subsec:DistribSpaces_IVPs}, formulating suitable initial conditions tailored to a problem is nontrivial in general. We bypassed this aspect by rather examining the general formulation (\ref{eq:DistribProblem}) that arises after transitioning from an initial value problem to a problem with non-homogeneity $f\in H^{-1}_{\rho}$ (cf., e.g., \cite{Trostorff2015-2,Trostorff2018}). Rather than exploring this general setup, we further focussed on the issues occuring due to the presence of state-dependent delay; specifically, we demanded $f \in L_{2,\rho}/H^{1}_{\rho}$ (cf.~formulation \eqref{eq:EvolProblem}).\\
In any case, the main idea for the well-posedness results \cref{th:localExistOut} and \cref{th:localExistIns} is to extend the approach from \cite{Waurick2023} developed for the ODE case (we also refer to \cite{Aigner2024} for a more comprehensive overview) to the PDE case. There, a contraction mapping argument paired with a projection argument was used to prove a generalized Picard--Lindel\"of theorem (cf.~\cite[thm.~4.1]{Waurick2023}) for ODEs with state-dependent delay. The idea for that result in turn stems from Morgenstern’s proof of the Picard--Lindel\"of theorem (cf.~\cite{Morgenstern1952}) utilizing exponentially weighted spaces in order to force a Lipschitz-continuous map to be a contraction. Observing that the delay operator $\Theta$ scales with the inverse of the weight parameter $\rho$ (cf.~\cref{th:ThetaIsLipschitz}) gave rise to $H^{1}$-solution theory for ODEs with state-dependent delay. Since exponentially weighted spaces naturally arise in the theory of evolutionary equations, it was an equally natural question to ask if the ODE approach could be generalized to PDEs. In the PDE case, there is the complication of the solution operator/Picard operator $S_{\rho}$ though, which is a bounded linear operator on $L_{2,\rho}(\mathbb{R};H)$ and certainly not as convenient for the solution theory as the antiderivative $I_{\rho}$ alone. In particular, there is no automatic increase in regularity due to the shape of the right-hand side of the fixed point problem any more. We had to remedy that fact by demanding that the right-hand side $f+F$ in \cref{eq:EvolProblem} is in some sense regularity-increasing. Naturally, an integrability condition also came into play, because the Picard operator is only defined on the entire real line. The increased demand on regularity introduces consistency conditions, more precisely the right-hand side $f +F$ needs to be an element of $H^{1}_{0,\rho}(0,\infty;H)$.\\
Consistency conditions are prevalent in solution theories for state-dependent differential equations: In the case of classical solution theory (i.e., in the sense of $\mathcal{C}^{1}$-solutions) for the ODE case, there is the concept of the {\em solution manifold}, cf.~\cite{Walther2003}, characterizing compatible initial prehistories via a consistency equation. In the case of evolutionary equations, the {\em space of admissible prehistories} $\mathrm{His}_{\rho}(M,A)$ (cf.~\cref{def:HistorySpace}) essentially takes the place of the solution manifold.\\
The increased regularity of the right-hand side $f+F$ is used in a projection argument. In order to remove the projection locally and turn a solution of the modified problem (\ref{eq:FPPproj}) into a solution of the original problem (\ref{eq:FPPinL2}), we made a continuity argument, cf.~the proofs of \cref{th:localExistOut} and \cref{th:localExistIns}. The complication lies in the definition of the spaces
\vspace{-1ex}
\begin{equation*}
  V_{\alpha} = \bigl\{u \in H^{1}(-h,0;H)\colon \norm{u'}_{\infty}\leq \alpha \bigr\}\text{.}
\end{equation*}
The $\norm{\argdot}_{\infty}$-norm in the definition of $V_{\alpha}$ destroys the Hilbert space structure that is otherwise present everywhere in the theory of evolutionary equations and to the best of our knowledge, there are no bounded-input/bounded-output results for general evolutionary equations. Why define $V_{\alpha}$ like that then? The simple answer is practical necessity: the ``standard'' example for state-dependent delay is of the form $-\tau(u(t))$ for a Lipschitz-continuous $\tau$. To force Lipschitz-continuity of $u\bigl(t - \tau(u(t))\bigr)$ one needs $u'$ to be bounded. Therefore, the $\norm{\argdot}_{\infty}$-norm in $V_{\alpha}$ is necessary and it is precisely by virtue of the projection argument that we can accomodate state-dependent delay via a weakened notion of Lipschitz-continuity. If one replaces the $\norm{\argdot}_{\infty}$-norm with the $\norm{\argdot}_{2}$-norm, the entire theory goes through much smoother, leads to well-posedness results in $L_{2,\rho}$ without the need of consistency conditions or regularity-increasing right-hand sides, but is simply not applicable in relevant cases. Precisely because state-dependent delay differential equations in general do not exhibit Lipschitz-continuous right-hand sides, the weakening almost uniform Lipschitz-continuity was introduced in \cite{Waurick2023} and reused here (cf.~\cref{def:AlmLC}).\\
To the best of our knowledge, there is little general solution theory for PDEs with state-dependent delay, general in the sense that the theory covers diverse physical phenomena modelled by parabolic or hyperbolic differential equations. Classical theory seems to have been done exclusively for the ODE case (cf.~the classical reference \cite{Hale1993}), the common approach for PDEs seems to be semigroup theory (cf.~the textbooks \cite{Batkai2005, Wu1996}, but there is a multitude of authors and articles in our references). Semigroup theory holds some advantages and disadvantages over evolutionary equations. The advantage is that existence results usually have a simple way of inferring strong solutions under slightly stronger assumptions, e.g., that the prehistory takes values in the domain of the semigroup generator or a suitable interpolation space thereof. The theory of evolutionary equations on the other hand is a theory purely designed around a notion of weak solutions. The advantage of evolutionary equations is that the concept is applicable to a larger set of equations, in particular time-nonlocal implicit partial differential algebraic equations. Here lies the biggest advantage of our approach over semigroup results. There are some notable technical differences as well. For instance, in evolutionary equations there is no need to consider extended state spaces or demand (possibly difficult to verify) Lipschitz-conditions of the nonlinearity in case the semigroup generator is not a differential operator in divergence form.\\
This article was a foray into developing $H^{1}$-solution theory for DDEs in the framework of evolutionary equations. Many questions require further investigation though. For instance, do global solutions actually require the stricter assumption of almost uniform Lipschitz-regularity in $L_{2}$, or can this assumption be lowered, at least in certain cases?

\appendix
\section{Differential operators}
\label{sec:DiffOp}
For the sake of completeness, we include the definitions of various differential operators used in \cref{sec:Examples}. In the following, let $\Omega \subseteq \mathbb{R}^{d}$ always be open.

\subsection{Weak gradient, divergence and rotation}
\label{subsec:WeakDiffOp}
We define
\begin{alignat*}{4}
  \operatorname{grad}_{\mathrm{s}}&\colon \mathcal{C}^{\infty}_{\mathrm{c}}(\Omega)^{\phantom{3}} &\to \mathcal{C}^{\infty}_{\mathrm{c}}(\Omega)^{d}\text{,} &&f &\mapsto \bigl(\partial_{i}f\bigr)_{1\leq i \leq d}\text{,}\\
  \operatorname{div}_{\mathrm{s}}&\colon \mathcal{C}^{\infty}_{\mathrm{c}}(\Omega)^{d} &\to \mathcal{C}^{\infty}_{\mathrm{c}}(\Omega)^{\phantom{d}}\text{,} &&\quad \bigl(f_{i}\bigr)_{1\leq i \leq d}&\mapsto \sum_{i=1}^{d}\partial_{i}f_{i}\text{,}\\
  \operatorname{curl}_{\mathrm{s}}&\colon \mathcal{C}^{\infty}_{\mathrm{c}}(\Omega)^{3}&\to \mathcal{C}^{\infty}_{\mathrm{c}}(\Omega)^{3}\text{,}&&\begin{psmallmatrix} f_{1} \\ f_{2} \\ f_{3} \end{psmallmatrix} &\mapsto \begin{psmallmatrix}\partial_{2}f_{3} - \partial_{3}f_{2} \\ \partial_{3}f_{1} - \partial_{1}f_{3} \\ \partial_{1}f_{2} - \partial_{2}f_{1} \end{psmallmatrix}\text{,}
\end{alignat*}
denoted as the {\em smooth} gradient, divergence and rotation respectively, where all partial derivatives are understood in the classical sense. Based on this definition, we can make use of the fact that $\mathcal{C}^{\infty}_{\mathrm{c}}(\Omega) \subseteq L_{2}(\Omega)$ is dense and define (as in \cite[sec.~6.1]{Waurick2022})
\begin{alignat*}{4}
  \operatorname{grad}&\colon L_{2}(\Omega)^{\phantom{3}} \supseteq \operatorname{dom}(\operatorname{grad}) &\to L_{2}(\Omega)^{d}\text{,} &&\quad \operatorname{grad} &\coloneq - \operatorname{div}^{\ast}_{\mathrm{s}}\text{,}\\
  \operatorname{div}&\colon L_{2}(\Omega)^{d} \supseteq \operatorname{dom}(\operatorname{div}) &\to L_{2}(\Omega)^{\phantom{d}}\text{,} &&\operatorname{div} &\coloneq - \operatorname{grad}^{\ast}_{\mathrm{s}}\text{,}\\
  \operatorname{curl}&\colon L_{2}(\Omega)^{3} \supseteq \operatorname{dom}(\operatorname{curl}) &\to L_{2}(\Omega)^{3}\text{,} &&\operatorname{curl} &\coloneq \operatorname{curl}^{\ast}_{\mathrm{s}}\text{,}
\end{alignat*}
where the adjoint is understood as the Hilbert space adjoint with the domains being the domains of the adjoint. We call these operators the {\em weak} gradient, divergence and rotation respectively. It is an easy task to verify that this definition coincides with the standard definition of weak (partial) derivatives (e.g., \cite[sec.~5.2.1]{Evans2010}), in fact the following holds:
\begin{proposition}[{\cite[thm.~6.1.2]{Waurick2022}}]\phantom{.}
  \begin{itemize}[leftmargin=4ex]
    \item $f \in \operatorname{dom}(\operatorname{grad}) \Longleftrightarrow \forall 1\leq i \leq d \exists g_{i}\in L_{2}(\Omega)\forall \varphi \in \mathcal{C}^{\infty}_{\mathrm{c}}(\Omega)\colon$
          \begin{equation*}
            \medint\int_{\Omega} f \partial_{i}\varphi = -\medint\int_{\Omega}g_{i}\varphi\text{.}
          \end{equation*}
    \item $(f_{i})_{1\leq i \leq d} \in \operatorname{dom}(\operatorname{div}) \Longleftrightarrow \exists g\in L_{2}(\Omega)\forall \varphi \in \mathcal{C}^{\infty}_{\mathrm{c}}(\Omega)\colon$
          \begin{equation*}
            \medint\int_{\Omega}g \varphi = - \medint\int_{\Omega}\textstyle\sum_{i=1}^{d}f_{i}\partial_{i}\varphi\text{.}
          \end{equation*}
    \item $\begin{psmallmatrix} f_{1} \\ f_{2} \\ f_{3} \end{psmallmatrix} \in \operatorname{dom}(\operatorname{curl}) \Longleftrightarrow \exists \begin{psmallmatrix} g_{1} \\ g_{2} \\ g_{3} \end{psmallmatrix} \in L_{2}(\Omega)^{3} \forall \begin{psmallmatrix} \varphi_{1} \\ \varphi_{2} \\ \varphi_{3} \end{psmallmatrix} \in \mathcal{C}^{\infty}_{\mathrm{c}}(\Omega)^{3}\colon$
          \begin{equation*}
            \medint\int_{\Omega} \begin{psmallmatrix} g_{1} \\ g_{2} \\ g_{3} \end{psmallmatrix} \cdot \begin{psmallmatrix} \varphi_{1} \\ \varphi_{2} \\ \varphi_{3} \end{psmallmatrix} = \medint\int_{\Omega} \begin{psmallmatrix} f_{1} \\ f_{2} \\ f_{3} \end{psmallmatrix} \cdot \operatorname{curl}_{\mathrm{s}} \begin{psmallmatrix} \varphi_{1} \\ \varphi_{2} \\ \varphi_{3} \end{psmallmatrix}\text{.}
          \end{equation*}
  \end{itemize}
\end{proposition}
By adjoining again we obtain the weak gradient, divergence and rotation with {\em Dirichlet boundary condition}:
\begin{alignat*}{4}
  \operatorname{grad}_{0}&\colon L_{2}(\Omega)^{\phantom{3}} \supseteq \operatorname{dom}(\operatorname{grad}_{0}) &\to L_{2}(\Omega)^{d}\text{,} &&\quad \operatorname{grad}_{0} &\coloneq - \operatorname{div}^{\ast}\text{,}\\
  \operatorname{div}_{0}&\colon L_{2}(\Omega)^{d} \supseteq \operatorname{dom}(\operatorname{div}_{0}) &\to L_{2}(\Omega)^{\phantom{d}}\text{,} &&\operatorname{div}_{0} &\coloneq - \operatorname{grad}^{\ast}\text{,}\\
  \operatorname{curl}_{0}&\colon L_{2}(\Omega)^{3} \supseteq \operatorname{dom}(\operatorname{curl}_{0}) &\to L_{2}(\Omega)^{3}\text{,} &&\operatorname{curl}_{0} &\coloneq \operatorname{curl}^{\ast}\text{.}
\end{alignat*}

\begin{remark}
  One can prove that the above definition is equivalent to the statement that an element of the domain (of the respective weak differential operator with Dirichlet boundary condition) is approximable in graph-norm of the associated weak differential operator by elements of $\mathcal{C}^{\infty}_{\mathrm{c}}(\Omega)/ \mathcal{C}^{\infty}_{\mathrm{c}}(\Omega)^{d} / \mathcal{C}^{\infty}_{\mathrm{c}}(\Omega)^{3}$.
\end{remark}

\subsection{Symmetrized differential operators}
\label{subsec:SymDiffOp}
Almost completely analogous to the vector-valued case, one can define differential operators on matrix-valued functions. For that purpose let
\begin{equation*}
  \mathbb{R}^{d\times d}_{\mathrm{sym}} \coloneq \bigl\{M\in \mathbb{R}^{d\times d}\colon M=M^{\mathrm{T}}\bigr\}
\end{equation*}
be the subspace of symmetric matrices in the space of $d\times d$ matrices. Let further
\begin{alignat*}{4}
  \operatorname{Grad}_{\mathrm{s}}&\colon \mathcal{C}^{\infty}_{\mathrm{c}}(\Omega)^{d} &\to \mathcal{C}^{\infty}_{\mathrm{c}}\bigl(\Omega;\mathbb{R}^{d\times d}_{\mathrm{sym}}\bigr)\text{,} &&(f_{i})_{1\leq i \leq d} &\mapsto \tfrac{1}{2}(\partial_{i}f_{j} + \partial_{j}f_{i})_{1\leq i,j \leq d}\text{,}\\
  \operatorname{Div}_{\mathrm{s}}&\colon \mathcal{C}^{\infty}_{\mathrm{c}}\bigl(\Omega;\mathbb{R}^{d\times d}_{\mathrm{sym}}\bigr) &\to \mathcal{C}^{\infty}_{\mathrm{c}}(\Omega)^{d}\phantom{\mathbb{R}^{d\times d}}\text{,} &&\quad (f_{i,j})_{1\leq i,j \leq d}&\mapsto \Bigl( \textstyle\sum_{i=1}^{d}\partial_{i}f_{j,i}\Bigr)_{1\leq j\leq d}\text{,}
\end{alignat*}
denote the {\em symmetrized smooth} gradient and divergence respectively. Analogous to the vector-valued case, one defines the {\em weak symmetrized} gradient/divergence and weak symmetrized gradient/divergence with {\em Dirichlet boundary} by adjoining in $L_{2}$:
\begin{alignat*}{4}
  \operatorname{Grad}&\colon L_{2}(\Omega)^{d}\phantom{\mathbb{R}^{d\times .}} \supseteq \operatorname{dom}(\operatorname{Grad}) &\to L_{2}\bigr(\Omega;\mathbb{R}^{d\times d}_{\mathrm{sym}}\bigl)\text{,} &&\quad \operatorname{Grad} &\coloneq - \operatorname{Div}^{\ast}_{\mathrm{s}}\text{,}\\
  \operatorname{Div}&\colon L_{2}\bigr(\Omega;\mathbb{R}^{d\times d}_{\mathrm{sym}}\bigl) \supseteq \operatorname{dom}(\operatorname{Div}) &\to L_{2}(\Omega)^{d}\phantom{\mathbb{R}^{d\times ..}}\text{,} &&\operatorname{Div} &\coloneq - \operatorname{Grad}^{\ast}_{\mathrm{s}}\text{,}\\
  \operatorname{Grad}_{0}&\colon L_{2}(\Omega)^{d}\phantom{\mathbb{R}^{d\times .}} \supseteq \operatorname{dom}(\operatorname{Grad}_{0}) &\to L_{2}\bigr(\Omega;\mathbb{R}^{d\times d}_{\mathrm{sym}}\bigl)\text{,} &&\quad \operatorname{Grad}_{0} &\coloneq - \operatorname{Div}^{\ast}\text{,}\\
  \operatorname{Div}_{0} &\colon L_{2}\bigr(\Omega;\mathbb{R}^{d\times d}_{\mathrm{sym}}\bigl) \supseteq \operatorname{dom}(\operatorname{Div}_{0}) &\to L_{2}(\Omega)^{d}\phantom{\mathbb{R}^{d\times ..}}\text{,} &&\operatorname{Div}_{0} &\coloneq - \operatorname{Grad}^{\ast}\text{,}
\end{alignat*}

\subsection{Dirac distributions}
\label{subsec:Dirac}
Appealing to \cref{th:Sobolev}, on $H^{1}_{\rho}(\mathbb{R};H)$ we define the so-called {\em Dirac distribution} $\delta_{t}$ for any $t \in \mathbb{R}$ as
\begin{equation*}
  \delta_{t} \colon H^{1}_{\rho}(\mathbb{R};H) \to \mathbb{R}\text{,}\qquad \varphi \mapsto \varphi(t)\text{.}
\end{equation*}
It is easily verified that this operator is bounded, indeed
\begin{equation*}
  \abs{\delta_{t}(\varphi)} = \abs{\varphi(t)} = \e^{\rho t}\e^{-\rho t} \abs{\varphi(t)}
  \leq \e^{\rho t}\norm{\varphi}_{\mathcal{C}_{0,\rho}(\mathbb{R};H)} \leq C\e^{\rho t}\norm{\varphi}_{H^{1}_{\rho}(\mathbb{R};H)}\text{,}
\end{equation*}
where we appeal to the embedding $H^{1}_{\rho}\hookrightarrow \mathcal{C}_{0,\rho}(\mathbb{R};H)$ (i.e., \cref{th:Sobolev}) again. Thus, $\delta_{t}\in H^{-1}_{\rho}(\mathbb{R};H)$.

\section{Initial value problems}
\label{sec:IVPs}

In this appendix, we provide some additional details regarding \cref{subsec:DistribSpaces_IVPs}. For that purpose, throughout this section, let $M$ be a material law and $A$ an $\mathrm{m}$-accretive operator in a Hilbert space $H$. Additionally, we assume that they satisfy the assumptions of Picard’s \cref{th:Picard}.
\subsection{Distributional problems and support on $\mathbb{R}_{\geq 0}$}
The principal problem in stating evolutionary equations on the half-line in general is to give meaning to the informal expression
\begin{equation*}
  \bigl(\partial_{\rho}M(\partial_{\rho}) + A\bigr)u = F\bigl(\argdot,u_{(\argdot)}\bigr)
\end{equation*}
on $\mathbb{R}_{\geq 0}$. We elaborate how to get to the problem \eqref{eq:DistribProblem}, i.e.,
\begin{equation*}
  \bigl(\partial_{\rho}M(\partial_{\rho}) + A\bigr)v = f + F\bigl(\argdot,(v + Z)_{(\argdot)}\bigr)
\end{equation*}
in $H^{-1}_{\rho}(\mathbb{R};H)$ with $f$ and
\begin{equation*}
  t\mapsto F\bigl(t,(v+Z)_{(t)}\bigr)
\end{equation*}
having support on the right half-line $\mathbb{R}_{\geq 0}$ only (this is critically important for the well-posedness theory in \cref{sec:SolThy}). For the latter map, this is an assumption on $F$. We only need to go into detail in regard to $f$. The key reference for this deliberation is \cite[sec.~3]{Trostorff2018}. There, IVPs of the form
\begin{alignat}{3}
  \bigl(\partial_{\rho}M(\partial_{\rho}) + A\bigr) u &= 0 &\qquad \text{on}\ \mathbb{R}_{>0}\text{,}\label{eq:IntuitiveIVPDiffEq}\\
  u &= g & \text{on}\ \mathbb{R}_{\leq 0}\text{,} \label{eq:IntuitiveIVPInitialCondition}
\end{alignat}
are studied for $g \in H^{1}_{\rho}(-\infty,0;H)$. The equation is given meaning by applying the projection $P_{0}$ (cf.~\cref{def:cutoff}) to the problem for $u=v+g$ on the entire real line. To give the accurate result, we briefly recall the necessary definitions and preliminary results from \cite{Trostorff2018}.
\begin{definition}
  Let $M\colon \C \supseteq \operatorname{dom}(M) \to \mathcal{L}(H)$ be a linear material law. $M$ is called {\em regularizing} if there exists $\rho_{0}\in \mathbb{R}_{>0}$ such that $\C_{\Re \geq \rho_{0}}\subseteq \operatorname{dom}(M)$, $M\vert_{\C_{\Re \geq \rho_{0}}}$ is bounded and for all $\rho\geq \rho_{0}$ and $x\in H$,
  \begin{equation*}
    \bigl(M(\partial_{\rho})\chi_{\mathbb{R}_{\geq 0}}x\bigr)(0+)\quad\text{exists.}
  \end{equation*}
\end{definition}
This definition is required for the following useful fact:
\begin{lemma}[{\cite[lem.~3.2.4]{Trostorff2018}}]
  Let $M$ be a regularizing linear material law and $\rho > s_{\mathrm{b}}(M)$. Then for $g\in \chi_{\mathbb{R}_{\leq 0}}H^{1}_{\rho}(\mathbb{R};H)$ holds: $\partial_{\rho}M(\partial_{\rho})g \in \operatorname{dom}(P_{0})$.
\end{lemma}
Appealing to this lemma, for regularizing material laws we can define
\begin{equation*}
  K_{\rho}\colon \chi_{\mathbb{R}_{\leq 0}}H^{1}_{\rho}(\mathbb{R},H) \to H^{-1}_{\rho}(\mathbb{R};H)\text{,}\quad
  g \mapsto P_{0}\partial_{\rho}M(\partial_{\rho})g
\end{equation*}
and subsequently:
\begin{definition}
  \label{def:HistorySpace}
  Let $M$ be a regularizing material law and $\rho > s_{\mathrm{b}}(M)$. Let
  \begin{align*}
    \mathrm{His}_{\rho}(M,A) &\!\coloneq \!\bigl\{g \in \chi_{\mathbb{R}_{\leq 0}}H^{1}_{\rho}(\mathbb{R};H) \colon \exists x\!\in \!H\!\colon \!
    S_{\rho}(\delta_{0}x \!- \!K_{\rho}g) \!- \!\chi_{\mathbb{R}_{\geq 0}}g(0-) \!\in \! H^{1}_{\rho}(\mathbb{R};H)\bigr\}\\
    \intertext{be the space of {\em admissible histories for $M$ and $A$}. Furthermore let}
    \mathrm{IV}_{\rho}(M,A) &\!\coloneq \!\bigl\{g(0-)\colon g \in \mathrm{His}_{\rho}(M,A)\bigr\}
  \end{align*}
  be the space of {\em admissible initial values for $M$ and $A$}.
\end{definition}
\begin{remark}
  It can be shown, that the element $x$ in the definition of the space of admissible histories is uniquely determined and is given by
  \begin{equation*}
    x = \bigl(M(\partial_{\rho})g\bigr)(0-) - \bigl(M(\partial_{\rho})g\bigr)(0+),
  \end{equation*}
  cf.~\cite[lem.~3.2.5]{Trostorff2018}. Hence, it makes sense to define
  \begin{equation*}
    \Gamma^{\rho}_{(M,A)} \colon \mathrm{His}_{\rho}(M,A) \to H\text{,} \quad
    g\mapsto \bigl(M(\partial_{\rho}g)\bigr)(0-) - \bigl(M(\partial_{\rho}g)\bigr)(0+)\text{.}
  \end{equation*}
\end{remark}
With these concepts, we can give the precise statement concerning the reformulation of the system (\ref{eq:IntuitiveIVPDiffEq}) + (\ref{eq:IntuitiveIVPInitialCondition}) as an evolutionary equation in $H^{-1}_{\rho}(\mathbb{R};H)$:
\begin{theorem}[{\cite[prop.~3.2.6]{Trostorff2018}}]
  \label{th:SaschasBigTheorem}
  For a regularizing material law $M$ and $\rho > s_{\mathrm{b}}(M)$ let
  \begin{equation*}
    v \coloneq S_{\rho}\bigl(\Gamma_{(M,A)}^{\rho} g - K_{\rho}g\bigr)\text{.}
  \end{equation*}
  Then the function $u \coloneq v + g$ defines an element of $H^{1}_{\rho}(\mathbb{R};H)$, satisfies the initial condition (\ref{eq:IntuitiveIVPInitialCondition}) and \cref{eq:IntuitiveIVPDiffEq} in the sense that
  \begin{equation*}
    \supp \bigl((\partial_{\rho}M(\partial_{\rho})+A)u\bigr)\subseteq \mathbb{R}_{\leq 0}\text{.}
  \end{equation*}
\end{theorem}
In the context of this article, we apply the projection operator $P_{0}$ to the equation
\begin{equation*}
  \bigl(\overline{\partial_{\rho}M(\partial_{\rho}) + A}\bigr)(v + Z) = F\bigl(\argdot, (v+Z)_{(\argdot)}\bigr)\text{.}
\end{equation*}
Since $v+Z$ is in the domain of the closure of $\partial_{\rho}M(\partial_{\rho}) + A$, appealing to \cref{th:EquivDistrProb}, solving the equation above is equivalent to the inclusion
\begin{equation*}
  \bigl(\partial_{\rho}M(\partial_{\rho}) + A\bigr)(v + Z) \subseteq F\bigl(\argdot, (v+Z)_{(\argdot)}\bigr)\text{,}
\end{equation*}
where we interpret the left-hand side as an element of $H^{-1}_{\rho}(\mathbb{R};H)\cap L_{2,\rho}\bigl(\mathbb{R};H^{-1}(A)\bigr)$. Appealing to \cref{rmk:PrehistoryVanishes}, we assume in the following that the initial prehistory $\Phi$ satisfies $\Phi(-h)=0$. Now we can elaborate how $f$ is obtained:\\
$\bigl(\partial_{\rho}M(\partial_{\rho}) + A\bigr)Z$ is a well-defined $L_{2,\rho}$-function, provided that $Z$ takes values in $\operatorname{dom}(A)$; the projection $P_{0}$ applied to it produces a function on $\mathbb{R}_{\geq 0}$, cf.~\cref{th:P_t_is_cutoff}. We observe that under the conditions of Trostorff's \cref{th:SaschasBigTheorem}, $v$ is of the form
\begin{equation*}
  \bigl[\bigl(M(\partial_{\rho})\tilde{Z}\bigr)(0-) - \bigl(M(\partial_{\rho})\tilde{Z}\bigr)(0+)\bigr]\delta_{0}
  - P_{0}\partial_{\rho}M(\partial_{\rho})\tilde{Z}\text{,}
\end{equation*}
where $\tilde{Z}$ is the extension of $Z$ by $0$ on $(-\infty,-h)$. We observe that the first expression is purely distributional and has support only at $0$. The second term is
\begin{equation*}
  \partial_{\rho}\chi_{\mathbb{R}_{\geq 0}}\partial_{\rho}^{-1}\partial_{\rho}M(\partial_{\rho})\tilde{Z}
  - \bigl[\partial_{\rho}^{-1}\partial_{\rho}M(\partial_{\rho})\bigr](0+)\delta_{0}\text{,}
\end{equation*}
the second summand of which only has support at $0$ again, whereas the first part reduces to
\begin{equation*}
  \partial_{\rho}\chi_{\mathbb{R}_{\geq 0}} M(\partial_{\rho})\tilde{Z}\text{,}
\end{equation*}
which is clearly supported on $\mathbb{R}_{\geq 0}$ only, appealing to causality of material laws. Hence, $\operatorname{spt}(f)\subseteq \mathbb{R}_{\geq 0}$.

\subsection{Simple material laws}
We conclude this section by providing a perspective on initial value problems for {\em simple} material laws, i.e., material laws of the form $M(z)= M_{0}+M_{1}z^{-1}$. In this case, the history space only depends on the values $g(0-)$ of the prehistory $g$, and not the entire prehistory. This simpler dependency could already be observed in several examples we provided in \cref{sec:Examples}, but can be proven in Trostorff’s general framework:
\begin{proposition}[{\cite[prop.~3.2.10]{Trostorff2018}}]
  \label{th:IV_for_simple_MatLaws}
  Let $M_{0},M_{1}\in \mathcal{L}(H)$ with $M_{0}$ selfadjoint satisfying $M_{0}\vert_{\operatorname{ran}(M_{0})}\geq c_{0}>0$ and $M_{1}$ satisfying $\Re M_{1}\vert_{\operatorname{ker}(M_{0})}\geq c_{1}>0$. Then for $M(z)\coloneq M_{0} + z^{-1}M_{1}$ and $\rho> s_{\mathrm{b}}(M)$ holds
  \begin{equation*}
    \mathrm{His}_{\rho}(M,A) = \bigl\{g \in \chi_{\mathbb{R}\leq 0} H^{1}_{\rho}(\mathbb{R};H)\colon g(0-) \in \mathrm{IV}_{\rho}(M,A)\bigr\}\text{.}
  \end{equation*}
  Furthermore, for $g\in \mathrm{His}_{\rho}(M,A)$ with $g(0-)= 0$ holds
  \begin{equation*}
    S_{\rho}\Bigl(\delta_{0}\bigl[\bigl(M(\partial_{\rho})g\bigr)(0-) - \bigl(M(\partial_{\rho})g\bigr)(0+)\bigr] - K_{\rho}g\Bigr)=0\text{.}
  \end{equation*}
\end{proposition}
It is an interesting observation, that in the context of formulating initial value spaces for evolutionary equations, one seems to obtain consistency conditions ``from the left'' for the initial prehistory. The consistency conditions required for solution theory of DDEs are akin to conditions on the history segments, which enforce a consistency condition ``from the right''. This phenomenon might be subject to future investigation.

\bibliographystyle{abbrvurl}
\bibliography{references}

\end{document}